\documentclass[review]{siamart1116}

\usepackage{lipsum}
\usepackage{amsfonts}
\usepackage{graphicx}
\usepackage{epstopdf}
\usepackage{algorithmic}
\ifpdf
  \DeclareGraphicsExtensions{.eps,.pdf,.png,.jpg}
\else
  \DeclareGraphicsExtensions{.eps}
\fi
\usepackage{amssymb}
\usepackage{isomath}
\usepackage{mathtools}
\usepackage{tikz}
\usepackage{pgfplots}
\usetikzlibrary{arrows.meta}
  \pgfplotsset{compat=newest}
  \usetikzlibrary{plotmarks}
  \usepackage{grffile}

\pgfplotsset{every axis/.append style={
        scaled ticks = false, 
        tick label style={/pgf/number format/fixed}
    }
}

\usetikzlibrary{external}
\newlength\figureheight 
\newlength\figurewidth 
\usetikzlibrary{plotmarks}
\usetikzlibrary{patterns}
\usepackage{tikz-3dplot}
\usepackage{multirow}
\usepackage{subfigure}
\usepackage{esint} 
\usepackage{multirow}
\usepackage{booktabs}
\usepackage{ulem}
\usepackage{cancel}
\usepackage{stmaryrd} 
\usepackage[autostyle]{csquotes} 
\usepackage{stackengine} 
\usepackage{float} 
\usepackage{placeins}

\numberwithin{theorem}{section}
\newtheorem{remark}{Remark}

\newcommand{\TheTitle}{Treatment of complex interfaces for Maxwell's equations with continuous coefficients using the correction function method} 
\newcommand{\TheAuthors}{Y.-M. Law and A. N. Marques and J.C. Nave}

\headers{Treatment of complex interfaces for Maxwell's equations}{\TheAuthors}

\title{{\TheTitle}}

\author{
  Yann-Meing Law\thanks{Department of Mathematics and Statistics, McGill University, Montr\'{e}al, QC H3A 0B9, Canada. (\email{yann-meing.law-kamcio@mcgill.ca}, \email{jean-christophe.nave@mcgill.ca})}
  \and 
  Alexandre Noll Marques\thanks{Department of Aeronautics and Astronautics, Massachusetts Institute of Technology, Cambridge, MA 02139-4307. (\email{noll@mit.edu})}
  \and
  Jean-Christophe Nave\footnotemark[1]
}

\usepackage{amsopn}

\ifpdf
\hypersetup{
  pdftitle={\TheTitle},
  pdfauthor={\TheAuthors}
}
\fi

\begin{document}

\maketitle

\begin{abstract}
We propose a high-order FDTD scheme based on the correction function method (CFM) 
	to treat interfaces with complex geometry without increasing the complexity of the numerical approach for 
	constant coefficients.
Correction functions are modeled by a system of PDEs based on Maxwell's equations with interface conditions. 
To be able to compute approximations of correction functions, 
	a functional that is a square measure of the error associated with the correction functions' system of PDEs 
	is minimized in a divergence-free discrete functional space. 
Afterward, 
	approximations of correction functions are used to correct a FDTD scheme in the vicinity of an interface where 
	it is needed.
We perform a perturbation analysis on the correction functions' system of PDEs. 
The discrete divergence constraint and the consistency of resulting schemes are studied.
Numerical experiments are performed for problems with different geometries of the interface.
A second-order convergence is obtained for a second-order FDTD scheme corrected using the CFM.
High-order convergence is obtained with a corrected fourth-order FDTD scheme. 
The discontinuities within solutions are accurately captured without spurious oscillations. 
\end{abstract}



\section{Introduction}
Maxwell interface problems arise when dielectric materials are considered,
	or when surface charges and currents are present at the interface. 
In computational electromagnetics, 
	the treatment of interface conditions between materials is challenging for several reasons,
	such as the treatment of complex geometries of the interface, 
	the level of complexity of a numerical method for arbitrarily complex interfaces and the consideration of 
	discontinuous coefficients to name a few \cite{Hesthaven2003}.
	
To handle interface conditions, 
	various numerical strategies use the Immersed Interface Method (IIM) \cite{Leveque1994} or 
	the Matched Interface and Boundary (MIB) method \cite{Zhao2004}
	for dielectric interface \cite{Deng2008}, 
	perfectly electric conducting (PEC) boundaries \cite{Zhao2010} and Drude materials \cite{Nguyen2015}. 
However,
	high-order schemes are difficult to achieve using these approaches for complex interfaces.
An alternative approach is to use the correction function method (CFM) \cite{Marques2011}, 
	which was inspired by the Ghost Fluid Method (GFM) \cite{Fedkiw1999}. 
This method was originally developed to handle Poisson's equation with interface jump conditions 
	for arbitrarily complex interfaces.
In contrast to the GFM for which high accuracy is hard to obtain,
	the CFM achieves high-order accuracy by means of a minimization problem.
The CFM's minimization problem is derived as follows.
Based on the original problem, 
	a system of partial differential equations (PDEs)  
	for which the solution corresponds to a function, 
	namely the correction function,
	is derived.
A functional that is a square measure of the error associated with the correction function's system of PDEs is 
	minimized on patches around the interface in an appropriate functional space.
This allows us to compute approximations of the correction function to correct the finite difference (FD) scheme
	in the vicinity of an interface.
The CFM was applied on Poisson's equation with piecewise constant coefficients \cite{Marques2017} 
	and on the wave equation with constant coefficients \cite{Abraham2018}.
	
In addition to the difficulties associated with the treatment of the interface, 
	one needs to satisfy at the discrete level or to accurately approximate the divergence-free constraints 
	coming from Maxwell's equations to obtain accurate results.
Many numerical methods were proposed to enforce these constraints,
	such as Yee's scheme \cite{Yee1966} in finite-difference time-domain (FDTD) methods,  
	local divergence-free shape functions in finite element methods \cite{Cockburn2004,Brenner2008,Jin2014} and
	penalization approaches \cite{Assous1993,Munz2000}.
	
In this work, 
	we focus on the construction of high-order FDTD schemes for arbitrarily complex interfaces
	without increasing the complexity of the numerical scheme for constant coefficients. 
The main goal of this paper is to demonstrate the feasibility to construct such schemes using the CFM.
To our opinion, 
	this is the first necessary stepping stone towards a general numerical approach to treat interface conditions with 
	discontinuous coefficients.
Discontinuous coefficients introduce additional complexity in the context of the CFM,
	and we will address such problems in future work.
We choose FDTD schemes composed of a staggered finite difference scheme in space, 
	similar to what is done for Yee's scheme, 
	and the fourth-order Runge-Kutta method as a time-stepping method.  
The staggered grid in space guarantees that the nodes far from the interface satisfy 
	the divergence constraints at the discrete level.
The CFM requires a functional to be minimized in a chosen functional space.
In our case, 
	the functional coming from correction functions' system of PDEs is minimized within a divergence-free functional space, 
	which again enforces the divergence constraints.
Two-dimensional numerical examples based on the transversal magnetic  (TM$_z$)  mode are investigated to verify the proposed numerical strategy.

The paper is structured as follows.
In \cref{sec:defPblm},
	we define the problem, 
	namely Maxwell's equations with interface jump conditions.
The correction function method is introduced in \cref{sec:CFM}.
We derive the correction functions' system of PDEs coming from Maxwell's equations 
	and perform a perturbation analysis. 
The minimization procedure of the discrete problem is described.
The combination of the staggered finite difference scheme with the fourth-order Runge-Kutta method 
	and the CFM is presented 
	in \cref{sec:2D}.
The consistency and the discrete divergence constraint of the proposed schemes are discussed.
Several two-dimensional numerical examples with complex interfaces are investigated in \cref{sec:numEx}.

\section{Definition of the Problem} \label{sec:defPblm}
Consider a domain $\Omega$ subdivided into two subdomains $\Omega^+$ and $\Omega^-$ for which the interface $\Gamma$
	between the subdomains is stationary, 
	that is it does not vary in time, 
	and allows the magnetic field and the electric field to be discontinuous.
The jumps in the magnetic field and the electric field are denoted as 
\begin{equation*}
\begin{aligned}
\llbracket \mathbold{H} \rrbracket =&\,\, \mathbold{H}^+ - \mathbold{H}^-, \\
\llbracket \mathbold{E} \rrbracket =&\,\, \mathbold{E}^+ - \mathbold{E}^-,
\end{aligned}
\end{equation*}
where $\mathbold{H}^+$ and $\mathbold{E}^+$ are the solutions in $\Omega^+$, 
	and $\mathbold{H}^-$ and $\mathbold{E}^-$ are the solutions in $\Omega^-$.
We also consider the boundary $\partial \Omega$ 
	and a time interval $I = [ 0, T ]$.
The geometry of a typical domain is illustrated in \cref{fig:typicalDomain}.  
 Assuming linear media in such a domain and Ohm's law, 
	Maxwell's equations are then given by
\begin{subequations} \label{eq:pblmDefinition}
\begin{align}
\partial_t (\mu\,\mathbold{H}) + \nabla\times \mathbold{E} =&\,\, 0 \quad \text{in } \Omega \times I, \label{eq:Faraday} \\
\partial_t (\epsilon\,\mathbold{E}) - \nabla\times\mathbold{H} =&\,\, -\sigma\,\mathbold{E} \quad \text{in } \Omega \times I , \label{eq:AmpereMaxwell}\\
\nabla\cdot(\epsilon\,\mathbold{E}) =&\,\, \rho \quad \text{in } \Omega \times I , \label{eq:divE}\\
\nabla\cdot(\mu\,\mathbold{H}) =&\,\, 0 \quad \text{in } \Omega \times I , \label{eq:divH}\\
\hat{\mathbold{n}}\times\llbracket \mathbold{E} \rrbracket =&\,\, 0 \quad \text{on } \Gamma \times I ,\label{eq:tangentEInterf}\\
\hat{\mathbold{n}}\times\llbracket \mathbold{H} \rrbracket =&\,\, \mathbold{J}_s \quad \text{on } \Gamma \times I ,\label{eq:tangentHInterf}\\
\hat{\mathbold{n}}\cdot\llbracket \epsilon\,\mathbold{E} \rrbracket =&\,\, \rho_s \quad \text{on } \Gamma \times I ,\label{eq:normalEInterf}\\
\hat{\mathbold{n}}\cdot\llbracket \mu\,\mathbold{H} \rrbracket =&\,\, 0 \quad \text{on } \Gamma \times I ,\label{eq:normalHInterf}\\
\mathbold{n}\times\mathbold{H} =&\,\, \mathbold{e}(\mathbold{x},t) \quad \text{on } \partial \Omega \times I,\label{eq:DirichletH}\\
\mathbold{n}\times\mathbold{E} =&\,\, \mathbold{g}(\mathbold{x},t) \quad \text{on } \partial \Omega \times I,\label{eq:DirichletE}\\
\mathbold{H} =&\,\, \mathbold{H}(\mathbold{x},0)	\quad \text{in } \Omega, \label{eq:InitialCdnH}\\
\mathbold{E} =&\,\, \mathbold{E}(\mathbold{x},0)	\quad \text{in } \Omega, \label{eq:InitialCdnE}
\end{align}
\end{subequations}
	where $\mu$ is the magnetic permeability, 
	$\epsilon$ is the electric permittivity,
	$\sigma$ is the conductivity,
	$\rho$ is the electric charge density,
	$\mathbold{J}_s$ is the surface current density,
	$\rho_s$ is the surface charge density,
	$\mathbold{n}$ is the unit outward normal to $\partial \Omega$ and 
	$\hat{\mathbold{n}}$ is the unit normal to the interface $\Gamma$ pointing toward $\Omega^+$.
Equation \cref{eq:Faraday} to \cref{eq:divE} are known respectively as Faraday's law, 
	Amp\`{e}re-Maxwell's law and Gauss' law.
The divergence-free constraint on the magnetic induction field is given by equation \cref{eq:divH}. 
Interface conditions on $\Gamma$ are given by equations \cref{eq:tangentEInterf} to \cref{eq:normalHInterf}, 
	and boundary conditions and initial conditions are given by equations \cref{eq:DirichletH} 
	to \cref{eq:InitialCdnE}.
Even if divergence constraints \cref{eq:divE} and \cref{eq:divH} seem to be redundant, 
	it is important to consider them in order to guarantee the uniqueness of the solution \cite{Jiang1996}.
As mentioned in the introduction,  
	it also helps to obtain accurate numerical solutions. 
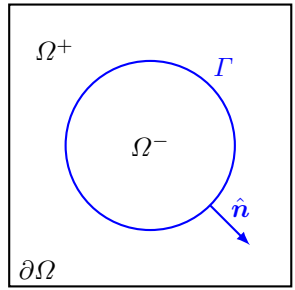
\begin{figure}[htbp]
 	\centering
 	\tdplotsetmaincoords{75}{105}
	\tikzset{external/export next=false}
  \begin{tikzpicture}[scale=0.75]
   	\draw[-latex,thick] (0,0)--(0,5)--(5,5)--(5,0)--cycle; 
	\draw[-latex,thick,blue] (2.5,2.5) circle [radius=1.5];
	\draw[-latex,thick,blue] (3.56066,1.43934)--(4.26777,0.732233);
  	\draw (4.1,1.4) node {$\color{blue}\hat{\mathbold{n}}$};
  	\draw (3.8,3.9) node {$\color{blue}\Gamma$};
  	\draw (0.8,4.2) node {$\Omega^+$};
   	\draw (2.5,2.5) node {$\Omega^-$};
  	\draw (0.5,0.3) node {$\partial \Omega$};
  \end{tikzpicture}
  \caption{Geometry of a domain $\Omega$ with an interface $\Gamma$.}
\label{fig:typicalDomain}
\end{figure}

To ease the verification of the proposed FDTD schemes, 
	we use divergence-free source terms in each subdomain, 
	that is $\mathbold{f}_1^+(\mathbold{x},t)$ in $\Omega^+$ and 
	$\mathbold{f}_1^-(\mathbold{x},t)$ in $\Omega^-$, 
	for Faraday's law.  
For Amp\`{e}re-Maxwell's law, 
	we consider $\mathbold{f}_2^+(\mathbold{x},t)$ and $\mathbold{f}_2^-(\mathbold{x},t)$ respectively 
	in $\Omega^+$ and $\Omega^-$ as source terms. 
We also use more general interface conditions, 
	given by 
\begin{equation*}
	\begin{aligned}
	\hat{\mathbold{n}}\times\llbracket \mathbold{E} \rrbracket =&\,\, \mathbold{a}(\mathbold{x},t)  & \text{on } & \Gamma \times I ,\\
	\hat{\mathbold{n}}\times\llbracket \mathbold{H} \rrbracket =&\,\, \mathbold{b}(\mathbold{x},t)  & \text{on } & \Gamma \times I ,\\
	\hat{\mathbold{n}}\cdot\llbracket \epsilon\,\mathbold{E} \rrbracket =&\,\, c(\mathbold{x},t)  & \text{on } & \Gamma \times I ,\\
	\hat{\mathbold{n}}\cdot\llbracket \mu\,\mathbold{H} \rrbracket =&\,\, d(\mathbold{x},t)  & \text{on } & \Gamma \times I.
	\end{aligned}
\end{equation*}
Hence, 
	we allow both the tangential and normal components 
	of $\mathbold{H}$ and $\mathbold{E}$ across the interface to be discontinuous.
Even if these source terms and interface conditions are not substantiated by physics, 
	it helps the verification of the numerical approach by using manufactured solutions in a more general framework.

\section{Correction Function Method} \label{sec:CFM}

In this section, 
	we first present the idea behind the correction function method and the benefits of using it.
We then define a system of PDEs coming from problem \cref{eq:pblmDefinition} that models correction functions.
A perturbation analysis is performed on the correction functions' system of PDEs.
A quadratic functional that is a square measure of the error associated with the correction functions' system of PDEs 
	is then derived.
This functional is then minimized in a discrete functional space to obtain approximations of correction functions.
Particular attention is paid to the choice of the discrete functional space in order to guarantee the divergence-free constraint. 

\subsection{Introduction to the CFM}
Noticing first that the solution to problem \cref{eq:pblmDefinition} is discontinuous, 
	one cannot use {\it a priori} a numerical method, 
	such as a standard finite difference method,
	that requires at least the solution to be in $\mathcal{C}^1(\Omega)$.
In the following,
	we show how to circumvent this issue by using a correction function
	that extends the solution in different subdomains and,
	hence, 
	allow us to use FD schemes.
	
For simplicity and without loss of generality, 
	we show the principle behind the CFM through an 1-D example problem. 
Let us assume a domain $\Omega = [x_{\ell},x_r]$ divided in $N_x$ cells. 
The nodes are defined as $x_{i+1/2} = x_{\ell} + i\,\Delta x$ for $i = 0, \dots, N_x$, 
	where $\Delta x = \tfrac{x_r-x_{\ell}}{N_x}$.
For a given $i$, 
	we now consider an interface $\Gamma$ between $x_{i-1/2} \in \Omega^+$ and $x_{i+1/2} \in \Omega^-$.
Let us suppose that we want to compute a second-order approximation of the first derivative of $H(x)$ at the cell center $x_i \in \Omega^+$. 
We clearly have 
$$\partial_xH^+(x_i) \approx \partial_x H^+_i \neq \frac{H^-_{i+1/2}-H^+_{i-1/2}}{\Delta x}$$
	because of the discontinuity at the interface $\Gamma$.
However, 
	assuming for the moment that we can extend the solution $H^+$ in the domain $\Omega^-$ in such a way that
\begin{equation*}
	\begin{aligned}
\partial_x H^+_i =&\,\, \frac{H^+_{i+1/2}-H^+_{i-1/2}}{\Delta x} \\
			=&\,\, \frac{(H^-_{i+1/2} + D_{i+1/2})- H^+_{i-1/2}}{\Delta x} \\
			=&\,\, \frac{H^-_{i+1/2}- H^+_{i-1/2}}{\Delta x} + \frac{D_{i+1/2}}{\Delta x},
	\end{aligned}
\end{equation*}
	where $D_{i+1/2} = H_{i+1/2}^+ - H_{i+1/2}^-$ is a correction function evaluated at $x_{i+1/2}$. 
We are therefore able to compute an accurate approximation of $\partial_x H^+_i$.
In a PDE context, 
	the term $\frac{D_{i+1/2}}{\Delta x}$ acts as a source term.
In the next subsection, 
	we build the governing correction functions' system of PDEs coming from Maxwell's equations \cref{eq:pblmDefinition} 
	for which the solutions are defined as correction functions, 
	namely $D$ in the above 1-D example.
	
\subsection{CFM for Maxwell's equations} \label{sec:MaxwellCFM}
To find the correction functions' system of PDEs associated with Maxwell's equations, 
	we consider a small region $\Omega_{\Gamma}$ of the domain that encloses the interface $\Gamma$. 
We assume that $\mathbold{H}^+$, 
	$\mathbold{H}^-$, 
	$\mathbold{E}^+$, 
	$\mathbold{E}^-$ and the associated source terms can be smoothly extended in $\Omega_{\Gamma}\times I$ in such a way that 
	Maxwell's equations are still satisfied, 
	that is 
\begin{equation} \label{eq:subDomainCF}
\begin{aligned}
\mu\,\partial_t \mathbold{H}^+ + \nabla\times \mathbold{E}^+ =&\,\, \mathbold{f}^+_1(\mathbold{x},t) \quad \text{in } \Omega_{\Gamma} \times I , \\
\epsilon\,\partial_t \mathbold{E}^+ - \nabla\times\mathbold{H}^+ =&\,\, -\sigma\,\mathbold{E}^+ + \mathbold{f}^+_2(\mathbold{x},t) \quad \text{in } \Omega_{\Gamma} \times I , \\
\nabla\cdot\mathbold{E}^+ =&\,\, \frac{\rho}{\epsilon} \quad \text{in } \Omega_{\Gamma} \times I , \\
\nabla\cdot\mathbold{H}^+ =&\,\, 0 \quad \text{in } \Omega_{\Gamma} \times I , \\
\mu\,\partial_t \mathbold{H}^- + \nabla\times \mathbold{E}^- =&\,\, \mathbold{f}^-_1(\mathbold{x},t) \quad \text{in } \Omega_{\Gamma} \times I , \\
\epsilon\,\partial_t \mathbold{E}^- - \nabla\times\mathbold{H}^- =&\,\, -\sigma\,\mathbold{E}^- + \mathbold{f}^-_2(\mathbold{x},t) \quad \text{in } \Omega_{\Gamma} \times I, \\
\nabla\cdot\mathbold{E}^- =&\,\, \frac{\rho}{\epsilon} \quad \text{in } \Omega_{\Gamma} \times I , \\
\nabla\cdot\mathbold{H}^- =&\,\, 0 \quad \text{in } \Omega_{\Gamma} \times I, \\
\hat{\mathbold{n}}\times\llbracket \mathbold{E} \rrbracket =&\,\, \mathbold{a}(\mathbold{x},t) \quad \text{on } \Gamma \times I ,\\
\hat{\mathbold{n}}\times\llbracket \mathbold{H} \rrbracket =&\,\, \mathbold{b}(\mathbold{x},t) \quad \text{on } \Gamma \times I ,\\
\hat{\mathbold{n}}\cdot\llbracket \epsilon\,\mathbold{E} \rrbracket =&\,\, c(\mathbold{x},t) \quad \text{on } \Gamma \times I ,\\
\hat{\mathbold{n}}\cdot\llbracket \mu\,\mathbold{H} \rrbracket =&\,\, d(\mathbold{x},t) \quad \text{on } \Gamma \times I.
\end{aligned}
\end{equation}
Subtracting from the equations for $\mathbold{H}^+$ and $\mathbold{E}^+$ the equations for $\mathbold{H}^-$ and $\mathbold{E}^-$ of system \cref{eq:subDomainCF}, 
	we obtain the following system of equations
\begin{subequations} \label{eq:systCFM}
\begin{align}
\mu\,\partial_t \mathbold{D}_H + \nabla\times \mathbold{D}_E =&\,\, \mathbold{f}_{D_1}(\mathbold{x},t) \quad \text{in } \Omega_{\Gamma} \times I , \label{eq:FaradayCFM} \\
\epsilon\,\partial_t \mathbold{D}_E - \nabla\times\mathbold{D}_H =&\,\, -\sigma\,\mathbold{D}_E + \mathbold{f}_{D_2}(\mathbold{x},t) \quad \text{in } \Omega_{\Gamma}  \times I , \label{eq:AmpereMaxwellCFM}\\
\nabla\cdot\mathbold{D}_E =&\,\, 0 \quad \text{in } \Omega_{\Gamma}  \times I , \label{eq:divECFM}\\
\nabla\cdot\mathbold{D}_H =&\,\, 0 \quad \text{in } \Omega_{\Gamma}  \times I , \label{eq:divHCFM}\\
\hat{\mathbold{n}}\times \mathbold{D}_E  =&\,\, \mathbold{a}(\mathbold{x},t) \quad \text{on } \Gamma \times I ,\label{eq:tangentEInterfCFM}\\
\hat{\mathbold{n}}\times \mathbold{D}_H  =&\,\, \mathbold{b}(\mathbold{x},t) \quad \text{on } \Gamma \times I ,\label{eq:tangentHInterfCFM}\\
\hat{\mathbold{n}}\cdot \mathbold{D}_E  =&\,\, c(\mathbold{x},t)/\epsilon \quad \text{on } \Gamma \times I ,\label{eq:normalEInterfCFM}\\
\hat{\mathbold{n}}\cdot \mathbold{D}_H  =&\,\, d(\mathbold{x},t)/\mu \quad \text{on } \Gamma \times I ,\label{eq:normalHInterfCFM}
\end{align}
\end{subequations}
	which determine the correction functions $\mathbold{D}_H = \llbracket \mathbold{H} \rrbracket$ and 
	$\mathbold{D}_E = \llbracket \mathbold{E} \rrbracket$.
Source terms are given by $\mathbold{f}_{D_1} = \mathbold{f}^+_1 - \mathbold{f}^-_1$ and
	$\mathbold{f}_{D_2} = \mathbold{f}^+_2-\mathbold{f}^-_2$. 
Interface conditions \cref{eq:tangentEInterf} to \cref{eq:normalHInterf} 
	become boundary conditions \cref{eq:tangentEInterfCFM} to \cref{eq:normalHInterfCFM} for system \cref{eq:systCFM}. 
	
\begin{remark}
It is worth to mention that system \cref{eq:systCFM} describes the behaviour of jumps 
	(or correction functions) in the magnetic field and 
	the electric field in a general approach.
	Hence, 
		by construction and consistency, 
		derivatives of correction functions $\mathbold{D}_H$ and $\mathbold{D}_E$ satisfy 
		derivative jump conditions \cite{Zhao2004} without explicitly imposing them.
\end{remark}

\subsection{Perturbation Analysis of CF's PDEs for Maxwell's Equations} \label{sec:wellPosedCFM}
In this subsection, 
	a perturbation analysis of the correction functions' system of PDEs coming from Maxwell's equations is investigated using 
	a standard Fourier analysis for initial value problem.
We follow the same procedure described in \cite{Marques2011,Abraham2018}.
The correction function's system of PDEs is not always well-posed.
An example of such a situation is Poisson problems for which the CFM leads to an ill-posed Cauchy problem \cite{Marques2011}.
This could influence the choice of the numerical scheme to be corrected and the construction of the discretization of the correction functions' system of PDEs. 

In the following, 
	we only focus on the first two equations of \cref{eq:systCFM} because divergence constraints are naturally satisfied by an appropriate choice of the functional space in which 
	we minimize the quadratic functional (see \cref{sec:discMaxwellCFM}).
We suppose, 
	without loss of generality, 
	that the interface is flat and is parallel to the $xy$-plane and $\mathbold{x} = 0 \in \Gamma$.
Let us also define the distance $d$ from the interface, 
	which is along the positive part of the $z$-axis in the subdomain $\Omega^+$. 
We therefore have an orthogonal coordinate system $(\mathbold{y},d)$, 
	where $\mathbold{y} = [x, y]^T$ spans the interface and $d=z$. 
Assume that physical parameters are such that $\mu >0$, $\epsilon > 0$ and $\sigma > 0$ and 
	there is no source term.
Consider a periodic domain $\Omega=[-\pi,\pi]^3$, 
	we search solutions for small perturbations of $\mathbold{D}_H$ and $\mathbold{D}_E$ on the interface, 
	namely $\tilde{\mathbold{D}}_H$ and $\tilde{\mathbold{D}}_H$, of the form 
\begin{equation} \label{eq:solFourierSeries}
	\tilde{\mathbold{U}}(\mathbold{x},t) = \sum\limits_{k_x,k_y,k_z \in \mathbb{Z}} \hat{\mathbold{U}}_{k_x,k_y,k_z}(t)\,e^{i\,\mathbold{k}\cdot\mathbold{x}},
\end{equation}
	where $\tilde{\mathbold{U}} = \big[ \tilde{\mathbold{D}}_H^T \;\; \tilde{\mathbold{D}}_E^T\big]^T$ and $\mathbold{k} = [k_x,k_y,k_z]^T$.
Substitute \cref{eq:solFourierSeries} into the first two equations of \cref{eq:systCFM} with $\mathbold{f}_{D_1} = \mathbold{f}_{D_2} = 0$ leads to a system of ordinary differential equations (ODE) for each coefficient, 
	given by : 
\begin{equation*}
	\partial_t \hat{\mathbold{U}}_{k_x,k_y,k_z} = A\,\hat{\mathbold{U}}_{k_x,k_y,k_z}
\end{equation*}
	with 
\begin{equation*}
A = 
	\begin{bmatrix} 
		0 & 0 & 0& 0 & i\,k_z/\mu & -i\,k_y/\mu \\
		0 & 0 & 0 & -i\,k_z/\mu & 0 & i\,k_x/\mu \\
		0 & 0 & 0 & i\,k_y/\mu & -i\,k_x/\mu & 0 \\
		0 & -i\,k_z/\epsilon & i\,k_y/\epsilon & -\sigma/\epsilon & 0 & 0 \\
		i\,k_z/\epsilon & 0 & -i\,k_x/\epsilon & 0 & -\sigma/\epsilon & 0 \\
		-i\,k_y/\epsilon & i\,k_x/\epsilon & 0 & 0 & 0 & -\sigma/\epsilon
	\end{bmatrix}.
\end{equation*}
Depending on the values of $\mathbold{k}\cdot\mathbold{k}$, 
	we have three cases:
\begin{itemize}
	\item[1)] If $\mathbold{k}\cdot\mathbold{k} = 0$, 
		we have $\mathbold{k} = 0$ and the matrix $A$ has two distinct eigenvalues 
		$\lambda_1 = 0$ and $\lambda_2 =-\tfrac{\sigma}{\epsilon}$.
		It is easy to show that $\dim (\ker(A-\lambda_i\,I)) = 3$ for $i = 1,2$, 
		and that $B = [ \mathbold{s}_1 \dots \mathbold{s}_6] = I$,
		where $\mathbold{s}_j$ for $j=1,\dots,6$ denotes an eigenvector. 
		Hence, 
		we have six linearly independent eigenvectors.
	
	\item[2)] If $\mathbold{k}\cdot\mathbold{k} = \tfrac{\mu\,\sigma^2}{4\,\epsilon}$, 
		the matrix $A$ has three distinct eigenvalues $\lambda_1=0$, 
		$\lambda_2 = -\tfrac{\sigma}{\epsilon}$ and $\lambda_3 = -\tfrac{\sigma}{2\,\epsilon}$.
		We have $\dim (\ker(A-\lambda_1\,I)) =\dim (\ker(A-\lambda_2\,I)) =1$.
		However, 
		the multiplicity of $\lambda_3$ is four,
		but $\dim(\ker(A-\lambda_3\,I))=2$.
		We therefore need to find two other solutions of the form $\mathbold{c} =  \mathbold{s}\,t+\mathbold{b}$ associated with 
		eigenvectors of $\lambda_3$.
		Using a standard method to solve an ODE with multiple eigenvalues, 
			we find 
			$$\det(B) = \frac{\epsilon^2\,\sigma^2\,\mu}{4\,\epsilon\,k_y^2+4\,\epsilon\,k_z^2-\mu\,\sigma^2} \neq 0,$$
			where $B=[\mathbold{s}_1 \dots \mathbold{s}_4 \;\; \mathbold{c}_1 \;\; \mathbold{c}_2]$.		
	
	\item[3)] Otherwise, 
		the matrix $A$ has four distinct eigenvalues given by $\lambda_1 = 0$, 
		$\lambda_2 = -\tfrac{\sigma}{\epsilon}$ and 
		$$\lambda_{3,4} = \frac{-\sigma\,\mu \pm \sqrt{\mu(4\,\epsilon\,\mathbold{k}\cdot\mathbold{k} - \mu\,\sigma^2)}\,i}{2\,\epsilon\,\mu}.$$
		We have $\dim(\ker(A-\lambda_i)) = 1$ for $i=1,2$, and $\dim(\ker(A-\lambda_i)) = 2$ for $i=3,4$. 
		A direction computation of $\det(B)$ shows that we have six linearly independent eigenvectors.
\end{itemize} 
For all cases, 
	it is possible to obtain a general solution of the form
\begin{equation*}
	\hat{\mathbold{U}}_{k_x,k_y,k_z}(t) = \sum_{i} \mathbold{a}_i\,e^{\lambda_it},
\end{equation*}
	where the vectors $\mathbold{a}_i$ are computed using given initial conditions of small perturbations and eigenvectors.
Since $\sigma>0$ and $\epsilon>0$, 
	there is no exponential growth of the form $e^{a\,t}$ with $a>0$.
Hence, 
	the problem coming from the first two equations of \cref{eq:systCFM} does not allow  perturbations to growth. 
A perturbation of $\mathbold{D}_H$ and $\mathbold{D}_E$ on the interface $\Gamma$ is therefore unchanged, 
	dispersed and/or diffused.
Hence, 
	this allows us to have more flexibility on the discretization of the correction functions' system of PDEs 
	(see \cref{sec:discMaxwellCFM}) and the choice of an appropriate numerical scheme.

\begin{remark}
For highly resistive medium, 
	it is common to consider $\sigma = 0$.
In this case, 
	if $\mathbold{k}\cdot\mathbold{k}\neq0$,
	the matrix $A$ has three distinct eigenvalues $\lambda_1 = 0$ and $$\lambda_{2,3} = \pm\sqrt{\frac{\mathbold{k}\cdot\mathbold{k}}{\epsilon\,\mu}}\,i.$$
Following the same procedure than the one for $\sigma>0$, 
	we find that the problem coming from the first two equations of \cref{eq:systCFM} does not allow perturbations to growth.
\end{remark}

\subsection{Discretization of Maxwell's equations CFM} \label{sec:discMaxwellCFM}
In this subsection, 
	we define a local patch $\Omega^{h}_{\Gamma} \subset \Omega_{\Gamma}$ and 
	a time interval $I_\Gamma^h = [t_n-\Delta t_\Gamma, t_n]$, 
	where correction functions,
	namely $\mathbold{D}_H$ and $\mathbold{D}_E$,
	need to be computed at a node $(\mathbold{x},t) \in \Omega^{h}_{\Gamma}\times I_\Gamma^h$.
Approximations of correction functions within a 
	patch are obtained by minimizing a quadratic functional. 

The construction of a patch is a slight modification of the \enquote{Node Centered} technique \cite{Marques2011}. 
It is recalled that the correction functions' system of PDEs for Maxwell's equations does not allow perturbations to growth. 
Hence, 
	some restrictions on the construction of the local patch are loosened, 
	such as the size of the patch and the representation of the interface within the patch. 
As in the \enquote{Node Centered} approach, 
	we construct a patch for each node that needs to be corrected. 
However, 
	we restrict the patch to be squared and aligned with the computational grid. 
We now summarize the procedure to compute $\Omega^{h}_{\Gamma}$.
For a given node $\mathbold{x}_c$ that needs to be corrected, 
	we find an approximation of the point $\mathbold{p}$ on the interface $\Gamma$ that is the closest to $\mathbold{x}_c$.  
We construct a square centered at $\mathbold{p}$ of length $\ell_h = \beta\,\max\{\Delta x,\Delta y,\Delta z\}$ where $\beta$ is a positive constant.
The parameter $\beta$ depends on the FD scheme and it is chosen
	to ensure that $\mathbold{x}_c \in \Omega^{h}_{\Gamma}$. 
For exemple, 
	$\beta = 1$ and $\beta = 3$ for respectively the second and the fourth order staggered FD scheme presented in \cref{sec:2D}.
This construction of the patch guarantees the uniqueness of a correction function at each node.
This is important for the conservation of the discrete divergence constraint for some nodes close to $\Gamma$ 
	(see \cref{thm:discreteDivergenceFreeWithCFMPart1}).
	
Let us now present the functional to be minimized in order to obtain approximations of correction functions.
We begin by introducing some notations. 		
The inner product in $L^2\big(\Omega_{\Gamma}^{h}\times I_\Gamma^h\big)$ is defined by
$$\langle\mathbold{v},\mathbold{w}\rangle = \int\limits_{I_\Gamma^h}\int\limits_{\Omega_{\Gamma}^{h}}\!\!\mathbold{v}\cdot\mathbold{w}\,\mathrm{d}V\,\mathrm{d}t.$$ 
For legibility, 
	we also use the notation  
	$$\langle \mathbold{v},\mathbold{w} \rangle_{\Gamma} = \int\limits_{I_\Gamma^h}\int\limits_{\Omega_{\Gamma}^{h}\cap\Gamma}\!\!\mathbold{v}\cdot\mathbold{w} \, \mathrm{d}S\,\mathrm{d}t.$$ 
To compute approximations of correction functions $\mathbold{D}_H$ and $\mathbold{D}_E$, 
	we consider the following quadratic functional to minimize
\begin{equation*}
\begin{aligned}
J(\mathbold{D}_H,&\mathbold{D}_E) = \frac{\ell_c}{2} \, \big\langle\mu\,\partial_t \mathbold{D}_H + \nabla\times\mathbold{D}_E - \mathbold{f}_{D_1},\mu\,\partial_t \mathbold{D}_H + \nabla\times\mathbold{D}_E - \mathbold{f}_{D_1}\big\rangle \\
+&\,\, 
\frac{\ell_c}{2} \, \langle \epsilon\,\partial_t \mathbold{D}_E - \nabla\times\mathbold{D}_H + \sigma\,\mathbold{D}_E - \mathbold{f}_{D_2},\epsilon\,\partial_t \mathbold{D}_E - \nabla\times\mathbold{D}_H + \sigma\,\mathbold{D}_E - \mathbold{f}_{D_2} \big\rangle\\
+&\,\, \frac{1}{2} \,\big\langle \hat{\mathbold{n}}\times\mathbold{D}_H - \mathbold{b},\hat{\mathbold{n}}\times\mathbold{D}_H - \mathbold{b}\big\rangle_{\Gamma} + \frac{1}{2}\,\big\langle\hat{\mathbold{n}}\cdot \mathbold{D}_H - \tfrac{d}{\mu}, \hat{\mathbold{n}}\cdot \mathbold{D}_H - \tfrac{d}{\mu}\big\rangle_{\Gamma}\\
+&\,\, \frac{1}{2} \, \big\langle\hat{\mathbold{n}}\times\mathbold{D}_E - \mathbold{a},\hat{\mathbold{n}}\times\mathbold{D}_E - \mathbold{a}\big\rangle_{\Gamma}+\frac{1}{2} \, \big\langle\hat{\mathbold{n}}\cdot \mathbold{D}_E - \tfrac{c}{\epsilon},\hat{\mathbold{n}}\cdot \mathbold{D}_E - \tfrac{c}{\epsilon} \big\rangle_{\Gamma} ,
\end{aligned}
\normalsize
\end{equation*}
	where $\ell_c>0$ is a scale factor.
The scale factor $\ell_c$ is chosen to ensure that all terms in the functional $J$ behave in a similar way when the 
	computational grid is refined (see \cref{rem:scaleFactor}).
As one can observe,
	we do not explicitly consider the divergence-free constraint \cref{eq:divECFM} and \cref{eq:divHCFM}. 
These constraints are naturally satisfied by an appropriate choice of polynomial spaces in which we minimize 
	the functional $J$.
The problem statement is then 
\begin{equation} \label{eq:minPblm}
\text{Find } (\mathbold{D}_H,\mathbold{D}_E) \in V \times W \text{ such that } (\mathbold{D}_H,\mathbold{D}_E) \in  \underset{\mathbold{v}\in V, \mathbold{w} \in W}{\arg\min}J(\mathbold{v},\mathbold{w}),
\end{equation}
	where $V$ and $W$ are two divergence-free polynomial spaces that is 
\begin{equation*}
	V = \big\{ \mathbold{v} \in \big[P^k\big(\Omega_{\Gamma}^{h}\times I_\Gamma^h \big)\big]^3 : \nabla\cdot\mathbold{v} = 0 \big\},
\end{equation*}
	where $P^k$ denotes the space of polynomials of degree $k$, 
	and $V=W$.
Space-time basis functions of $V$ are obtained using the tensor product between basis functions of $P^k(I_\Gamma^h)$ 
	and basis functions of 
$$\tilde{V} = \big\{ \mathbold{v} \in \big[P^k\big(\Omega_{\Gamma}^{h}\big)\big]^3 : \nabla\cdot\mathbold{v} = 0\big\}.$$

Computing Gateaux derivatives and using a necessary condition to obtain a minimum, 
	we have the following problem : \\
	
Find $(\mathbold{D}_H,\mathbold{D}_E) \in V\times W$ such that
\small
\begin{equation*}
\left\{
\begin{aligned}
\ell_c\,\big\langle\mu^2\,\partial_t \mathbold{D}_H &+ \mu\,\nabla\times\mathbold{D}_E - \mu\,\mathbold{f}_{D_1},\partial_t\mathbold{v}\big\rangle - \ell_c\,\big\langle\epsilon\,\partial_t \mathbold{D}_E + \nabla\times\mathbold{D}_H- \sigma\,\mathbold{D}_E + \mathbold{f}_{D_2},\nabla\times\mathbold{v}\big\rangle \\
&+ \big\langle\hat{\mathbold{n}}\times\mathbold{D}_H-\mathbold{b}, \hat{\mathbold{n}}\times\mathbold{v}\big\rangle_{\Gamma} + \big\langle\hat{\mathbold{n}}\cdot\mathbold{D}_H - \tfrac{d}{\mu}, \hat{\mathbold{n}}\cdot\mathbold{v}\big\rangle_{\Gamma} = 0, \qquad \forall \mathbold{v} \in V, \\
\ell_c\,\big\langle\mu\,\partial_t\mathbold{D}_H &+ \nabla\times\mathbold{D}_E - \mathbold{f}_{D_2}, \nabla\times\mathbold{w}\big\rangle + \ell_c\,\big\langle\epsilon^2\,\partial_t \mathbold{D}_E-\epsilon\,\nabla\times\mathbold{D}_H + \epsilon\,\sigma\,\mathbold{D}_E - \epsilon\,\mathbold{f}_{D_2},\partial_t\mathbold{w}\big\rangle\\
&+\big\langle\sigma\,\epsilon\,\partial_t \mathbold{D}_E-\sigma\,\nabla\times\mathbold{D}_H+\sigma^2\,\mathbold{D}_E-\sigma\,\mathbold{f}_{D_2},\mathbold{w}\big\rangle\\
&+\big\langle\hat{\mathbold{n}}\times\mathbold{D}_E-\mathbold{a},\hat{\mathbold{n}}\times\mathbold{w}\big\rangle_{\Gamma} +\big\langle\hat{\mathbold{n}}\cdot\mathbold{D}_E-\tfrac{c}{\epsilon},\hat{\mathbold{n}}\cdot\mathbold{w}\big\rangle_{\Gamma} = 0, \quad\;\,\, \forall \mathbold{w}\in W.
\end{aligned}
\right.
\end{equation*}
\normalsize
\begin{remark}
For simplicity, 
	consider the 1-D version of system \cref{eq:systCFM} with $\sigma = 0$, 
	$\rho=0$ and without source term, 
	it can be shown that the information is propagated at a speed of $\frac{1}{\sqrt{\epsilon\,\mu}}$ as 
	it is well-known for homogeneous Maxwell's equations.
This gives us an insight on how to choose an appropriate time step $\Delta t_\Gamma$ for the CFM. 
For the general case, 
	we choose $\Delta t_\Gamma\approx \sqrt{\epsilon\,\mu}\,\ell_h$ 
	to allow information coming from the interface $\Gamma$ to propagate in the whole local patch $\Omega_\Gamma^h$.
\end{remark}
\begin{remark} \label{rem:scaleFactor}
Consider a square patch of length $\ell_h$ and $\Delta t_\Gamma = \mathcal{O}(\ell_h)$. 
Using discrete polynomial spaces $P^k$, 
	correction functions are $(k+1)$-order accurate and we have 
\begin{equation*}
\begin{aligned}
\mu\,\partial_t \mathbold{D}_H + \nabla\times\mathbold{D}_E - \mathbold{f}_{D_1} =&\,\, \mathcal{O}(\ell_h^k), \\
\epsilon\,\partial_t \mathbold{D}_E - \nabla\times\mathbold{D}_H + \sigma\,\mathbold{D}_E - \mathbold{f}_{D_2} =&\,\, \mathcal{O}(\ell_h^k), \\
\hat{\mathbold{n}}\times\mathbold{D}_E - \mathbold{a}=&\,\, \mathcal{O}(\ell_h^{k+1}), \\
\hat{\mathbold{n}}\times\mathbold{D}_H - \mathbold{b} =&\,\, \mathcal{O}(\ell_h^{k+1} ), \\
\hat{\mathbold{n}}\cdot \mathbold{D}_E - c/\epsilon=&\,\, \mathcal{O}(\ell_h^{k+1}), \\
\hat{\mathbold{n}}\cdot \mathbold{D}_H - d/\mu =&\,\, \mathcal{O}(\ell_h^{k+1}).
\end{aligned}
\end{equation*}
Substituting these terms in the functional $J$, 
	we find that the terms $\langle\cdot,\cdot\rangle$ and $\langle\cdot,\cdot\rangle_{\Gamma}$ behave respectively 
	as $\mathcal{O}(\ell_c\,\ell_h^{2\,k+4})$ and $\mathcal{O}(\ell_h^{2\,k+5})$.
Hence, 
	we need $\ell_c = \ell_h$ to have all terms converging in a similar way when the computational grid is refined. 
\end{remark}
\begin{remark}
The computational cost of minimization problems for the CFM is not small. 
However, 
	only nodes around the interface need a correction. 
Assuming an uniform mesh of $N^d$ nodes, 
	where $d$ is the dimension and $N$ is the number of nodes used in each dimension,  
	the computational cost scales as $N^{d-1}$ \cite{Marques2011}. 
For large problems, 
	this cost then becomes less significant.
Moreover, 
	it has been shown that a parallel implementation of the CFM can help to overcome this issue \cite{Abraham2017} 
	and make the CFM suitable for more complex problems.
\end{remark}
\begin{remark}
In this work, 
	2-D numerical examples are investigated.
We use a similar procedure proposed by \cite{Cockburn2004} to generate basis functions of $\tilde{V}$. 
Besides being at divergence-free, 
	the dimension of $\tilde{V}$,
	given by $\frac{(k+1)(k+4)}{2}$, 
	is smaller than the dimension of $[P^k\big(\Omega_{\Gamma}^{h}\big)\big]^2$ given by $(k+1)(k+2)$.
This reduces the computational cost of the CFM.
\end{remark}


\section{$2$-D Staggered Discretization}
\label{sec:2D}

Considering the transverse magnetic (TM$_{\text{z}}$) mode, 
	the unknowns are $H_x(x,y,t)$, $H_y(x,y,t)$ and $E_z(x,y,t)$. 
For a domain $\Omega \subset \mathbb{R}^2$ and constant physical parameters, 
	problem \cref{eq:pblmDefinition} is then simplified to 
\begin{equation*}
	\begin{aligned}
	\mu\,\partial_t H_x + \partial_y E_z =&\,\, f_{1_x} \quad \text{in } \Omega \times I,\\
	\mu\,\partial_t H_y - \partial_x E_z  =&\,\, f_{1_y} \quad \text{in } \Omega \times I,\\
	\epsilon\,\partial_t E_z - \partial_x H_y + \partial_y H_x =&\,\, -\sigma\,E_z + f_2 \quad \text{in } \Omega \times I,\\
	\partial_x H_x + \partial_y H_y =&\,\, 0 \quad \text{in } \Omega \times I,\\
	\end{aligned}
\end{equation*}
	with the associated interface, 
	boundary and initial conditions. 
\begin{remark}
In this work, 
	we demonstrate the feasibility of the numerical strategy in 2-D using the TM$_{\text{z}}$ mode.  
From a conceptual point of view, 
	there is, 
	in principle,
	no additional difficulties if one chooses the transverse electric (TE$_{\text{z}}$) mode or a fully 3-D problem as long as $\rho=0$. 
However, 
	the implementation for a fully 3-D problem is more involved due 
	to the treatment of the interface which is a surface in 3-D.
It is worth noting that recent progress has been made to ease the implementation of the CFM in 3-D \cite{Marques2018_3D}.
\end{remark}

\subsection{Numerical Scheme} \label{sec:numScheme}

Let us now define the staggered space discretization which is similar to what is done in space for Yee's scheme.
For simplicity, 
	we consider a rectangular domain $\Omega \in [x_{\ell},x_r]\times[y_{b},y_{t}]$.
The nodes of the grid are defined as 
\begin{equation*}
(x_{i+1/2},y_{j+1/2}) = \big(x_{\ell} + i\,\Delta x,y_{b} +j\,\Delta y\big)
\end{equation*}
	for $i=0,1,\dots,N_x$ and $j=0,1,\dots,N_y$ with $\Delta x := (x_r-x_{\ell})/N_x$ and $\upDelta y := (y_t-y_{b})/N_y$.
We also define the center of a cell
	$\Omega_{i,j} = [x_{i-1/2},x_{i+1/2}]\times[y_{j-1/2},y_{j+1/2}]$
	by
\begin{equation*}
        (x_i,y_i) = \big(x_{\ell}+(i-\tfrac{1}{2})\,\Delta x, y_{b}+(j-\tfrac{1}{2})\,\Delta y\big)
\end{equation*}
	for $i = 1, \dots, N_x$ and for $j = 1, \dots, N_y$.
The midpoints of edges parallel to the $x$-axis and those parallel to the $y$-axis are respectively defined as 
\begin{equation*}
	(x_i,y_{j+1/2}) = \big(x_{\ell}+(i-\tfrac{1}{2})\,\Delta x, y_{b}+j\,\Delta y\big)
\end{equation*}
	for $i = 1, \dots, N_x$ and for $j = 0, \dots, N_y$,
	and
\begin{equation*}
	(x_{i+1/2},y_j) = \big(x_{\ell}+i\,\Delta x, y_{b}+(j-\tfrac{1}{2})\,\Delta y\big)
\end{equation*}
	for $i = 0, \dots, N_x$ and for $j = 1, \dots, N_y$.
For time discretization, 
	the time interval $I=[0,T]$ is subdivided into $N_t$ subintervals of length $\Delta t := T/N_t$.
Unlike the space discretization, 
	we do not staggered variables in time. 
The components of the magnetic field are then approximated at the edges of the cell,
 	that is 
	$$H_x(x_i,y_{j+1/2},t_n) \approx H_{x,i,j+1/2}^n$$ and 
	$$H_y(x_{i+1/2},y_j,t_n) \approx H_{y,i+1/2,j}^n,$$
	and the $z$-component of the electric field is approximated at the center of the cell 
	$$E_z(x_i,y_j,t_n) \approx E_{z,i,j}^n.$$
	
The spatial derivatives are computed using either the second or fourth order centered approximation. 
For example, 
	the fourth-order centered approximation of $\partial_x H_y(x_i,y_j,t_n)$ is given by
\begin{equation} \label{eq:fourthOrderDerivative}
	\frac{H_{y,i-3/2,j}^n -27\,H_{y,i-1/2,j}^n+27\,H_{y,i+1/2,j}^n-H_{y,i+3/2,j}^n}{24\,\Delta x}.
\end{equation}

For time discretization, 
	we use the fourth-order Runge-Kutta (RK4) method,
	which is given by 
\begin{equation} \label{eq:RK4Method}
	\mathbold{U}^{n+1} = \mathbold{U}^n + \frac{1}{6}\,(\mathbold{k}_1 + 2\,\mathbold{k}_2 + 2\,\mathbold{k}_3 + \mathbold{k}_4),
\end{equation}
	with $\mathbold{U}^n = [H_x^n, H_y^n, E_z^n]^T$,
\begin{equation*}
\begin{aligned}
	\mathbold{k}_1 =&\,\, \Delta t \, \mathbold{G}(t_n,\mathbold{U}^n), \\
	\mathbold{k}_2 =&\,\, \Delta t \, \mathbold{G}(t_n+\tfrac{\Delta t}{2},\mathbold{U}^n+\tfrac{\mathbold{k}_1}{2}), \\
	\mathbold{k}_3 =&\,\, \Delta t \, \mathbold{G}(t_n+\tfrac{\Delta t}{2},\mathbold{U}^n+\tfrac{\mathbold{k}_2}{2}), \\
	\mathbold{k}_4 =&\,\, \Delta t \, \mathbold{G}(t_n+\Delta t,\mathbold{U}^n+\mathbold{k}_3), 
\end{aligned}
\end{equation*}
	and 
\begin{equation} \label{eq:GRK4}
	\mathbold{G}(t_n,\mathbold{U}^n) = 
	\begin{bmatrix}
	\tfrac{1}{\mu}\,(f_{1_x}^n - \partial_{y_h}E_z^n) \\ 
	\tfrac{1}{\mu}\,(f_{1_y}^n + \partial_{x_h}E_z^n) \\
	-\sigma\,E_z^n + f_2^n + \partial_{x_h} H_y^n - \partial_{y_h} H_x^n
	\end{bmatrix},
\end{equation}
	where the subscript $h$ in spatial derivatives denotes a given finite difference approximation of them in $\Omega$.
Let us now consider a FD approximation of spatial derivatives for which we apply correction functions, that is
	$D_{H_x}$, 
	$D_{H_y}$ and $D_{E_z}$.
It has been shown that a direct interpolation of approximations of correction functions at times $t_n$, $t_{n+1/2}$ and $t_{n+1}$,
	which are needed for different stages of the RK4 method,
	results in a suboptimal second-order accurate approximation in time. 
As proposed in \cite{Abraham2018},
	we need to slightly modify an approximation of a correction function to regain a full fourth-order approximation in time.
Based on Taylor expansions, 
	the modified approximations of correction functions at each stage are 
\begin{equation*}
\begin{aligned}
	1 \text{st stage} :&\,\, \qquad \hat{\mathbold{D}}_1^n = \mathbold{D}^n, \\
	2 \text{nd stage} :&\,\, \qquad \hat{\mathbold{D}}_2^{n} \approx \mathbold{D}^n + \tfrac{\Delta t}{2}\,\partial_t \mathbold{D}^n, \\
	3 \text{rd stage} :&\,\, \qquad \hat{\mathbold{D}}_3^{n} \approx \mathbold{D}^n + \tfrac{\Delta t}{2}\,\partial_t \mathbold{D}^n  + \tfrac{\Delta t^2}{4}\,\partial_t^2 \mathbold{D}^n,\\
	4 \text{th stage} :&\,\, \qquad \hat{\mathbold{D}}_4^{n}\approx \mathbold{D}^n + \Delta t\,\partial_t \mathbold{D}^n  + \tfrac{\Delta t^2}{2}\,\partial_t^2 \mathbold{D}^n + \tfrac{\Delta t^3}{4}\,\partial_t^3 \mathbold{D}^n,
	\end{aligned}
	\end{equation*}
	where $\mathbold{D}^n = [D_{H_x}^n, D_{H_y}^n, D_{E_z}^n]^T$. 
Time derivatives of a correction function can be computed directly using their polynomial approximations 
	coming from the minimization problem \cref{eq:minPblm}.
\begin{remark} 
It is worth mentioning that correction functions can be seen as additional source terms. 
Hence, 
	the stability condition of an original FD scheme should remain the same when the 
	CFM is used if correction functions are bounded \cite{Abraham2018}.  
This observation has been corroborated by numerical experiments in \cite{Abraham2018} for the wave equation. 
In our case, 
	the assumption of bounded correction functions is reasonable because 
	the correction functions' system of PDEs for Maxwell's equations do not allow perturbations to growth (see \cref{sec:wellPosedCFM}).
\end{remark}

\subsection{Truncation Error Analysis}
\label{sec:TruncationErrAnalysis}
	
In this short subsection, 
	we study the impact of an approximation of a correction function on 
	a finite difference scheme. 
As shown in \cref{lem:errorCFM}, 
	the error associated with an approximation of a correction function 
	coming from the minimization problem \cref{eq:minPblm} can reduce the order of 
	an original finite difference scheme, 
	that is without correction.

\begin{lemma}   \label{lem:errorCFM}
Let us consider a domain $\Omega$ subdivided into two subdomains $\Omega^+$ and $\Omega^-$ 
	for which the interface $\Gamma$
	between subdomains allows the solution $A(x)$ to be discontinuous. 
	Assume that there is sufficiently smooth extensions of $A(x)$ in each subdomain, 
	namely $A^+(x)$ and $A^-(x)$.
Moreover, 
	assume that an approximation of the correction function $D$ is $p$-order accurate and
	the fourth-order centered FD scheme, 
	namely
\begin{equation} \label{eq:thmErrorFDScheme}
	\partial_x A_{i} =  \frac{A_{i-3/2} -27\,A_{i-1/2}+27\,A_{i+1/2}-A_{i+3/2}}{24\,\Delta x}.
\end{equation}
The order of the fourth-order centered FD scheme when a correction is 
	applied is $q=\min\{p-1,4\}$.	
\end{lemma}

\begin{proof}
Consider that the fourth-order centered FD scheme \cref{eq:thmErrorFDScheme} 
	involves approximations of $A$ that belongs to different subdomains. 
For simplicity and without loss of generality, 
	suppose that $x_i \in \Omega^+$ and only one node belongs to the domain 
	$\Omega^-$, 
	that is $x_{i+1/2} \in \Omega^-$ and $x_{i-3/2},x_{i-1/2},x_{i+3/2} \in \Omega^+$.
Hence, 
\begin{equation} \label{eq:thmErrorFDCFM}
\partial_x A_i^+ =  \frac{A_{i-3/2}^+ -27\,A_{i-1/2}^++27\,(A_{i+1/2}^-+D_{i+1/2})-A_{i+3/2}^+}{24\,\Delta x},
\end{equation}
	where $D_{i+1/2}$ is 
	an approximation of the correction function evaluated at $x_{i+1/2}$. 
Since the approximation of the correction function is $p$-order accurate,   
\begin{equation*}
	D(x_{i+1/2}) = D_{i+1/2} + \mathcal{O}(\Delta x^p).
\end{equation*} 
Using appropriate Taylor's expansions about $x_i$ of $A_{i+1/2}^-$ and $D(x_{i+1/2})$, 
	we find 
\begin{equation} \label{eq:thmErrorEq1}
	\begin{aligned}
	A_{i+1/2}^-+D_{i+1/2} =&\,\, A_{i+1/2}^- + D(x_{i+1/2}) + \mathcal{O}(\Delta x^p) \\
	=&\,\, \displaystyle\sum_{j=0}^\infty \frac{1}{2^j\,j!} \Big(\partial_x^{(j)}A^-(x_i)+\partial_x^{(j)}D(x_i)\Big)\,\Delta x^j + \mathcal{O}(\Delta x^p) \\
	 =&\,\, \displaystyle\sum_{j=0}^\infty \frac{1}{2^j\,j!} \partial_x^{(j)}A^+(x_i)\,\Delta x^j + \mathcal{O}(\Delta x^p).
	\end{aligned}
\end{equation}
Using \cref{eq:thmErrorEq1} and performing a standard Taylor's expansion of \cref{eq:thmErrorFDCFM} about $x_i$,
	we find 
\begin{equation*}
	\partial_xA_i^+ = \partial_x A^+(x_i) + \mathcal{O}(\Delta x^4 + \Delta x^{p-1}).
\end{equation*}
\phantom{a}
\end{proof}

\subsection{Discrete Divergence Constraint}
\label{sec:divergence}
In this subsection, 
	we discuss about the conservation of the discrete divergence of the finite difference scheme,
	presented in \cref{sec:numScheme},
	combined with the CFM. 
We first show that the standard FD scheme preserves the divergence of the initial data at the discrete level. 
Secondly, 
	we show that the discrete divergence is still conserved for the FD scheme when combined with the CFM 
	except for some nodes close to the interface.
	
A common second-order discrete approximation of the divergence of a 2-D vector field is computed using 
\begin{equation} \label{eq:secondOrderDiscreteDivergence}
\big(\nabla\cdot\mathbold{A}\big)_{i+1/2,j+1/2}^{n} := 
                \frac{A_{x,i+1,j+1/2}^{n} - A_{x,i,j+1/2}^{n}}{\upDelta x} +
                 \frac{A_{y,i+1/2,j+1}^{n} - A_{y,i+1/2,j}^{n}}{\upDelta y},
\end{equation}
	where $A_x(x,y,t)$ and $A_y(x,y,t)$ \cite{Toth2000}. 
We also introduce the centered fourth-order discrete approximation of the divergence, 
	given by
\begin{equation} \label{eq:fourthOrderDiscreteDivergence}
	\begin{aligned}
	\big(\tilde{\nabla}\cdot\mathbold{A}\big)_{i+1/2,j+1/2}^{n} :=&\,\,
                \tfrac{A_{x,i-1,j+1/2}^n-27\,A_{x,i,j+1/2}^n+27\,A_{x,i+1,j+1/2}^n-A_{x,i+2,j+1/2}^{n}}{24\,\upDelta x} \\
                + & 
                \tfrac{A_{y,i+1/2,j-1}^n-27\,A_{y,i+1/2,j}^n+27\,A_{y,i+1/2,j+1}^n-A_{y,i+1/2,j+2}^{n}}{24\,\upDelta y}\,,
         \end{aligned}
\end{equation}
	which is better suited for the fourth-order centered scheme.

For the TM$_{\text{z}}$ mode,
	we remark that the $z$-component of the electric field $E_z(x,y,t)$ is at divergence-free. 
We then focus on the magnetic field. 
The following lemma shows that the standard staggered finite difference scheme combined with the RK4 time-stepping method
	 preserves the discrete divergence of the initial data at all later times.

\begin{lemma}   \label{lem:discreteDivergenceFree}
Assume that source terms satisfy 
$$
      \big(\tilde{\nabla}\cdot\mathbold{f}_1\big)_{i+1/2,j+1/2}^{n} = 0,
$$
for all $i,j$ and all $n \geq 0$.
The magnetic field, 
	computed with the standard fourth-order staggered FD scheme combined with the RK4 method, 
	is such that
$$
\big(\tilde{\nabla}\cdot\mathbold{H}\big)_{i+1/2,j+1/2}^{n+1} =\big(\tilde{\nabla}\cdot\mathbold{H}\big)_{i+1/2,j+1/2}^{0}\,,
$$
for all $i,j$ and all $n\geq0$.
\end{lemma}

\begin{proof}
The following demonstration is similar to the proof given in \cite{Toth2000}.
For a given time $t_n$, 
	let us consider the two first components of \cref{eq:GRK4}, 
	that is 
$$\mathbold{G}_H(t_n,E_z^n) = 
	\frac{1}{\mu}\,\begin{bmatrix} 
	f_{1_x}^n - \partial_{y_h}E_z^n \\ 
	f_{1_y}^n + \partial_{x_h}E_z^n
	\end{bmatrix},$$
where $\partial_{y_h} \cdot$ and $\partial_{x_h} \cdot$ denote the centered fourth-order approximation 
	\eqref{eq:fourthOrderDerivative}. 
Applying the discrete divergence operator to $\mathbold{G}_H(t_n,E_z^n)$ leads to 
\begin{equation*}
\big(\tilde{\nabla}\cdot\mathbold{G}_H\big)_{i+1/2,j+1/2}^{n} =  \big(\tilde{\nabla}\cdot\mathbold{f}_1\big)_{i+1/2,j+1/2}^{n} + \big(\tilde{\nabla}\cdot\mathbold{A}\big)_{i+1/2,j+1/2}^{n} \,,
\end{equation*}
	where 
\begin{equation*}
\begin{aligned}
A_{x,i,j+1/2}^{n} =&\,\, -\frac{E_{z,i,j-1}^n-27\,E_{z,i,j}^n+27\,E_{z,i,j+1}^n-E_{z,i,j+2}^{n}}{24\,\upDelta y}\,, \\
A_{y,i+1/2,j}^{n} =&\,\, \frac{E_{z,i-1,j}^n-27\,E_{z,i,j}^n+27\,E_{z,i+1,j}^n-E_{z,i+2,j}^{n}}{24\,\upDelta x} \,,
\end{aligned}
\end{equation*}
	which is a fourth-order approximation of the curl of the electric field at cell edges. 
We can easily verify that 
\begin{equation*}
	\big(\tilde{\nabla}\cdot\mathbold{A}\big)_{i+1/2,j+1/2}^{n} = 0\,, \quad \forall i,j,n.
\end{equation*}
Using $\big(\tilde{\nabla}\cdot\mathbold{f}_1\big)_{i+1/2,j+1/2}^{n} = 0\,$, 
	we obtain 
	$$\big(\tilde{\nabla}\cdot\mathbold{G}_H\big)_{i+1/2,j+1/2}^{n} = 0\,,$$
	for all $i,j$ and all $n\geq0$. 
Applying the discrete divergence operator to \eqref{eq:RK4Method}, 
	we find 
	$\big(\nabla\cdot\mathbold{H}\big)_{i+1/2,j+1/2}^{n+1} = \big(\nabla\cdot\mathbold{H}\big)_{i+1/2,j+1/2}^{n}\,$.
Hence, we obtain the desired result.
\end{proof}

Due to possible discontinuities at the interface $\Gamma$, 
	we need to investigate the discrete divergence for nodes that are close to $\Gamma$. 
We distinguish two cases that are illustrated in \cref{fig:thmCases}. 
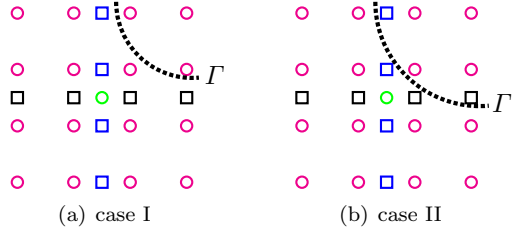
\begin{figure}[htbp]
 	\centering
  \subfigure[case I]{
 	\tdplotsetmaincoords{75}{105}
	\tikzset{external/export next=false}
  \begin{tikzpicture}[scale=0.75]
	\draw[-latex,thick,green] (2.5,2.5) circle [radius=0.1];
	\draw[-latex,thick,blue]  (2.4,3.9) rectangle (2.6,4.1);
	\draw[-latex,thick,blue]  (2.4,2.9) rectangle (2.6,3.1);
	\draw[-latex,thick,blue]  (2.4,1.9) rectangle (2.6,2.1);
	\draw[-latex,thick,blue]  (2.4,0.9) rectangle (2.6,1.1);
	\draw[-latex,thick,black]  (0.9,2.4) rectangle (1.1,2.6);
	\draw[-latex,thick,black]  (1.9,2.4) rectangle (2.1,2.6);
	\draw[-latex,thick,black]  (2.9,2.4) rectangle (3.1,2.6);
	\draw[-latex,thick,black]  (3.9,2.4) rectangle (4.1,2.6);
	\draw[-latex,thick,magenta] (1,4) circle [radius=0.1];
	\draw[-latex,thick,magenta] (1,3) circle [radius=0.1];
	\draw[-latex,thick,magenta] (1,2) circle [radius=0.1];
	\draw[-latex,thick,magenta] (1,1) circle [radius=0.1];
	\draw[-latex,thick,magenta] (2,4) circle [radius=0.1];
	\draw[-latex,thick,magenta] (2,3) circle [radius=0.1];
	\draw[-latex,thick,magenta] (2,2) circle [radius=0.1];
	\draw[-latex,thick,magenta] (2,1) circle [radius=0.1];
	\draw[-latex,thick,magenta] (3,4) circle [radius=0.1];
	\draw[-latex,thick,magenta] (3,3) circle [radius=0.1];
	\draw[-latex,thick,magenta] (3,2) circle [radius=0.1];
	\draw[-latex,thick,magenta] (3,1) circle [radius=0.1];
	\draw[-latex,thick,magenta] (4,4) circle [radius=0.1];
	\draw[-latex,thick,magenta] (4,3) circle [radius=0.1];
	\draw[-latex,thick,magenta] (4,2) circle [radius=0.1];
	\draw[-latex,thick,magenta] (4,1) circle [radius=0.1];	
	\draw[ultra thick,densely dotted,black] (2.75,4.2) arc (180:275:1.35);
  	\draw (4.5,2.9) node {$\Gamma$};
  \end{tikzpicture}
  	}
  \subfigure[case II]{  
 	\tdplotsetmaincoords{75}{105}
	\tikzset{external/export next=false}
  \begin{tikzpicture}[scale=0.75]
	\draw[-latex,thick,green] (2.5,2.5) circle [radius=0.1];
	\draw[-latex,thick,blue]  (2.4,3.9) rectangle (2.6,4.1);
	\draw[-latex,thick,blue]  (2.4,2.9) rectangle (2.6,3.1);
	\draw[-latex,thick,blue]  (2.4,1.9) rectangle (2.6,2.1);
	\draw[-latex,thick,blue]  (2.4,0.9) rectangle (2.6,1.1);
	\draw[-latex,thick,black]  (0.9,2.4) rectangle (1.1,2.6);
	\draw[-latex,thick,black]  (1.9,2.4) rectangle (2.1,2.6);
	\draw[-latex,thick,black]  (2.9,2.4) rectangle (3.1,2.6);
	\draw[-latex,thick,black]  (3.9,2.4) rectangle (4.1,2.6);
	\draw[-latex,thick,magenta] (1,4) circle [radius=0.1];
	\draw[-latex,thick,magenta] (1,3) circle [radius=0.1];
	\draw[-latex,thick,magenta] (1,2) circle [radius=0.1];
	\draw[-latex,thick,magenta] (1,1) circle [radius=0.1];
	\draw[-latex,thick,magenta] (2,4) circle [radius=0.1];
	\draw[-latex,thick,magenta] (2,3) circle [radius=0.1];
	\draw[-latex,thick,magenta] (2,2) circle [radius=0.1];
	\draw[-latex,thick,magenta] (2,1) circle [radius=0.1];
	\draw[-latex,thick,magenta] (3,4) circle [radius=0.1];
	\draw[-latex,thick,magenta] (3,3) circle [radius=0.1];
	\draw[-latex,thick,magenta] (3,2) circle [radius=0.1];
	\draw[-latex,thick,magenta] (3,1) circle [radius=0.1];
	\draw[-latex,thick,magenta] (4,4) circle [radius=0.1];
	\draw[-latex,thick,magenta] (4,3) circle [radius=0.1];
	\draw[-latex,thick,magenta] (4,2) circle [radius=0.1];
	\draw[-latex,thick,magenta] (4,1) circle [radius=0.1];	
	\draw[ultra thick,densely dotted,black] (2.3,4.2) arc (180:275:1.85);
  	\draw (4.55,2.4) node {$\Gamma$};	
  \end{tikzpicture}
  	}
  \caption{Illustration of the two cases for the computation of the centered fourth-order divergence of the magnetic field 
  		around the interface $\Gamma$ (dotted line). 
  		For legibility, we only show nodes involve in \cref{thm:discreteDivergenceFreeWithCFMPart1} and \cref{thm:discreteDivergenceFreeWithCFMPart2} for the computation of the discrete divergence of $\mathbold{H}$ at the node represented by {\color{green}$\circ$}. The components $H_x$, $H_y$ and $E_z$ are respectively represented by {\tiny$\square$}, {\color{blue}\tiny$\square$} and {\color{magenta}$\circ$}.}
\label{fig:thmCases}
\end{figure}
In the first case,
	we consider that the discrete divergence operator involves only components of the magnetic field that belong to the same 
	subdomain.
However, 
	there is no restriction on the electric field. 
In contrast, 
	the second case considers $H_x$ and $H_y$ that belong to different subdomains in the computation of the discrete 
	divergence operator.
In that situation, 
	discrete divergence operators \cref{eq:secondOrderDiscreteDivergence} and \cref{eq:fourthOrderDiscreteDivergence} 
	are not 
	well suited and need to be redefined. 
In the spirit of the CFM, 
	we propose a corrected discrete divergence operator that uses correction functions if it is necessary.
The corrected discrete divergence operator is denoted as either 
	$\big(\nabla^D\cdot\mathbold{A}\big)_{i+1/2,j+1/2}^{n}$ or $\big(\tilde{\nabla}^D\cdot\mathbold{A}\big)_{i+1/2,j+1/2}^{n}$ 
	for respectively the second and fourth order centered approximation. 
The following theorems analyze the discrete divergence of the approximation of $\mathbold{H}$ in both situations.

\begin{theorem}   \label{thm:discreteDivergenceFreeWithCFMPart1}
Under assumptions of \cref{lem:discreteDivergenceFree} and assuming that the approximation of the correction function 
	$\hat{D}_{E_z}$ 
	at each node is unique.
If the computation of $\big(\tilde{\nabla}\cdot\mathbold{H}\big)_{i+1/2,j+1/2}^{\boldsymbol{\circ},n+1}$,
	where the superscript $\boldsymbol{\circ}$ can be either $+$ or $-$ depending in which subdomain ($\Omega^+$ or $\Omega^-$)
	the node $(x_{i+1/2},y_{j+1/2})$ belongs,
	involves only approximations of the magnetic field in the same subdomain,
	then the approximation of $\mathbold{H}$, 
	computed with the fourth-order staggered FD scheme combined with the RK4 method and the CFM, 
	is such that
$$
\big(\tilde{\nabla}\cdot\mathbold{H}\big)_{i+1/2,j+1/2}^{\boldsymbol{\circ},n+1} =\big(\tilde{\nabla}\cdot\mathbold{H}\big)_{i+1/2,j+1/2}^{\boldsymbol{\circ},0}\,,
$$
for all $i,j$ and all $n\geq0$.
\end{theorem}

\begin{proof}
Let us consider that the discrete divergence operator \cref{eq:fourthOrderDiscreteDivergence} involves only approximations of $H_x$ and $H_y$ in the same subdomain than the node $(x_{i+1/2},y_{j+1/2})$. 
For simplicity and without loss of generality, 
	consider that the corner where the discrete divergence operator is computed belongs to $\Omega^+$.
Suppose that some approximations of the electric field in \cref{eq:GRK4} belong to $\Omega^-$.
Using the uniqueness of correction functions and repeating the same procedure as in \cref{lem:discreteDivergenceFree}, 
	but with correction functions, 
	that is 
	$$E_{z}^{+,n} \mapsto E_{z}^{-,n} + \hat{D}_{E_z}^n$$
	where it is needed, 
	we find the desired result. 
\end{proof}
	
\begin{theorem}   \label{thm:discreteDivergenceFreeWithCFMPart2}
Assume that correction functions, 
	namely $D_{H_x}$ and $D_{H_y}$, 
	and the magnetic field $\mathbold{H}$ 
	satisfy assumptions of \cref{lem:errorCFM}, 
	and a stability condition of the form 
\begin{equation*}
	\Delta t = \alpha\,\min\{\Delta x,\Delta y\},
\end{equation*} 
	where $\alpha$ is a positive constant.
The approximation of $\mathbold{H}$, 
	computed with the fourth-order staggered FD scheme combined with the RK4 method and the CFM, 
	is such that
$$
\big(\tilde{\nabla}^D\cdot\mathbold{H}\big)_{i+1/2,j+1/2}^{\boldsymbol{\circ},n} = \nabla\cdot\mathbold{H}(x_{i+1/2},y_{j+1/2},t_{n})+ \mathcal{O}(\Delta x^{r} + \Delta y^{r} +  \Delta t^{s})\,,
$$
	for all $i,j$ and all $n\geq0$, 
	where $r = \min\{p-2,3\}$, 
	$s = \min\{p-1,3\}$ 
	and the superscript $\boldsymbol{\circ}$ can be either $+$ or $-$ depending in which subdomain ($\Omega^+$ or $\Omega^-$)
	the node $(x_{i+1/2},y_{j+1/2})$ belongs. 
\end{theorem}

\begin{proof}	
Consider that the corrected discrete divergence operator involves approximations 
	of the components of $\mathbold{H}$ that belong to different subdomains. 
For simplicity and without loss of generality, 
	suppose that the corner, 
	where the corrected discrete divergence operator is computed, 
	belongs to $\Omega^+$.
For a given time $t_n$, 
	assume that we need a correction on $H_{x,i+2,j+1/2}^{+,n}$ and $H_{y,i+1/2,j+2}^{+,n}$ in the computation of 
	$\big(\tilde{\nabla}^D\cdot\mathbold{H}\big)_{i+1/2,j+1/2}^{+,n}$, 
	that is 
\begin{equation*}
	\begin{aligned}
		H_{x,i+2,j+1/2}^{+,n} \approx&\,\, H_{x,i+2,j+1/2}^{-,n}  + D_{H_x,i+2,j+1/2}^{n} \,,\\
		H_{y,i+1/2,j+2}^{+,n} \approx&\,\, H_{y,i+1/2,j+2}^{-,n} + D_{H_y,i+1/2,j+2}^{n}\,.
	\end{aligned}
\end{equation*}
Let us compute the Taylor expansion associated with $H_{x,i+2,j+1/2}^{+,n}$.
By \cref{lem:errorCFM}, 
	using the fourth-order staggered FD scheme combined with the RK4 method and a $p$-order accurate approximation of correction functions leads to 
\begin{equation*}
	H_{x,i+2,j+1/2}^{+,n} \approx H_{x}^-(x_{i+2},y_{j+1/2},t_n)  + \mathcal{O}(\Delta x^q + \Delta y^q +\Delta t^4)+ D_{H_x,i+2,j+1/2}^{n},
\end{equation*}
	where $q=\min\{p-1,4\}$.
Hence,
\begin{equation}	\label{eq:errorHx}
	\begin{aligned}
		H_{x,i+2,j+1/2}^{+,n} =&\,\, H_{x}^-(x_{i+2},y_{j+1/2},t_n) + D_{H_x}(x_{i+2},y_{j+1/2},t_n) \\
		&+ \mathcal{O}(\Delta x^q + \Delta y^q +\Delta t^4+\Delta t^p)\\
						=&\,\, H_{x}^+(x_{i+2},y_{j+1/2},t_n) + \mathcal{O}(\Delta x^q + \Delta y^q +\Delta t^k)\,,
	\end{aligned}
\end{equation}
	where $k = \min\{p,4\}$.
Using a similar procedure, 
	we also have 
\begin{equation} \label{eq:errorHy}
	H_{y,i+1/2,j+2}^{+,n}  = H_{y}^+(x_{i+1/2},y_{j+2},t_n) + \mathcal{O}(\Delta x^q + \Delta y^q +\Delta t^k)\,.
\end{equation}
Substituting \cref{eq:errorHx} and \cref{eq:errorHy} in $\big(\tilde{\nabla}^D\cdot\mathbold{H}\big)_{i+1/2,j+1/2}^{+,n}$, 	and using appropriate Taylor expansions and the stability condition, 
	we find the desired result.
\phantom{a}	
\end{proof}
\begin{remark} 
Similar statements can be obtained with the second-order staggered FD scheme.
	However, 
	we need to consider the second-order discrete divergence 
	operator \cref{eq:secondOrderDiscreteDivergence}. 
\end{remark}

\section{Numerical Examples} \label{sec:numEx}
In the following, 
	we perform convergence analysis of the proposed numerical schemes for problems 
	with a manufactured solution with various interfaces.
We use a fourth-order approximation of the correction functions with the RK4 method and either the second-order or  fourth-order staggered FD scheme.
The domain is $\Omega = [ 0,1]\times[0,1]$ and the time interval is $I = [0,0.5]$.
The physical parameters are $\mu = \sigma = \epsilon = 1$ in all $\Omega$. 
Periodic boundary conditions are imposed on all $\partial\Omega$ for all numerical experiments.
We also choose the mesh grid size to be 
$h \in \big\{ \tfrac{1}{20}, \tfrac{1}{28}, \tfrac{1}{40}, \tfrac{1}{52}, \tfrac{1}{72}, \tfrac{1}{96}, \tfrac{1}{132}, \tfrac{1}{180},\tfrac{1}{244},\tfrac{1}{336} \big\}$ and $\Delta x = \Delta y = h$.
The time-step size is chosen to satisfy a stability condition and to reach exactly the final time, 
	that is $\Delta t =  \tfrac{h}{2}$. 
\cref{fig:interfaceGeo} illustrates different geometries of the interface that are studied in this work.
We have $\phi(x,y)\geq0$ in $\Omega^+$, 
	$\phi(x,y) < 0$ in $\Omega^-$ and $\phi(x,y) = 0$ on $\Gamma$, where $\phi(x,y)$ is the level-set function.
 \begin{figure} 
 \centering
  \subfigure[circular]{ \label{fig:circleInterfaceGeo}
		\setlength\figureheight{0.2\linewidth} 
		\setlength\figurewidth{0.2\linewidth} 
		\tikzset{external/export next=false}
%
%
\begin{tikzpicture}

\begin{axis}[%
width=\figurewidth,
height=\figureheight,
at={(0\figurewidth,0\figureheight)},
scale only axis,
xmin=0,
xmax=1,
xminorticks=true,
xlabel={\scriptsize$x$},
ymin=0,
ymax=1,
yminorticks=true,
ylabel={\scriptsize $y$},
axis background/.style={fill=white},
legend style={at={(0.01,0.99)},anchor=north west,legend cell align=left,align=left,draw=white!15!black,draw=none,fill=none},
legend style={font=\scriptsize},
ylabel style={yshift=-5pt},xlabel style={yshift=2.5pt},tick label style={font=\tiny} 
]
\addplot [color=black,line width=1pt,solid]
  table[row sep=crcr]{%
0.75	0.5\\
0.749496669117971	0.515855979914141\\
0.747988703207699	0.531648113393437\\
0.745482174315677	0.547312811090103\\
0.741987175349089	0.56278699679527\\
0.737517779435236	0.578008361424622\\
0.732091983254018	0.592915613915082\\
0.725731634571655	0.607448728022293\\
0.718462344267446	0.621549184025117\\
0.710313383207795	0.635160204363899\\
0.701317564382765	0.64822698226366\\
0.691511110779744	0.660696902421635\\
0.680933509526268	0.672519752870528\\
0.669627352889283	0.683647927164383\\
0.657638166771131	0.694036616072939\\
0.6450142273928	0.703643988012584\\
0.631806366902626	0.712431357487379\\
0.618067768693171	0.720363340861895\\
0.603853753250472	0.72740799883863\\
0.589221555397968	0.733536965066277\\
0.574230093832069	0.738725560361018\\
0.558939733877357	0.742952892080885\\
0.543412044416733	0.746201938253052\\
0.527709549975253	0.748459616115314\\
0.511895478955936	0.749716834795752\\
0.496033509041298	0.749968531918469\\
0.480187510785803	0.749213693987986\\
0.464421290431679	0.747455360470233\\
0.448798332983702	0.744700611553695\\
0.433381546577491	0.740960539639986\\
0.418233009170645	0.736250204678667\\
0.403413718576718	0.730588573526145\\
0.388983346848556	0.723998443572834\\
0.375	0.71650635094611\\
0.361519984033472	0.708142463658693\\
0.348597578215583	0.698940460132708\\
0.336284816513679	0.688937393588565\\
0.32463127807342	0.678173542844716\\
0.313683887581061	0.666692250129073\\
0.303486726314303	0.654539746555151\\
0.294080854642542	0.641764965965693\\
0.285504146691256	0.628419347893352\\
0.277791137836269	0.614556630431853\\
0.270972885641983	0.600232633851653\\
0.265076844803523	0.585505035831417\\
0.260126756596376	0.570433139210357\\
0.256142553278648	0.555077633196635\\
0.253140277830901	0.539500348993337\\
0.251132019356729	0.523764010826046\\
0.250125864404204	0.507931983374517\\
0.250125864404204	0.492068016625483\\
0.251132019356729	0.476235989173954\\
0.253140277830901	0.460499651006662\\
0.256142553278648	0.444922366803365\\
0.260126756596376	0.429566860789643\\
0.265076844803523	0.414494964168583\\
0.270972885641983	0.399767366148347\\
0.277791137836269	0.385443369568148\\
0.285504146691256	0.371580652106648\\
0.294080854642542	0.358235034034307\\
0.303486726314303	0.345460253444849\\
0.313683887581061	0.333307749870927\\
0.32463127807342	0.321826457155284\\
0.336284816513679	0.311062606411435\\
0.348597578215583	0.301059539867292\\
0.361519984033472	0.291857536341307\\
0.375	0.28349364905389\\
0.388983346848557	0.276001556427166\\
0.403413718576718	0.269411426473855\\
0.418233009170645	0.263749795321333\\
0.433381546577491	0.259039460360015\\
0.448798332983702	0.255299388446305\\
0.464421290431679	0.252544639529767\\
0.480187510785803	0.250786306012014\\
0.496033509041298	0.250031468081531\\
0.511895478955936	0.250283165204248\\
0.527709549975253	0.251540383884686\\
0.543412044416733	0.253798061746948\\
0.558939733877357	0.257047107919115\\
0.574230093832069	0.261274439638982\\
0.589221555397968	0.266463034933723\\
0.603853753250472	0.27259200116137\\
0.61806776869317	0.279636659138104\\
0.631806366902626	0.287568642512621\\
0.645014227392799	0.296356011987416\\
0.657638166771131	0.305963383927061\\
0.669627352889283	0.316352072835617\\
0.680933509526267	0.327480247129472\\
0.691511110779744	0.339303097578365\\
0.701317564382765	0.35177301773634\\
0.710313383207795	0.364839795636101\\
0.718462344267446	0.378450815974883\\
0.725731634571655	0.392551271977707\\
0.732091983254018	0.407084386084918\\
0.737517779435236	0.421991638575378\\
0.741987175349089	0.43721300320473\\
0.745482174315677	0.452687188909897\\
0.747988703207699	0.468351886606563\\
0.749496669117971	0.484144020085859\\
0.75	0.5\\
};
\end{axis}
\end{tikzpicture}%
		} 
  \subfigure[5-star]{\label{fig:starInterfaceGeo}
		\setlength\figureheight{0.2\linewidth} 
		\setlength\figurewidth{0.2\linewidth} 
		\tikzset{external/export next=false}
%
%
\begin{tikzpicture}

\begin{axis}[%
width=\figurewidth,
height=\figureheight,
at={(0\figurewidth,0\figureheight)},
scale only axis,
xmin=0,
xmax=1,
xminorticks=true,
xlabel={\scriptsize$x$},
ymin=0,
ymax=1,
yminorticks=true,
ylabel={\scriptsize $y$},
axis background/.style={fill=white},
legend style={at={(0.01,0.99)},anchor=north west,legend cell align=left,align=left,draw=white!15!black,draw=none,fill=none},
legend style={font=\scriptsize},
ylabel style={yshift=-5pt},xlabel style={yshift=2.5pt},tick label style={font=\tiny} 
]
\addplot [color=black,line width=1pt,solid]
  table[row sep=crcr]{%
0.75	0.5\\
0.765066930189006	0.516845499123648\\
0.777395596897263	0.535400996867555\\
0.785474949486592	0.555020786717682\\
0.788201994577402	0.574778085589939\\
0.785015355939205	0.593608069890764\\
0.77595740609245	0.610476680159308\\
0.76165735877021	0.624549447537036\\
0.743238589707334	0.6353343167743\\
0.722163808645616	0.64277601035563\\
0.700040086323819	0.647286395096484\\
0.678410979268275	0.649704586901987\\
0.658564374561849	0.651190825763033\\
0.641382028781861	0.653067981559738\\
0.627250504584707	0.656632481897858\\
0.616044232300894	0.662961322323235\\
0.607181039780345	0.672742897877369\\
0.599740192449566	0.686156495288\\
0.592624237639838	0.702818789561018\\
0.584739992587695	0.721806497324522\\
0.57517168653353	0.74175374251308\\
0.563320871123122	0.761012185776997\\
0.54899298926918	0.777853049331026\\
0.532418651828539	0.790684106941049\\
0.514207511767961	0.798252376567682\\
0.495242705950108	0.799804996894593\\
0.47653268388508	0.795186358147882\\
0.459043571509212	0.784858217139645\\
0.443538105232274	0.769840045925559\\
0.430446197184106	0.751577768613709\\
0.41978749849487	0.731758802741356\\
0.411158596481512	0.712098573477648\\
0.403787839421047	0.694127399905181\\
0.396650635094611	0.67900635094611\\
0.388628899709452	0.667396393140633\\
0.378691488291971	0.659397523849207\\
0.366068930318577	0.654564693923155\\
0.350396159890447	0.651996581389543\\
0.331801144729265	0.650483204539873\\
0.310924802629974	0.64869037868741\\
0.28886729267282	0.645354240940172\\
0.267066301407991	0.639458144361026\\
0.247122803447367	0.63036725568928\\
0.230597541587669	0.617902712298803\\
0.218806015883607	0.602346240273091\\
0.212640420665489	0.584376385492237\\
0.212443602463266	0.564947476486013\\
0.217953180803277	0.545130682692334\\
0.228324437528831	0.525941870830668\\
0.24222977195585	0.508182636263719\\
0.258021956852558	0.492318669514685\\
0.273939601184627	0.478413849178577\\
0.288327374858526	0.466129984705659\\
0.299841504094031	0.454792210092742\\
0.307613092527262	0.443510107071522\\
0.311347673723439	0.431336168610256\\
0.311348229696296	0.417437444595497\\
0.308459472225172	0.401253994825574\\
0.303941991974521	0.382619448574323\\
0.299294416612264	0.361824309008787\\
0.296048649998632	0.339610885577108\\
0.295566630432858	0.317098704281728\\
0.298866396256393	0.295649495700111\\
0.30650070270878	0.276689906746026\\
0.318503668139195	0.261516603583791\\
0.334411068357493	0.251111465823247\\
0.353349364905389	0.24599364905389\\
0.374178854276066	0.246130512759513\\
0.395668840671924	0.250921426425358\\
0.416678519846419	0.259258393384021\\
0.436316895970877	0.269656689333738\\
0.454058560735131	0.280438822818169\\
0.469799009354145	0.289947496199178\\
0.483842337686526	0.296758970171911\\
0.496824312132489	0.299867933057655\\
0.50958344614391	0.298818706976178\\
0.523000448121967	0.293764874710422\\
0.537831099564285	0.285449172824922\\
0.554558596631592	0.275106401615227\\
0.573288501130608	0.264302621791043\\
0.593703118208241	0.254732567191969\\
0.615083268861105	0.248002791883759\\
0.636395344936775	0.245429813564209\\
0.656431694024907	0.247880182902612\\
0.673984222484705	0.255673346298067\\
0.688025828957554	0.268559249751979\\
0.697872676996705	0.285772127230971\\
0.703302644490686	0.306151320021977\\
0.704611242291214	0.328310782058717\\
0.702595042441711	0.350832430569164\\
0.698462957769975	0.372455601627831\\
0.693686098827558	0.392235948724066\\
0.689805910373101	0.40965199149245\\
0.688226560415586	0.424645452329144\\
0.690020202931268	0.437591347041521\\
0.695772356120776	0.449204091999399\\
0.705489399144761	0.460395164537477\\
0.718581809518134	0.47210477008068\\
0.733926408046936	0.485133539295366\\
0.75	0.5\\
};
\end{axis}
\end{tikzpicture}%
		} 
  \subfigure[3-star]{\label{fig:triStarInterfaceGeo}
		\setlength\figureheight{0.2\linewidth} 
		\setlength\figurewidth{0.2\linewidth} 
		\tikzset{external/export next=false}
%
%
\begin{tikzpicture}

\begin{axis}[%
width=\figurewidth,
height=\figureheight,
at={(0\figurewidth,0\figureheight)},
scale only axis,
xmin=0,
xmax=1,
xminorticks=true,
xlabel={\scriptsize$x$},
ymin=0,
ymax=1,
yminorticks=true,
ylabel={\scriptsize $y$},
axis background/.style={fill=white},
legend style={at={(0.01,0.99)},anchor=north west,legend cell align=left,align=left,draw=white!15!black,draw=none,fill=none},
legend style={font=\scriptsize},
ylabel style={yshift=-5pt},xlabel style={yshift=2.5pt},tick label style={font=\tiny} 
]
\addplot [color=black,line width=1pt,solid]
  table[row sep=crcr]{%
0.75	0.55\\
0.777827202174583	0.567656438271724\\
0.803289557453815	0.58870556271793\\
0.82511278435136	0.612660353208428\\
0.842181337802838	0.638783790004716\\
0.85360354222264	0.666134602589831\\
0.858762959553902	0.693627023086182\\
0.857352802859646	0.720100678158521\\
0.84939128458551	0.744396090041338\\
0.835217000935954	0.765430885384937\\
0.815464722280778	0.782271754020518\\
0.791023203005085	0.794197462305598\\
0.762977763212807	0.80074879075029\\
0.732541356387677	0.801762097050854\\
0.700978560081743	0.797384250275455\\
0.669527364829854	0.788067872869486\\
0.639323761934663	0.774547088077271\\
0.611333935732935	0.757795220498922\\
0.586298363185971	0.73896705666323\\
0.564691345602096	0.719329266348019\\
0.546698494101894	0.700183350134391\\
0.532213526603137	0.682785965213926\\
0.520854484439667	0.668271658460186\\
0.511998221527555	0.657582891727411\\
0.504830838876719	0.651411788308068\\
0.498410708003185	0.650157290504339\\
0.491739914515188	0.653900443506685\\
0.483839409977741	0.662399372486807\\
0.473822920977521	0.675104271379505\\
0.460964744468054	0.691191453039411\\
0.444756952828365	0.709614302409751\\
0.424952215298844	0.729167904235782\\
0.40158937070074	0.748563253064262\\
0.375	0.76650635094611\\
0.345795474829044	0.781777195809682\\
0.314835220691593	0.793303680098579\\
0.283178150135716	0.800225753739155\\
0.252020313510234	0.801945836235844\\
0.222622712787477	0.798162349138054\\
0.196233869560482	0.788884325367478\\
0.174012090083937	0.774426266310765\\
0.156952405335079	0.755383683291261\\
0.145822880028894	0.732590995898499\\
0.141114399296272	0.707064586482698\\
0.143007192601117	0.679934755740321\\
0.151356295635617	0.652371028197555\\
0.165696950050079	0.625505673536793\\
0.185269674724417	0.600360413509079\\
0.209063491855173	0.577781068144543\\
0.23587463641193	0.558384373147599\\
0.264377092396477	0.542520406398565\\
0.293200546858284	0.530253046492452\\
0.321010880937386	0.521359715522404\\
0.346588156507218	0.515350407143522\\
0.368897217557134	0.51150474977684\\
0.387146497005929	0.508924684077486\\
0.400831371987693	0.506599318779391\\
0.409759395643644	0.503477735034794\\
0.414055888047432	0.498544987504558\\
0.414149619201147	0.490896334379379\\
0.410739583068125	0.479804832257175\\
0.404745062374646	0.464777848879909\\
0.397242242636605	0.445598750546413\\
0.389391482891641	0.422350966562026\\
0.382359935739573	0.395422759833163\\
0.377244493237901	0.365492268492297\\
0.375	0.33349364905389\\
0.376377322996373	0.300566365918594\\
0.381875221854592	0.267990757183491\\
0.391709065512924	0.237113893052417\\
0.405798348686928	0.20927037375944\\
0.423773744989883	0.185703048272115\\
0.445003170885616	0.167488651546341\\
0.468635107056417	0.155473055530714\\
0.493656310079411	0.150220226667401\\
0.518960119035153	0.151978118716564\\
0.54342087842295	0.160663659496784\\
0.565969604393798	0.175867781954081\\
0.585665941151576	0.196880181052155\\
0.601761693562244	0.222732229412354\\
0.61375176519384	0.252255336215466\\
0.621409143314972	0.28415105898597\\
0.624801601653406	0.317068538775131\\
0.624288971870588	0.349684373102513\\
0.620501089955745	0.380779896844318\\
0.614297773460518	0.409311018129577\\
0.606713349390889	0.434466242722087\\
0.598889255839728	0.455709285009234\\
0.591999018554404	0.472803657462328\\
0.587170406484751	0.485817789493198\\
0.585409765479637	0.495110476657139\\
0.587533403949382	0.501297721991104\\
0.594110466283665	0.505203222113936\\
0.605421006954134	0.507795795256018\\
0.621432016647833	0.510117879740587\\
0.64179301289534	0.513209796414176\\
0.665851564279994	0.518034731028223\\
0.692687848961582	0.525409335931056\\
0.721166136061359	0.535944478443442\\
0.75	0.55\\
};
\end{axis}
\end{tikzpicture}%
		} 
  \caption{Different geometries of the interface.}
  \label{fig:interfaceGeo}
\end{figure}
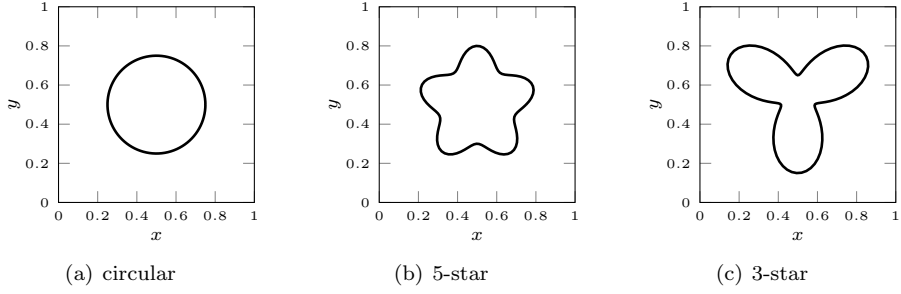

It is worth to mention that the proposed numerical method can be applied directly to problems that involve 
	interface conditions and a perfect electric conductor (PEC) material 
	for which the surface current and charge density are known explicitly.
Unfortunately, 
	to our knowledge, 
	there is no analytical solution for arbitrary geometries of the interface.
We therefore use manufactured solutions to verify the proposed numerical method.
The manufactured solutions that are used satisfy the divergence-free property in each subdomain, 
	but not in the entire domain. 	
However,  
	it is the interface condition \cref{eq:normalHInterf} that allows 
	the divergence-free property of the magnetic field to hold in the whole domain, 
	which can be imposed by the proposed numerical method. 

\FloatBarrier
\subsection{Circular interface} \label{sec:circularInterfacePblm}
The level set function 
\begin{equation*}
\phi(x,y) = (x-x_0)^2 +(y-y_0)^2 - r_0^2,
\end{equation*}
	where $x_0 = y_0 = 0.5$ and $r_0 = 0.25$, 
	is used to describe the interface.
The manufactured solutions are :
\begin{equation*}
	\begin{aligned}
		H_x^+ =&\,\, \sin(2\,\pi\,x)\,\sin(2\,\pi\,y)\,\sin(2\,\pi\,t), \\
		H_y^+ =&\,\, \cos(2\,\pi\,x)\,\cos(2\,\pi\,y)\,\sin(2\,\pi\,t), \\
		E_z^+ =&\,\, \sin(2\,\pi\,x)\,\cos(2\,\pi\,y)\,\cos(2\,\pi\,t)
	\end{aligned}
\end{equation*}
	in $\Omega^+$, 
	and 
\begin{equation*}
	\begin{aligned}
		H_x^- =&\,\, -2\,\sin(2\,\pi\,x)\,\sin(2\,\pi\,y)\,\sin(2\,\pi\,t) + 5, \\
		H_y^- =&\,\, -2\,\cos(2\,\pi\,x)\,\cos(2\,\pi\,y)\,\sin(2\,\pi\,t) + 3, \\
		E_z^- =&\,\, -2\,\sin(2\,\pi\,x)\,\cos(2\,\pi\,y)\,\cos(2\,\pi\,t) +2
	\end{aligned}
\end{equation*}
	in $\Omega^-$.
The associated source terms are $\mathbold{f}_1^+ = \mathbold{f}_1^- = 0$ and 
\begin{equation*}
	\begin{aligned}
		f_2^+ =&\,\, (2\,\pi\,\sin(2\,\pi\,t) + \cos(2\,\pi\,t) )\, \sin(2\,\pi\,x)\,\cos(2\,\pi\,y), \\
		f_2^-  =&\,\, -(4\,\pi\,\sin(2\,\pi\,t) + 2\,\cos(2\,\pi\,t) )\, \sin(2\,\pi\,x)\,\cos(2\,\pi\,y) + 2.
	\end{aligned}
\end{equation*}
\cref{fig:convPlotCircularInterface1} and \cref{fig:convPlotCircularInterface2} illustrate 
	convergence plots for respectively the second-order and fourth-order staggered FD scheme using the $L^\infty$-norm and the $L^1$-norm.
For the second-order scheme, 
	a second-order convergence is obtained for components $H_x$, $H_y$ 
	and $E_z$ in both norms as expected by \cref{lem:errorCFM}.
The divergence constraint converges to second and third order using respectively the 
	$L^\infty$-norm and the $L^1$-norm,
	which is better than expected and still in agreement with the theory. 
For the fourth-order scheme, 
	the magnetic field and the electric field converge to third-order in $L^\infty$-norm, 
	while a fourth-order convergence is obtained in $L^1$-norm.
A second and third order convergence are observed for the divergence of $\mathbold{H}$ in $L^\infty$-norm and the $L^1$-norm. 
These results support our previous analysis presented in \cref{sec:2D}.
 \cref{fig:plotCircleInterfaceFields} shows components $H_x$, $H_y$ and $E_z$ at different time steps using the smallest mesh grid size, 
	namely $h = \tfrac{1}{336}$, 
	and the fourth-order staggered FD scheme with the CFM.
The discontinuities are accurately captured without spurious oscillations. 
 \begin{figure} 
 \centering
  \subfigure[second-order staggered FD scheme]{ \label{fig:convPlotCircularInterface1}
		\setlength\figureheight{0.33\linewidth} 
		\setlength\figurewidth{0.33\linewidth} 
		\tikzset{external/export next=false}
%
%
\begin{tikzpicture}

\begin{axis}[%
width=0.951\figurewidth,
height=\figureheight,
at={(0\figurewidth,0\figureheight)},
scale only axis,
xmode=log,
xmin=0.001,
xmax=0.1,
xminorticks=true,
xlabel={\scriptsize$h$},
ymode=log,
ymin=1e-05,
ymax=0.1,
yminorticks=true,
ylabel={\scriptsize$\|\mathbold{U}-\mathbold{U}_h\|$},
axis background/.style={fill=white},
legend style={at={(0.005,0.99)},anchor=north west,legend cell align=left,align=left,draw=white!15!black,draw=none,fill=none},
legend style={font=\scriptsize},
ylabel style={yshift=-5pt},xlabel style={yshift=2.5pt},tick label style={font=\tiny} 
]
\addplot [color=black,line width=1pt,solid,mark=o,mark options={solid}]
  table[row sep=crcr]{%
0.05	0.0118842031069072\\
0.0357142857142857	0.00596817184745566\\
0.025	0.00294335584597448\\
0.0192307692307692	0.00174303134332458\\
0.0138888888888889	0.000907039210860897\\
0.0104166666666667	0.000508319344181549\\
0.00757575757575758	0.00026941328597109\\
0.00555555555555556	0.000144802440152196\\
0.00409836065573771	7.87716206021293e-05\\
0.00297619047619048	4.15396477746633e-05\\
};
\addlegendentry{$L^\infty$};

\addplot [color=blue,line width=1pt,solid,mark=o,mark options={solid}]
  table[row sep=crcr]{%
0.05	0.00890299557057497\\
0.0357142857142857	0.00457364079652462\\
0.025	0.00226813826471604\\
0.0192307692307692	0.00135262559681515\\
0.0138888888888889	0.000708159228449227\\
0.0104166666666667	0.000399289838995268\\
0.00757575757575758	0.00021231765258101\\
0.00555555555555556	0.000114422947756\\
0.00409836065573771	6.23658707409317e-05\\
0.00297619047619048	3.29347202476255e-05\\
};
\addlegendentry{$L^1$};

\addplot [color=red,line width=1pt,solid]
  table[row sep=crcr]{%
0.05	0.005\\
0.0357142857142857	0.00255102040816327\\
0.025	0.00125\\
0.0192307692307692	0.000739644970414201\\
0.0138888888888889	0.000385802469135802\\
0.0104166666666667	0.000217013888888889\\
0.00757575757575758	0.000114784205693297\\
0.00555555555555556	6.17283950617284e-05\\
0.00409836065573771	3.35931201289976e-05\\
0.00297619047619048	1.77154195011338e-05\\
};
\addlegendentry{$h^2$};

\end{axis}
\end{tikzpicture}%
		\setlength\figureheight{0.33\linewidth} 
		\setlength\figurewidth{0.33\linewidth} 
		\tikzset{external/export next=false}
%
%
\begin{tikzpicture}

\begin{axis}[%
width=0.951\figurewidth,
height=\figureheight,
at={(0\figurewidth,0\figureheight)},
scale only axis,
xmode=log,
xmin=0.001,
xmax=0.1,
xminorticks=true,
xlabel={\scriptsize$h$},
ymode=log,
ymin=1e-07,
ymax=0.1,
yminorticks=true,
ylabel={\scriptsize$\|\nabla^D\cdot\mathbold{H}_h\|$},
axis background/.style={fill=white},
legend style={at={(0.65,0.45)},anchor=north west,legend cell align=left,align=left,draw=white!15!black,draw=none,fill=none},
legend style={font=\scriptsize},
ylabel style={yshift=-5pt},xlabel style={yshift=2.5pt},tick label style={font=\tiny} 
]
\addplot [color=black,line width=1pt,solid,mark=o,mark options={solid}]
  table[row sep=crcr]{%
0.05	0.0828380136152802\\
0.0357142857142857	0.0614893928654681\\
0.025	0.0386802494688538\\
0.0192307692307692	0.0163217078820708\\
0.0138888888888889	0.0108407917273894\\
0.0104166666666667	0.00861079912505147\\
0.00757575757575758	0.00541052808966924\\
0.00555555555555556	0.00311365516517981\\
0.00409836065573771	0.00143839460406525\\
0.00297619047619048	0.000899218663334977\\
};
\addlegendentry{$L^\infty$};

\addplot [color=blue,line width=1pt,solid,mark=o,mark options={solid}]
  table[row sep=crcr]{%
0.05	0.00288162498150164\\
0.0357142857142857	0.00131104962244058\\
0.025	0.000476619295513047\\
0.0192307692307692	0.000185660742825441\\
0.0138888888888889	7.72690006047546e-05\\
0.0104166666666667	3.63351474526224e-05\\
0.00757575757575758	1.52683023013448e-05\\
0.00555555555555556	6.07277168269196e-06\\
0.00409836065573771	2.19831537159133e-06\\
0.00297619047619048	8.9942827341066e-07\\
};
\addlegendentry{$L^1$};

\addplot [color=red,line width=1pt,densely dashed]
  table[row sep=crcr]{%
0.05	0.05\\
0.0357142857142857	0.0255102040816327\\
0.025	0.0125\\
0.0192307692307692	0.00739644970414201\\
0.0138888888888889	0.00385802469135802\\
0.0104166666666667	0.00217013888888889\\
0.00757575757575758	0.00114784205693297\\
0.00555555555555556	0.000617283950617284\\
0.00409836065573771	0.000335931201289976\\
0.00297619047619048	0.000177154195011338\\
};
\addlegendentry{$h^2$};

\addplot [color=red,line width=1pt,solid]
  table[row sep=crcr]{%
0.05	0.0125\\
0.0357142857142857	0.00455539358600583\\
0.025	0.0015625\\
0.0192307692307692	0.000711197086936732\\
0.0138888888888889	0.000267918381344307\\
0.0104166666666667	0.00011302806712963\\
0.00757575757575758	4.34788657929154e-05\\
0.00555555555555556	1.71467764060357e-05\\
0.00409836065573771	6.8838360920077e-06\\
0.00297619047619048	2.63622314004967e-06\\
};
\addlegendentry{$h^3$};

\end{axis}
\end{tikzpicture}%
		} 
  \subfigure[fourth-order staggered FD scheme]{ \label{fig:convPlotCircularInterface2}
		\setlength\figureheight{0.33\linewidth} 
		\setlength\figurewidth{0.33\linewidth} 
		\tikzset{external/export next=false}
%
%
\begin{tikzpicture}

\begin{axis}[%
width=0.951\figurewidth,
height=\figureheight,
at={(0\figurewidth,0\figureheight)},
scale only axis,
xmode=log,
xmin=0.001,
xmax=0.1,
xminorticks=true,
xlabel={\scriptsize$h$},
ymode=log,
ymin=1e-09,
ymax=1,
yminorticks=true,
ylabel={\scriptsize$\|\mathbold{U}-\mathbold{U}_h\|$},
axis background/.style={fill=white},
legend style={at={(0.005,0.99)},anchor=north west,legend cell align=left,align=left,draw=white!15!black,draw=none,fill=none},
legend style={font=\scriptsize},
ylabel style={yshift=-5pt},xlabel style={yshift=2.5pt},tick label style={font=\tiny} 
]
\addplot [color=black,line width=1pt,solid,mark=o,mark options={solid}]
  table[row sep=crcr]{%
0.05	0.0239710094290927\\
0.0357142857142857	0.00866915862248696\\
0.025	0.00534267066182986\\
0.0192307692307692	0.00159444465686544\\
0.0138888888888889	0.000755993901364224\\
0.0104166666666667	0.000312091948571958\\
0.00757575757575758	0.000145375852371924\\
0.00555555555555556	5.42493779150452e-05\\
0.00409836065573771	2.22800334757784e-05\\
0.00297619047619048	8.3758357872783e-06\\
};
\addlegendentry{$L^\infty$};

\addplot [color=blue,line width=1pt,solid,mark=o,mark options={solid}]
  table[row sep=crcr]{%
0.05	0.00613543625147415\\
0.0357142857142857	0.00134260579312115\\
0.025	0.000336694709209418\\
0.0192307692307692	0.000122095520149661\\
0.0138888888888889	3.24824826902807e-05\\
0.0104166666666667	9.17686211191629e-06\\
0.00757575757575758	2.68091423743548e-06\\
0.00555555555555556	7.75667623956443e-07\\
0.00409836065573771	2.26200547789853e-07\\
0.00297619047619048	6.13230558755571e-08\\
};
\addlegendentry{$L^1$};

\addplot [color=red,line width=1pt,densely dashed]
  table[row sep=crcr]{%
0.05	0.3125\\
0.0357142857142857	0.113884839650146\\
0.025	0.0390625\\
0.0192307692307692	0.0177799271734183\\
0.0138888888888889	0.00669795953360768\\
0.0104166666666667	0.00282570167824074\\
0.00757575757575758	0.00108697164482288\\
0.00555555555555556	0.000428669410150892\\
0.00409836065573771	0.000172095902300193\\
0.00297619047619048	6.59055785012418e-05\\
};
\addlegendentry{$h^3$};

\addplot [color=red,line width=1pt,solid]
  table[row sep=crcr]{%
0.05	0.000625\\
0.0357142857142857	0.000162692628071637\\
0.025	3.90625e-05\\
0.0192307692307692	1.36768670564756e-05\\
0.0138888888888889	3.72108862978204e-06\\
0.0104166666666667	1.17737569926698e-06\\
0.00757575757575758	3.29385346916026e-07\\
0.00555555555555556	9.52598689224204e-08\\
0.00409836065573771	2.82124430000316e-08\\
0.00297619047619048	7.84590220252878e-09\\
};
\addlegendentry{$h^4$};

\end{axis}
\end{tikzpicture}%
		\setlength\figureheight{0.33\linewidth} 
		\setlength\figurewidth{0.33\linewidth} 
		\tikzset{external/export next=false}
%
%
\begin{tikzpicture}

\begin{axis}[%
width=0.951\figurewidth,
height=\figureheight,
at={(0\figurewidth,0\figureheight)},
scale only axis,
xmode=log,
xmin=0.001,
xmax=0.1,
xminorticks=true,
xlabel={\scriptsize$h$},
ymode=log,
ymin=1e-06,
ymax=10,
yminorticks=true,
ylabel={\scriptsize$\|\tilde{\nabla}^D\cdot\mathbold{H}_h\|$},
axis background/.style={fill=white},
legend style={at={(0.65,0.45)},anchor=north west,legend cell align=left,align=left,draw=white!15!black,draw=none,fill=none},
legend style={font=\scriptsize},
ylabel style={yshift=-5pt},xlabel style={yshift=2.5pt},tick label style={font=\tiny} 
]
\addplot [color=black,line width=1pt,solid,mark=o,mark options={solid}]
  table[row sep=crcr]{%
0.05	1.02954951917225\\
0.0357142857142857	0.732228651443478\\
0.025	0.575557832796079\\
0.0192307692307692	0.25248240001855\\
0.0138888888888889	0.161022485801425\\
0.0104166666666667	0.118957610295297\\
0.00757575757575758	0.0669816347503911\\
0.00555555555555556	0.0360179253698334\\
0.00409836065573771	0.0185018464628683\\
0.00297619047619048	0.0101131773758425\\
};
\addlegendentry{$L^\infty$};

\addplot [color=blue,line width=1pt,solid,mark=o,mark options={solid}]
  table[row sep=crcr]{%
0.05	0.0413651918376417\\
0.0357142857142857	0.0144310435495455\\
0.025	0.0053972639668553\\
0.0192307692307692	0.00213078332888018\\
0.0138888888888889	0.000736424227112529\\
0.0104166666666667	0.000360168253231505\\
0.00757575757575758	0.000146607519535708\\
0.00555555555555556	5.31508214398413e-05\\
0.00409836065573771	1.89463570087043e-05\\
0.00297619047619048	7.62555948796126e-06\\
};
\addlegendentry{$L^1$};

\addplot [color=red,line width=1pt,densely dashed]
  table[row sep=crcr]{%
0.05	7.5\\
0.0357142857142857	3.8265306122449\\
0.025	1.875\\
0.0192307692307692	1.1094674556213\\
0.0138888888888889	0.578703703703704\\
0.0104166666666667	0.325520833333333\\
0.00757575757575758	0.172176308539945\\
0.00555555555555556	0.0925925925925926\\
0.00409836065573771	0.0503896801934964\\
0.00297619047619048	0.0265731292517007\\
};
\addlegendentry{$h^2$};

\addplot [color=red,line width=1pt,solid]
  table[row sep=crcr]{%
0.05	0.0125\\
0.0357142857142857	0.00455539358600583\\
0.025	0.0015625\\
0.0192307692307692	0.000711197086936732\\
0.0138888888888889	0.000267918381344307\\
0.0104166666666667	0.00011302806712963\\
0.00757575757575758	4.34788657929154e-05\\
0.00555555555555556	1.71467764060357e-05\\
0.00409836065573771	6.8838360920077e-06\\
0.00297619047619048	2.63622314004967e-06\\
};
\addlegendentry{$h^3$};

\end{axis}
\end{tikzpicture}%
		} 
  \caption{Convergence plots for a problem with a manufactured solution and the circular interface using fourth-order approximations of correction functions, and either the second-order or fourth-order staggered FD scheme. It is recalled that $\mathbold{U} = [H_x,H_y,E_z]^T$.}
   \label{fig:convPlotCircularInterface}
\end{figure}
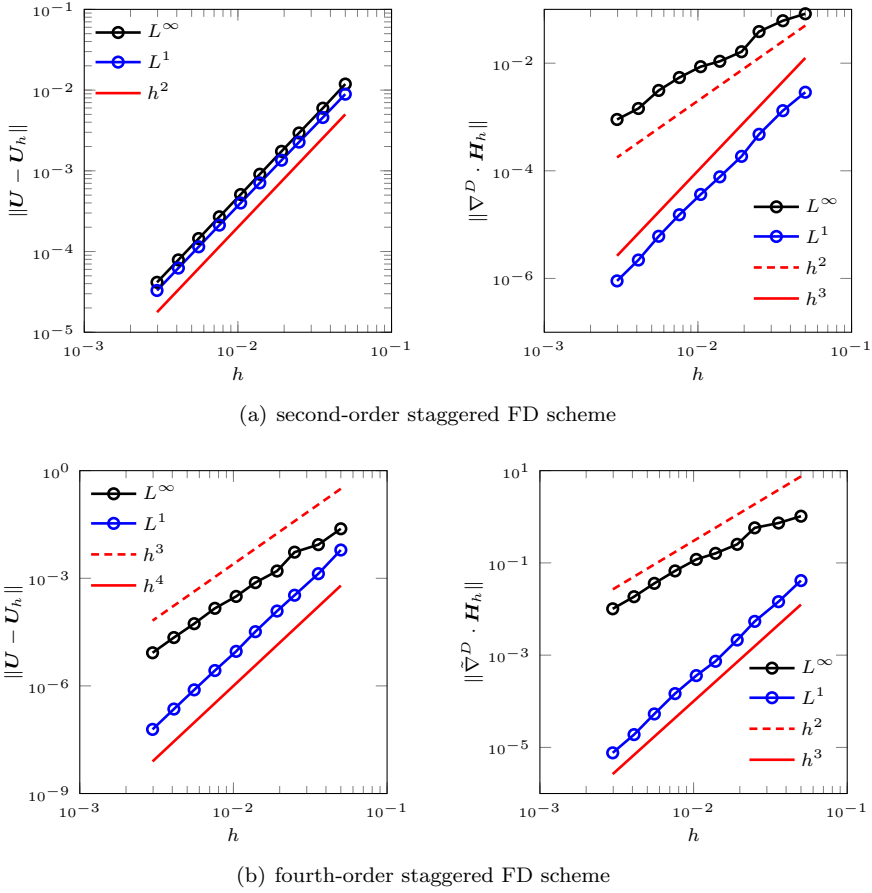
\begin{figure}     
	\centering
	\subfigure[$H_x$]{
	\stackunder[5pt]{
	\stackunder[5pt]{\includegraphics[width=2.25in]{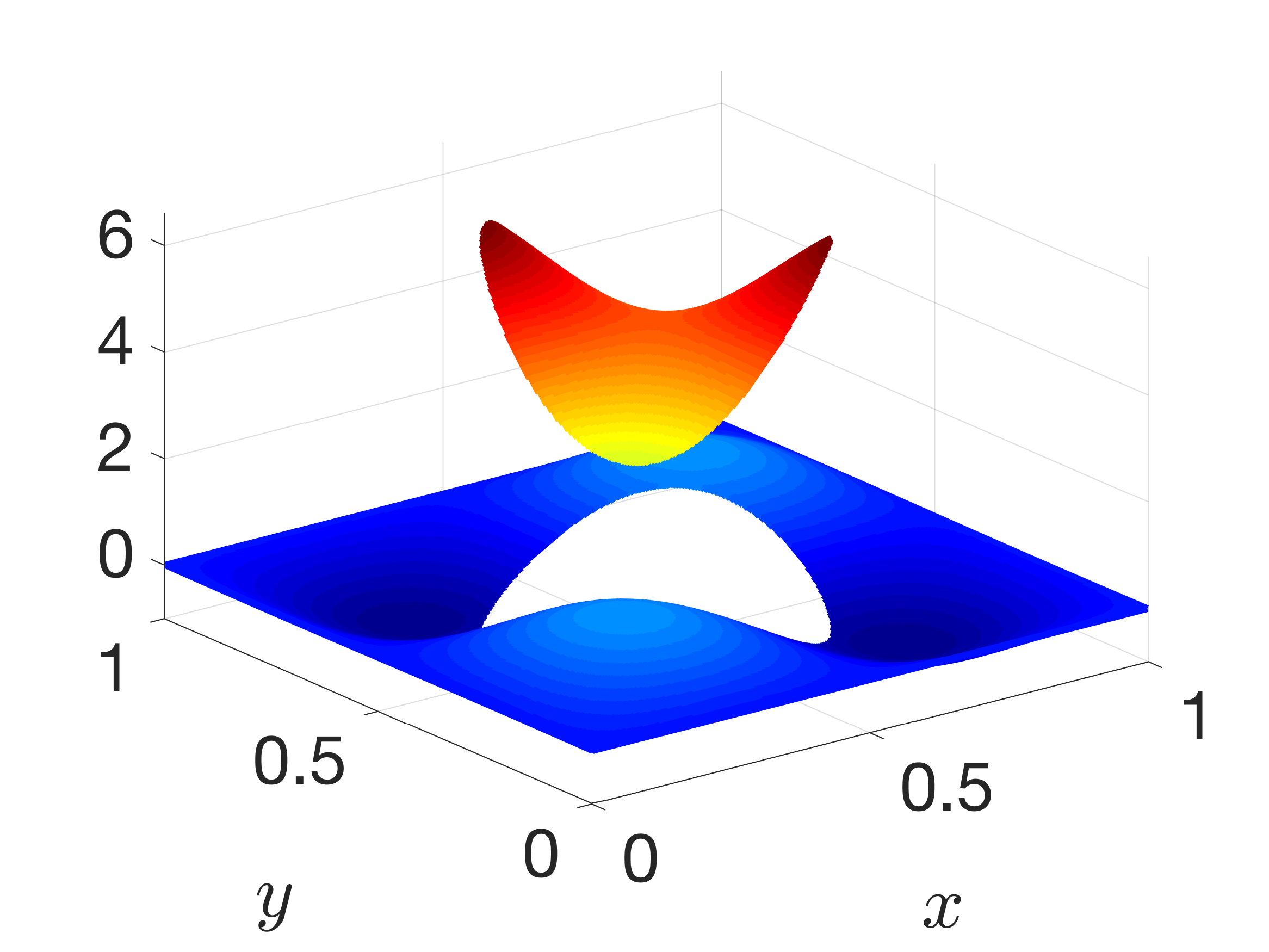}}{\footnotesize$t=0.25$}
	\stackunder[5pt]{	\includegraphics[width=2.25in]{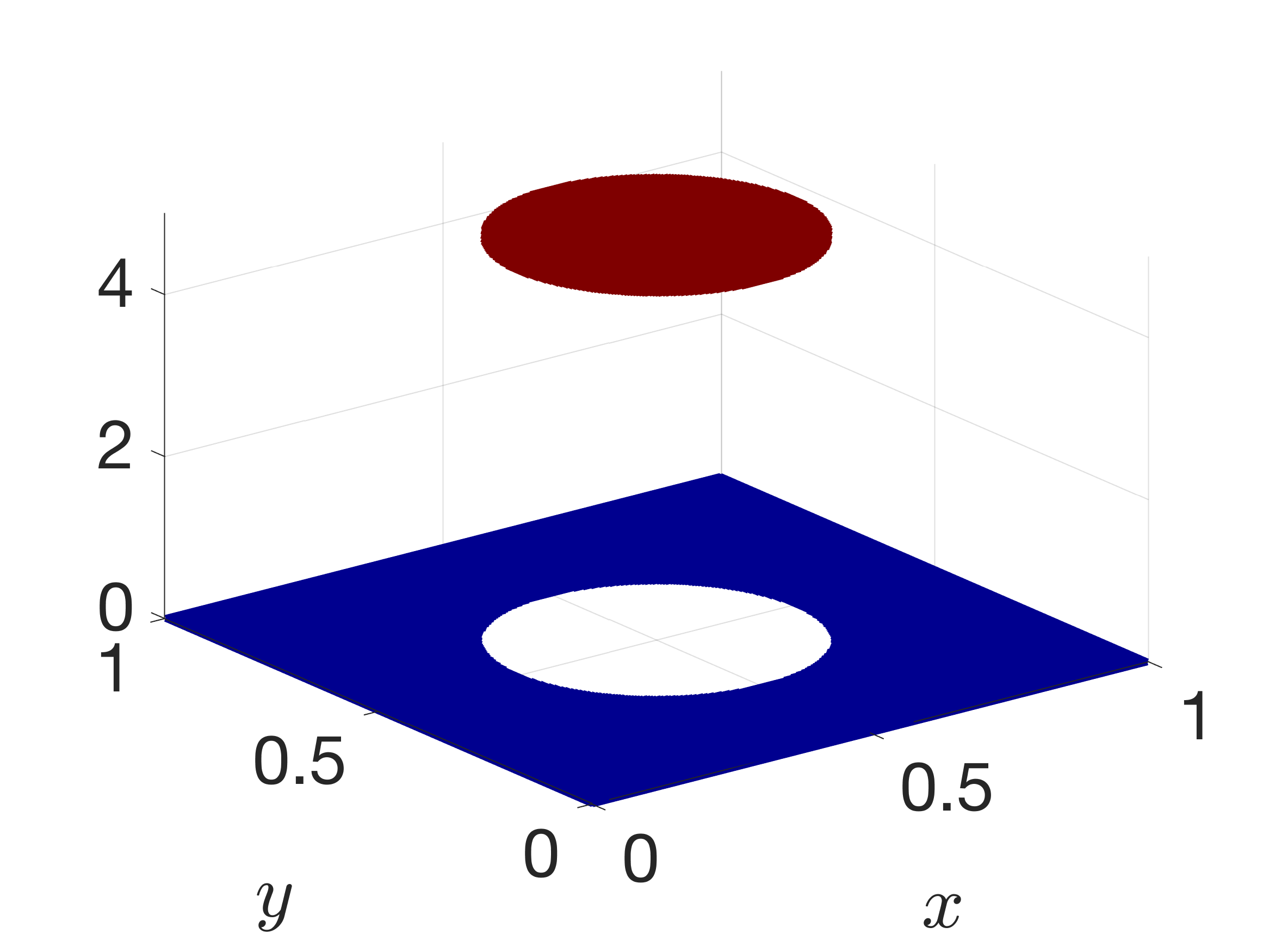}}{\footnotesize$t=0.5$}
	} 
	{\phantom{A}}
	}
\subfigure[$H_y$]{
	\stackunder[5pt]{
	\stackunder[5pt]{\includegraphics[width=2.25in]{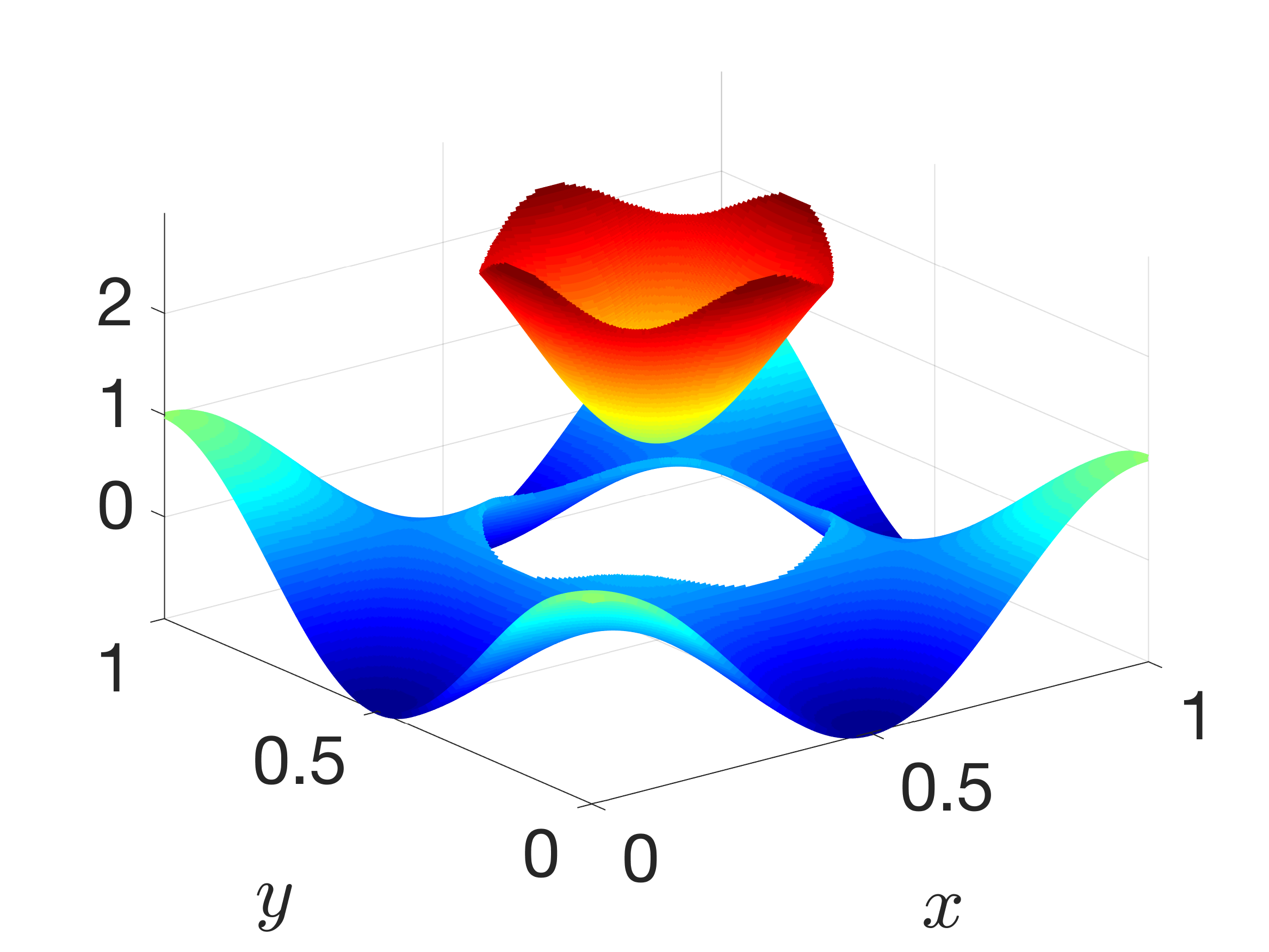}}{\footnotesize$t=0.25$}
	\stackunder[5pt]{\includegraphics[width=2.25in]{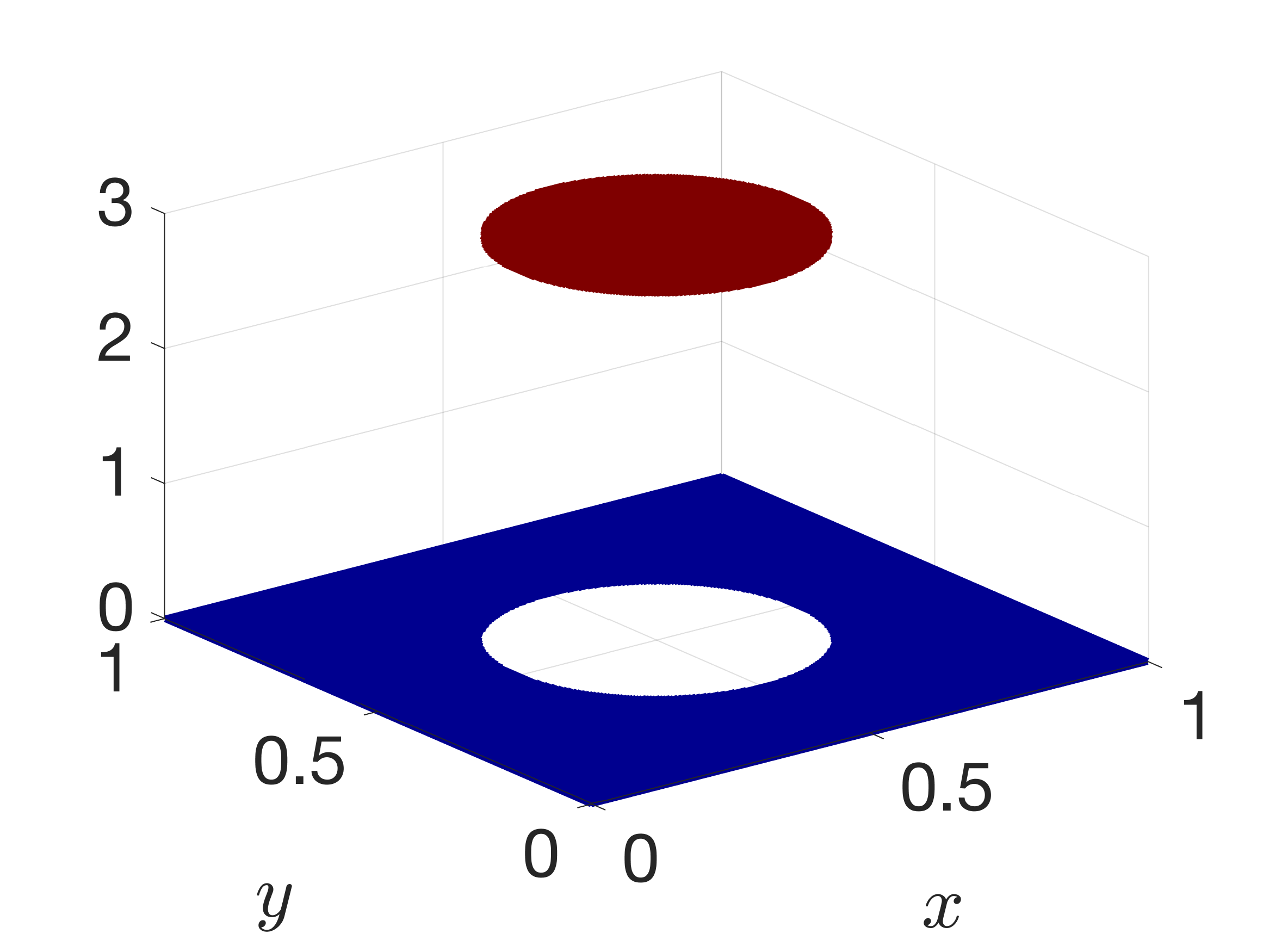}}{\footnotesize$t=0.5$}
	} 
	{\phantom{A}}
	}
	\subfigure[$E_z$]{
	\stackunder[5pt]{
	\stackunder[5pt]{\includegraphics[width=2.25in]{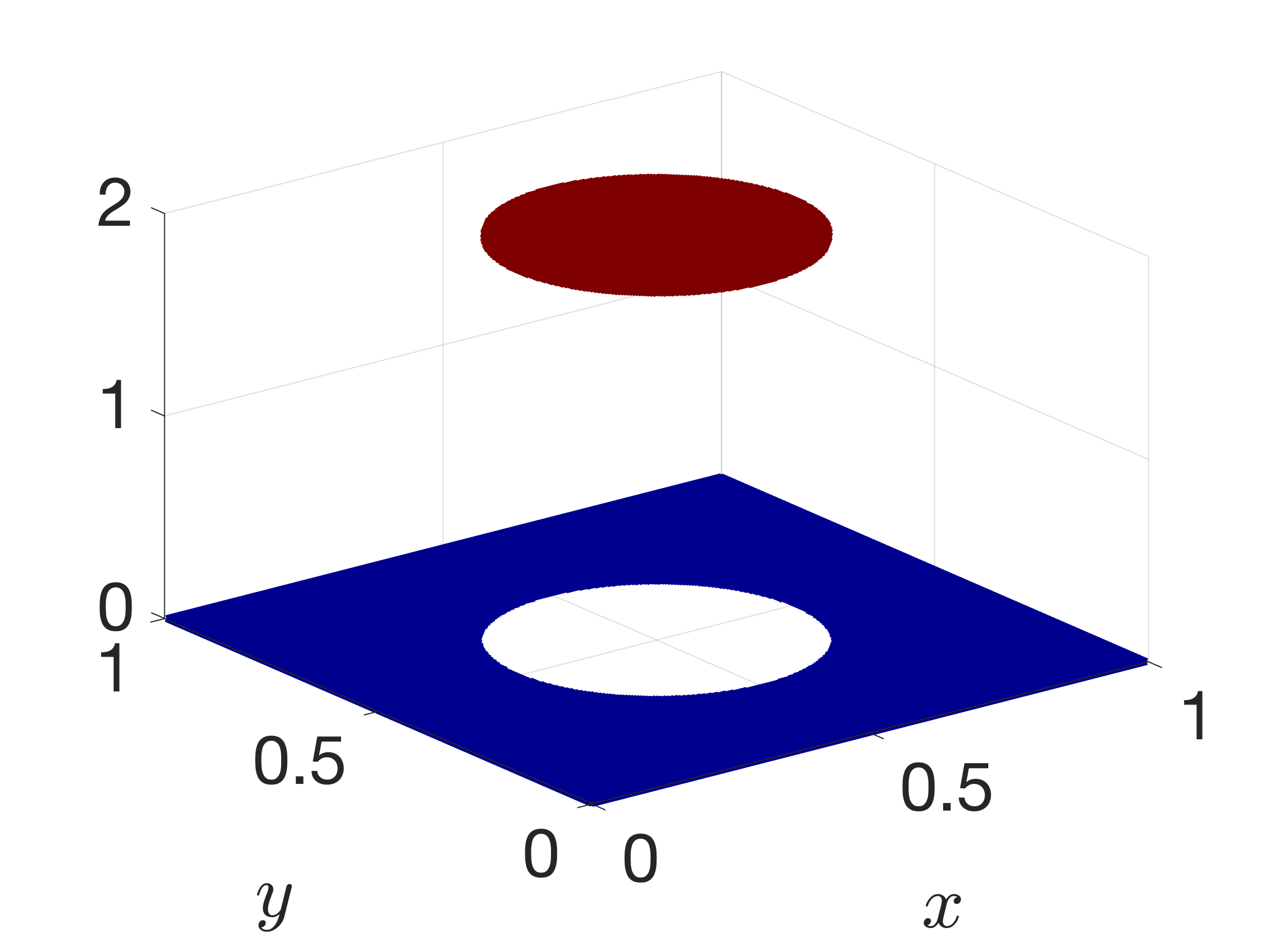}}{\footnotesize$t=0.25$}
	\stackunder[5pt]{\includegraphics[width=2.25in]{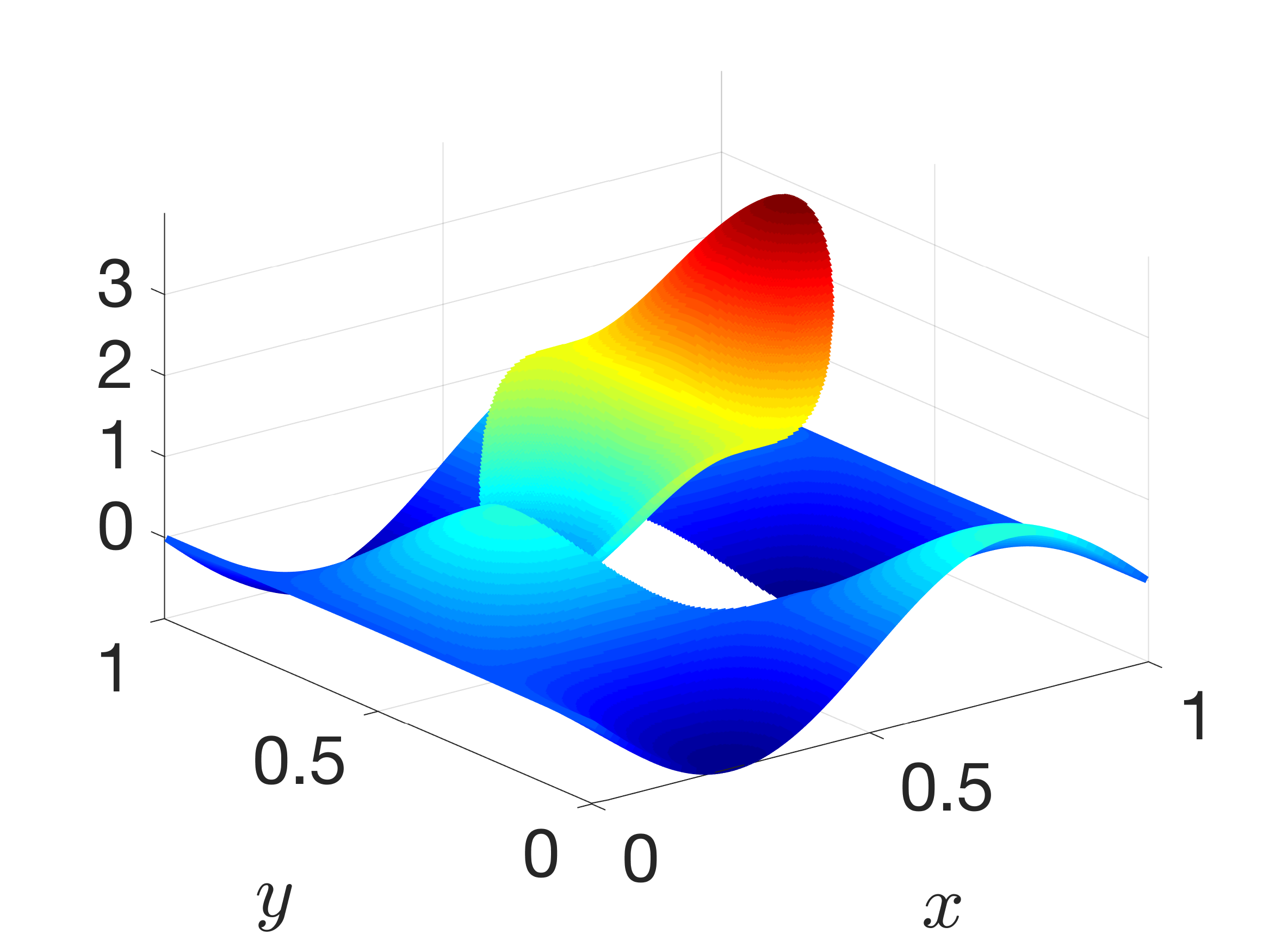}}{\footnotesize$t=0.5$}
	} 
	{\phantom{A}}
	}
       \caption{The components $H_x$, $H_y$ and $E_z$ at two time steps with $h = \tfrac{1}{336}$ and $\Delta t = \tfrac{h}{2}$ using a fourth-order FDTD scheme with the CFM for a problem with a manufactured solution and the circle interface.}
       \label{fig:plotCircleInterfaceFields}
\end{figure}
		
\FloatBarrier
\subsection{5-star interface}
The level set function is given by
\begin{equation*}
\phi(x,y) = (x-x_0)^2 +(y-y_0)^2 - r^2(\theta),
\end{equation*}
	where 
\begin{equation*}
	r(\theta) = r_0 + \epsilon\,\sin(\omega\,\theta(x,y)),
\end{equation*}
	$\omega=5$, $x_0 = y_0 = 0.5$, $r_0 = 0.25$, $\epsilon = 0.05$ and $\theta(x,y)$ is the angle between the vector $[x-x_0,y-y_0]^T$ and the $x$-axis.
\cref{fig:starInterfaceGeo} illustrates the geometry of the interface.
The manufactured solutions are :
\begin{equation*}
	\begin{aligned}
		H_x^+ =&\,\, \sin(4\,\pi\,x)\,\sin(4\,\pi\,y)\,\cos(2\,\pi\,t), \\
		H_y^+ =&\,\, \cos(4\,\pi\,x)\,\cos(4\,\pi\,y)\,\cos(2\,\pi\,t), \\
		E_z^+ =&\,\, 0, \\
		H_x^- =&\,\, (-x\,e^{-x\,y}+2)\,\sin(2\,\pi\,t), \\
		H_y^- =&\,\, (y\,e^{-x\,y}+3)\,\sin(2\,\pi\,t), \\
		E_z^- =&\,\, \sin(2\,\pi\,x\,y)\,\cos(2\,\pi\,t).
	\end{aligned}
\end{equation*}
The associated source terms are 
\begin{equation*}
	\begin{aligned}
		f_{1_x}^+ =&\,\, -2\,\pi\,\sin(4\,\pi\,x)\,\sin(4\,\pi\,y)\,\sin(2\,\pi\,t),\\
		f_{1_x}^-  =&\,\, \big(2\,\pi\,(-x\,e^{-x\,y}+2)+2\,\pi\,x\,\cos(2\,\pi\,x\,y)\big)\,\cos(2\,\pi\,t),\\
		f_{1_y}^+ =&\,\, -2\,\pi\,\cos(4\,\pi\,x)\,\cos(4\,\pi\,y)\,\sin(2\,\pi\,t),\\
		f_{1_y}^-  =&\,\, 2\,\pi\,(y\,e^{-x\,y}-y\,\cos(2\,\pi\,x\,y)+3)\,\cos(2\,\pi\,t),\\
		f_2^+ =&\,\, 8\,\pi\,\sin(4\,\pi\,x)\,\cos(4\,\pi\,y)\,\cos(2\,\pi\,t), \\
		f_2^-  =&\,\, \big(-2\,\pi\,\sin(2\,\pi\,x\,y)+y^2\,e^{-x\,y}+x^2\,e^{-x\,y}\big)\,\sin(2\,\pi\,t)+\sin(2\,\pi\,x\,y)\,\cos(2\,\pi\,t).
	\end{aligned}
\end{equation*}
\cref{fig:convPlotStarInterface} illustrates the convergence plots for fourth-order approximations of correction functions, and either the second-order or fourth-order staggered FD scheme.
A second-order convergence for the solutions is obtained with the second-order FD scheme in both norms 
	while a second and third order convergence for the divergence constraint 
	are observed with respectively the $L^{\infty}$-norm and the $L^1$-norm.
For the fourth-order FD scheme, 
	the solutions converge to third and fourth order in respectively 
	$L^{\infty}$-norm and $L^1$-norm.
We also observe a second-order convergence for the divergence constraint 
	using the $L^{\infty}$-norm and a third-order convergence using the $L^1$-norm.
\cref{fig:plotStarInterfaceFields} shows the evolution of components $H_x$, $H_y$ and $E_z$. 
Here again, 
	the results are in agreement with the theory and 
	the discontinuities are accurately captured for a more complex interface.
 \begin{figure} 
 \centering
  \subfigure[second-order staggered FD scheme]{
		\setlength\figureheight{0.33\linewidth} 
		\setlength\figurewidth{0.33\linewidth} 
		\tikzset{external/export next=false}
%
%
\begin{tikzpicture}

\begin{axis}[%
width=0.951\figurewidth,
height=\figureheight,
at={(0\figurewidth,0\figureheight)},
scale only axis,
xmode=log,
xmin=0.001,
xmax=0.1,
xminorticks=true,
xlabel={\scriptsize$h$},
ymode=log,
ymin=1e-05,
ymax=0.1,
yminorticks=true,
ylabel={\scriptsize$\|\mathbold{U}-\mathbold{U}_h\|$},
axis background/.style={fill=white},
legend style={at={(0.005,0.99)},anchor=north west,legend cell align=left,align=left,draw=white!15!black,draw=none,fill=none},
legend style={font=\scriptsize},
ylabel style={yshift=-5pt},xlabel style={yshift=2.5pt},tick label style={font=\tiny} 
]
\addplot [color=black,line width=1pt,solid,mark=o,mark options={solid}]
  table[row sep=crcr]{%
0.05	0.0193458400317951\\
0.0357142857142857	0.0096174406023526\\
0.025	0.00453667762350904\\
0.0192307692307692	0.00268604356198591\\
0.0138888888888889	0.0013968215884405\\
0.0104166666666667	0.000783955538834223\\
0.00757575757575758	0.000413856014637165\\
0.00555555555555556	0.0002225532082971\\
0.00409836065573771	0.0001211469735934\\
0.00297619047619048	6.38563376380486e-05\\
};
\addlegendentry{$L^\infty$};

\addplot [color=blue,line width=1pt,solid,mark=o,mark options={solid}]
  table[row sep=crcr]{%
0.05	0.0159988491707464\\
0.0357142857142857	0.0076550644970154\\
0.025	0.00368988883465972\\
0.0192307692307692	0.00215923881959111\\
0.0138888888888889	0.00112639055598226\\
0.0104166666666667	0.000633264703994734\\
0.00757575757575758	0.000335690925683246\\
0.00555555555555556	0.00018079960362083\\
0.00409836065573771	9.8444685359695e-05\\
0.00297619047619048	5.19413373252499e-05\\
};
\addlegendentry{$L^1$};

\addplot [color=red,line width=1pt,solid]
  table[row sep=crcr]{%
0.05	0.005\\
0.0357142857142857	0.00255102040816327\\
0.025	0.00125\\
0.0192307692307692	0.000739644970414201\\
0.0138888888888889	0.000385802469135802\\
0.0104166666666667	0.000217013888888889\\
0.00757575757575758	0.000114784205693297\\
0.00555555555555556	6.17283950617284e-05\\
0.00409836065573771	3.35931201289976e-05\\
0.00297619047619048	1.77154195011338e-05\\
};
\addlegendentry{$h^2$};

\end{axis}
\end{tikzpicture}%
		\setlength\figureheight{0.33\linewidth} 
		\setlength\figurewidth{0.33\linewidth} 
		\tikzset{external/export next=false}
%
%
\begin{tikzpicture}

\begin{axis}[%
width=0.951\figurewidth,
height=\figureheight,
at={(0\figurewidth,0\figureheight)},
scale only axis,
xmode=log,
xmin=0.001,
xmax=0.1,
xminorticks=true,
xlabel={\scriptsize$h$},
ymode=log,
ymin=1e-07,
ymax=1,
yminorticks=true,
ylabel={\scriptsize$\|\nabla^D\cdot\mathbold{H}_h\|$},
axis background/.style={fill=white},
legend style={at={(0.65,0.45)},anchor=north west,legend cell align=left,align=left,draw=white!15!black,draw=none,fill=none},
legend style={font=\scriptsize},
ylabel style={yshift=-5pt},xlabel style={yshift=2.5pt},tick label style={font=\tiny} 
]
\addplot [color=black,line width=1pt,solid,mark=o,mark options={solid}]
  table[row sep=crcr]{%
0.05	0.388718264038786\\
0.0357142857142857	0.195329912680464\\
0.025	0.0499486733282097\\
0.0192307692307692	0.0249625354881676\\
0.0138888888888889	0.0128589861203068\\
0.0104166666666667	0.00571209751087842\\
0.00757575757575758	0.00502769154861227\\
0.00555555555555556	0.00234950704111725\\
0.00409836065573771	0.00131907609159043\\
0.00297619047619048	0.000583592695619473\\
};
\addlegendentry{$L^\infty$};

\addplot [color=blue,line width=1pt,solid,mark=o,mark options={solid}]
  table[row sep=crcr]{%
0.05	0.0135229809556459\\
0.0357142857142857	0.00224610616663869\\
0.025	0.000609438177417783\\
0.0192307692307692	0.000236049000685062\\
0.0138888888888889	8.03325733190095e-05\\
0.0104166666666667	2.89535941738915e-05\\
0.00757575757575758	1.36286942130171e-05\\
0.00555555555555556	4.82072485575168e-06\\
0.00409836065573771	1.9368607176008e-06\\
0.00297619047619048	6.78662604676997e-07\\
};
\addlegendentry{$L^1$};

\addplot [color=red,line width=1pt,densely dashed]
  table[row sep=crcr]{%
0.05	0.05\\
0.0357142857142857	0.0255102040816327\\
0.025	0.0125\\
0.0192307692307692	0.00739644970414201\\
0.0138888888888889	0.00385802469135802\\
0.0104166666666667	0.00217013888888889\\
0.00757575757575758	0.00114784205693297\\
0.00555555555555556	0.000617283950617284\\
0.00409836065573771	0.000335931201289976\\
0.00297619047619048	0.000177154195011338\\
};
\addlegendentry{$h^2$};

\addplot [color=red,line width=1pt,solid]
  table[row sep=crcr]{%
0.05	0.0125\\
0.0357142857142857	0.00455539358600583\\
0.025	0.0015625\\
0.0192307692307692	0.000711197086936732\\
0.0138888888888889	0.000267918381344307\\
0.0104166666666667	0.00011302806712963\\
0.00757575757575758	4.34788657929154e-05\\
0.00555555555555556	1.71467764060357e-05\\
0.00409836065573771	6.8838360920077e-06\\
0.00297619047619048	2.63622314004967e-06\\
};
\addlegendentry{$h^3$};

\end{axis}
\end{tikzpicture}%
		} 
  \subfigure[fourth-order staggered FD scheme]{
		\setlength\figureheight{0.33\linewidth} 
		\setlength\figurewidth{0.33\linewidth} 
		\tikzset{external/export next=false}
%
%
\begin{tikzpicture}

\begin{axis}[%
width=0.951\figurewidth,
height=\figureheight,
at={(0\figurewidth,0\figureheight)},
scale only axis,
xmode=log,
xmin=0.001,
xmax=0.1,
xminorticks=true,
xlabel={\scriptsize$h$},
ymode=log,
ymin=1e-07,
ymax=1,
yminorticks=true,
ylabel={\scriptsize$\|\mathbold{U}-\mathbold{U}_h\|$},
axis background/.style={fill=white},
legend style={at={(0.005,0.99)},anchor=north west,legend cell align=left,align=left,draw=white!15!black,draw=none,fill=none},
legend style={font=\scriptsize},
ylabel style={yshift=-5pt},xlabel style={yshift=2.5pt},tick label style={font=\tiny} 
]
\addplot [color=black,line width=1pt,solid,mark=o,mark options={solid}]
  table[row sep=crcr]{%
0.05	0.24065824891718\\
0.0357142857142857	0.0655033167283471\\
0.025	0.01026366867242\\
0.0192307692307692	0.00441840150919359\\
0.0138888888888889	0.000977887133516098\\
0.0104166666666667	0.000392137465385911\\
0.00757575757575758	0.000112801015494268\\
0.00555555555555556	3.87065285868723e-05\\
0.00409836065573771	1.49724972757781e-05\\
0.00297619047619048	4.52980349041621e-06\\
};
\addlegendentry{$L^\infty$};

\addplot [color=blue,line width=1pt,solid,mark=o,mark options={solid}]
  table[row sep=crcr]{%
0.05	0.0372797528497772\\
0.0357142857142857	0.0100410712912017\\
0.025	0.00185063941218045\\
0.0192307692307692	0.000495013939740834\\
0.0138888888888889	0.00012854267961707\\
0.0104166666666667	4.37503233913714e-05\\
0.00757575757575758	1.10700185082256e-05\\
0.00555555555555556	3.61778087092278e-06\\
0.00409836065573771	1.04489325369673e-06\\
0.00297619047619048	2.96808424952396e-07\\
};
\addlegendentry{$L^1$};

\addplot [color=red,line width=1pt,densely dashed]
  table[row sep=crcr]{%
0.05	0.3125\\
0.0357142857142857	0.113884839650146\\
0.025	0.0390625\\
0.0192307692307692	0.0177799271734183\\
0.0138888888888889	0.00669795953360768\\
0.0104166666666667	0.00282570167824074\\
0.00757575757575758	0.00108697164482288\\
0.00555555555555556	0.000428669410150892\\
0.00409836065573771	0.000172095902300193\\
0.00297619047619048	6.59055785012418e-05\\
};
\addlegendentry{$h^3$};

\addplot [color=red,line width=1pt,solid]
  table[row sep=crcr]{%
0.05	0.0625\\
0.0357142857142857	0.0162692628071637\\
0.025	0.00390625\\
0.0192307692307692	0.00136768670564756\\
0.0138888888888889	0.000372108862978204\\
0.0104166666666667	0.000117737569926698\\
0.00757575757575758	3.29385346916026e-05\\
0.00555555555555556	9.52598689224204e-06\\
0.00409836065573771	2.82124430000316e-06\\
0.00297619047619048	7.84590220252878e-07\\
};
\addlegendentry{$h^4$};

\end{axis}
\end{tikzpicture}%
		\setlength\figureheight{0.33\linewidth} 
		\setlength\figurewidth{0.33\linewidth} 
		\tikzset{external/export next=false}
%
%
\begin{tikzpicture}

\begin{axis}[%
width=0.951\figurewidth,
height=\figureheight,
at={(0\figurewidth,0\figureheight)},
scale only axis,
xmode=log,
xmin=0.001,
xmax=0.1,
xminorticks=true,
xlabel={\scriptsize$h$},
ymode=log,
ymin=1e-06,
ymax=100,
yminorticks=true,
ylabel={\scriptsize$\|\tilde{\nabla}^D\cdot\mathbold{H}_h\|$},
axis background/.style={fill=white},
legend style={at={(0.65,0.45)},anchor=north west,legend cell align=left,align=left,draw=white!15!black,draw=none,fill=none},
legend style={font=\scriptsize},
ylabel style={yshift=-5pt},xlabel style={yshift=2.5pt},tick label style={font=\tiny} 
]
\addplot [color=black,line width=1pt,solid,mark=o,mark options={solid}]
  table[row sep=crcr]{%
0.05	11.3607841946349\\
0.0357142857142857	4.0962890336975\\
0.025	1.245600591072\\
0.0192307692307692	0.51018474378737\\
0.0138888888888889	0.185514848126637\\
0.0104166666666667	0.0825214701220602\\
0.00757575757575758	0.0459705908345995\\
0.00555555555555556	0.0200575714758742\\
0.00409836065573771	0.0122413262425027\\
0.00297619047619048	0.00470806260022982\\
};
\addlegendentry{$L^\infty$};

\addplot [color=blue,line width=1pt,solid,mark=o,mark options={solid}]
  table[row sep=crcr]{%
0.05	0.330436598210592\\
0.0357142857142857	0.0807459885304712\\
0.025	0.0176675481830632\\
0.0192307692307692	0.00549472546360746\\
0.0138888888888889	0.00146422908973815\\
0.0104166666666667	0.000396111465307835\\
0.00757575757575758	0.0001565699710908\\
0.00555555555555556	5.68032500125252e-05\\
0.00409836065573771	2.11763180603505e-05\\
0.00297619047619048	7.09919943680722e-06\\
};
\addlegendentry{$L^1$};

\addplot [color=red,line width=1pt,densely dashed]
  table[row sep=crcr]{%
0.05	12.500000000000002\\
0.0357142857142857	6.377551020408157\\
0.025	3.125000000000000\\
0.0192307692307692	1.849112426035497\\
0.0138888888888889	0.964506172839508\\
0.0104166666666667	0.542534722222226\\
0.00757575757575758	0.286960514233242\\
0.00555555555555556	0.154320987654321\\
0.00409836065573771	0.083982800322494\\
0.00297619047619048	0.044288548752835\\
};
\addlegendentry{$h^2$};

\addplot [color=red,line width=1pt,solid]
  table[row sep=crcr]{%
0.05	0.625\\
0.0357142857142857	0.227769679300292\\
0.025	0.078125\\
0.0192307692307692	0.0355598543468366\\
0.0138888888888889	0.0133959190672154\\
0.0104166666666667	0.00565140335648148\\
0.00757575757575758	0.00217394328964577\\
0.00555555555555556	0.000857338820301783\\
0.00409836065573771	0.000344191804600385\\
0.00297619047619048	0.000131811157002484\\
};
\addlegendentry{$h^3$};

\end{axis}
\end{tikzpicture}%
		} 
  \caption{Convergence plots for the problem with a manufactured solution and the 5-star interface using fourth-order approximations of correction functions, and either the second-order or fourth-order staggered FD scheme. It is recalled that $\mathbold{U} = [H_x,H_y,E_z]^T$.}
   \label{fig:convPlotStarInterface}
\end{figure}
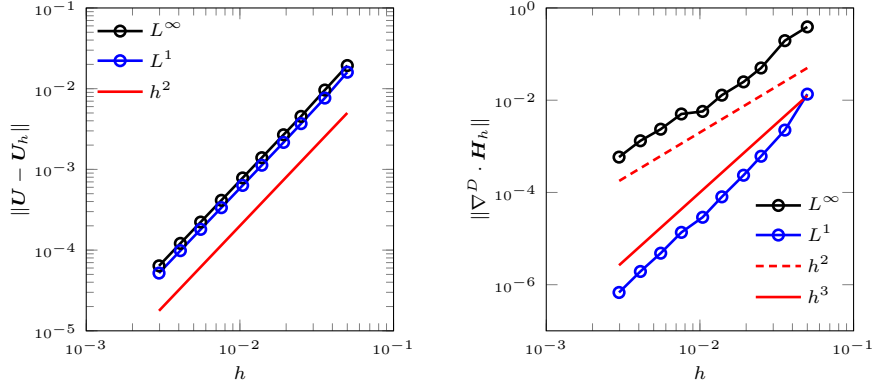
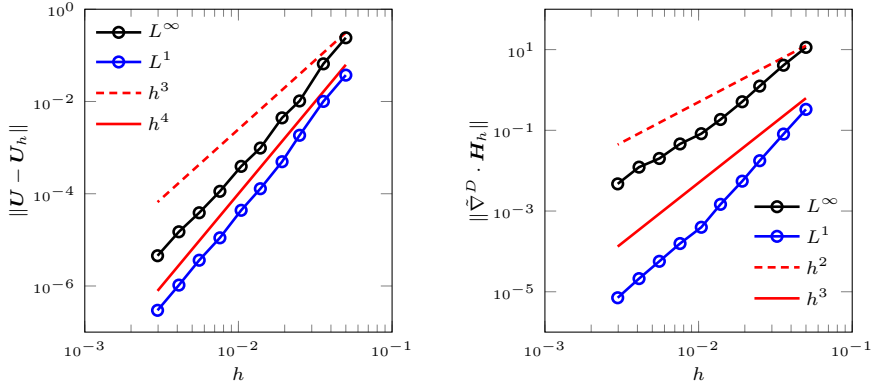
\begin{figure}    
	\centering
	\subfigure[$H_x$]{
	\stackunder[5pt]{
	\stackunder[5pt]{\includegraphics[width=2.25in]{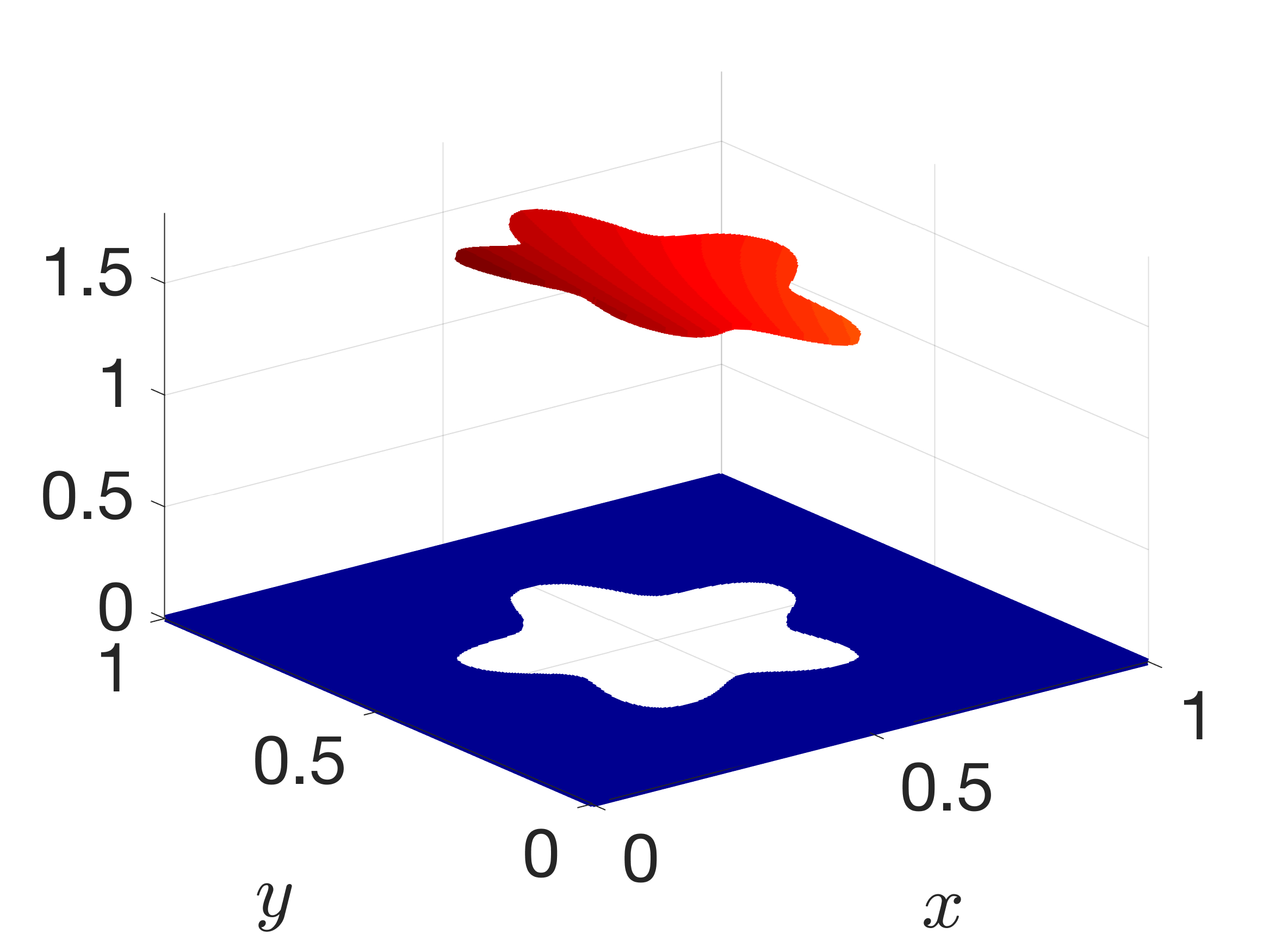}}{\footnotesize$t=0.25$}
	\stackunder[5pt]{	\includegraphics[width=2.25in]{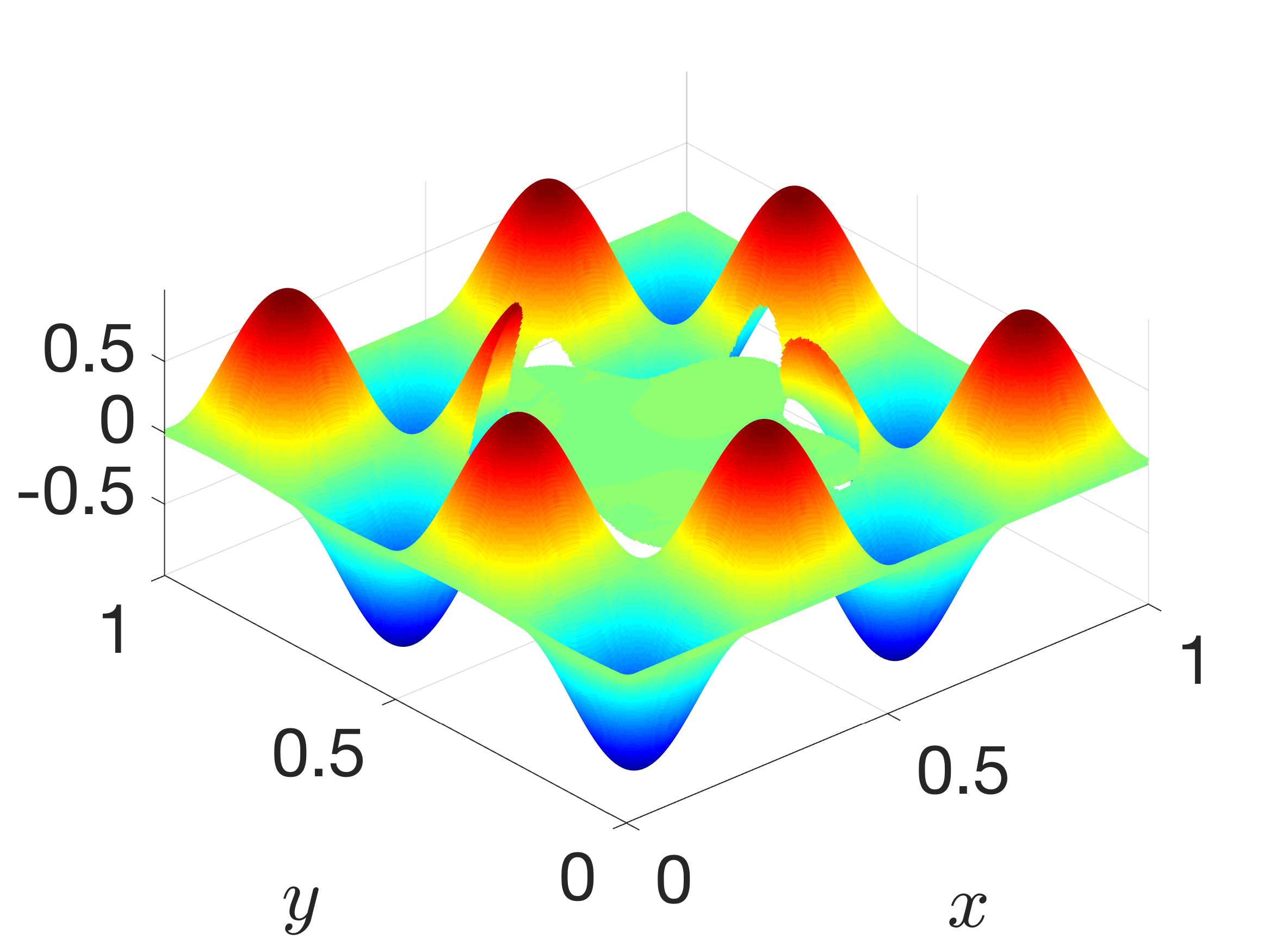}}{\footnotesize$t=0.5$}
	} 
	{\phantom{A}}
	}
\subfigure[$H_y$]{
	\stackunder[5pt]{
	\stackunder[5pt]{\includegraphics[width=2.25in]{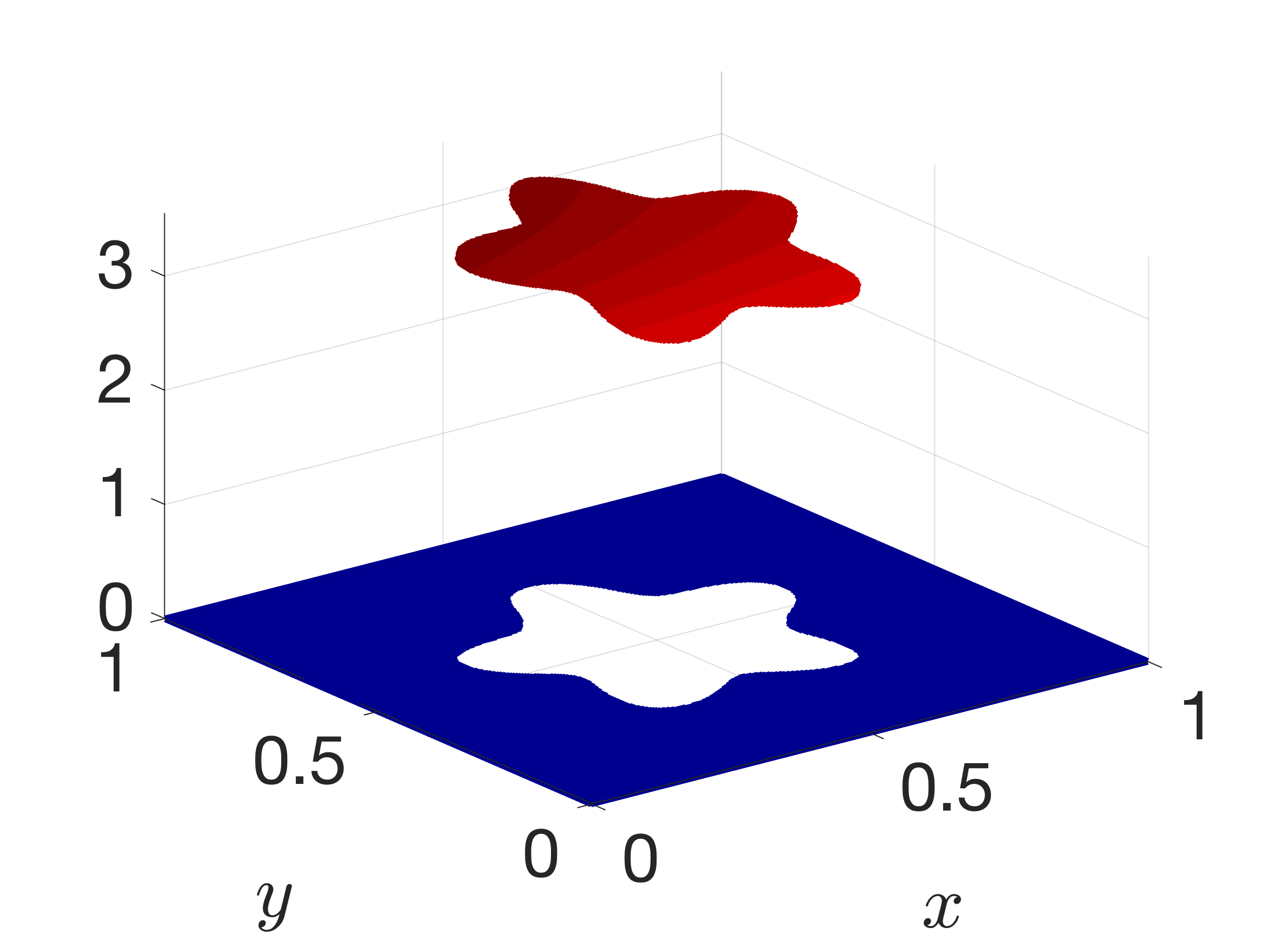}}{\footnotesize$t=0.25$}
	\stackunder[5pt]{\includegraphics[width=2.25in]{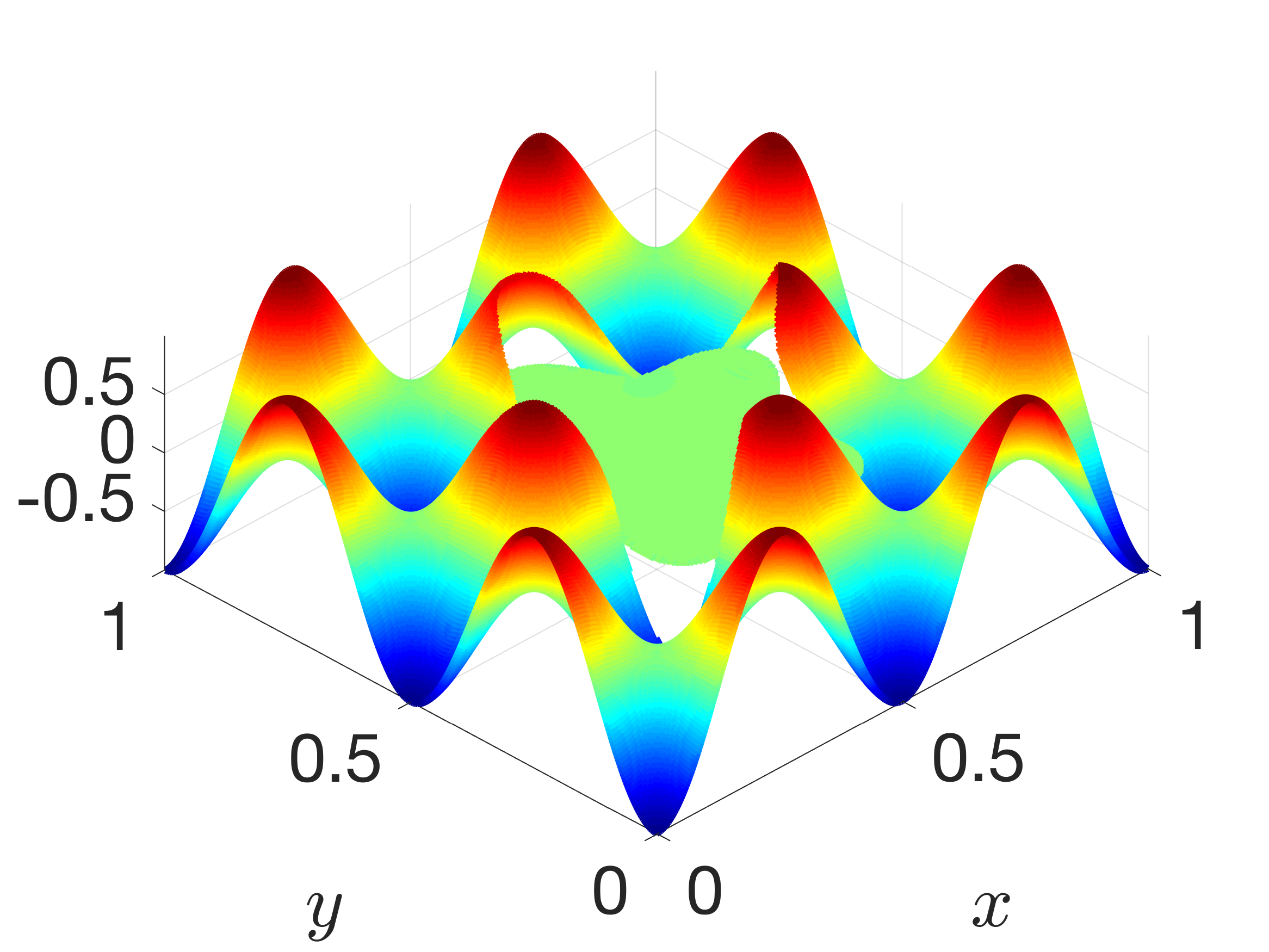}}{\footnotesize$t=0.5$}
	} 
	{\phantom{A}}
	}
	\subfigure[$E_z$]{
	\stackunder[5pt]{
	\stackunder[5pt]{\includegraphics[width=2.25in]{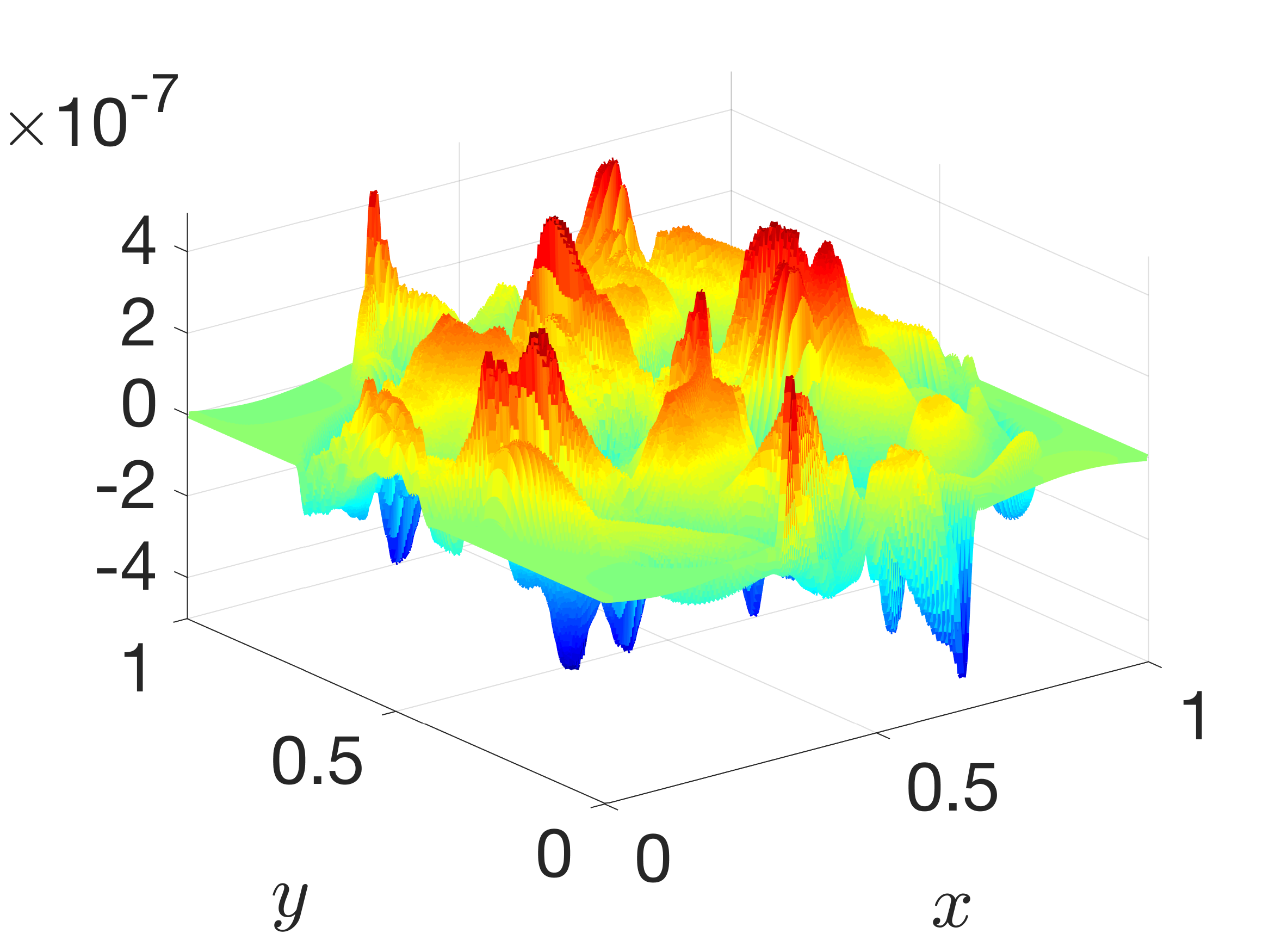}}{\footnotesize$t=0.25$}
	\stackunder[5pt]{\includegraphics[width=2.25in]{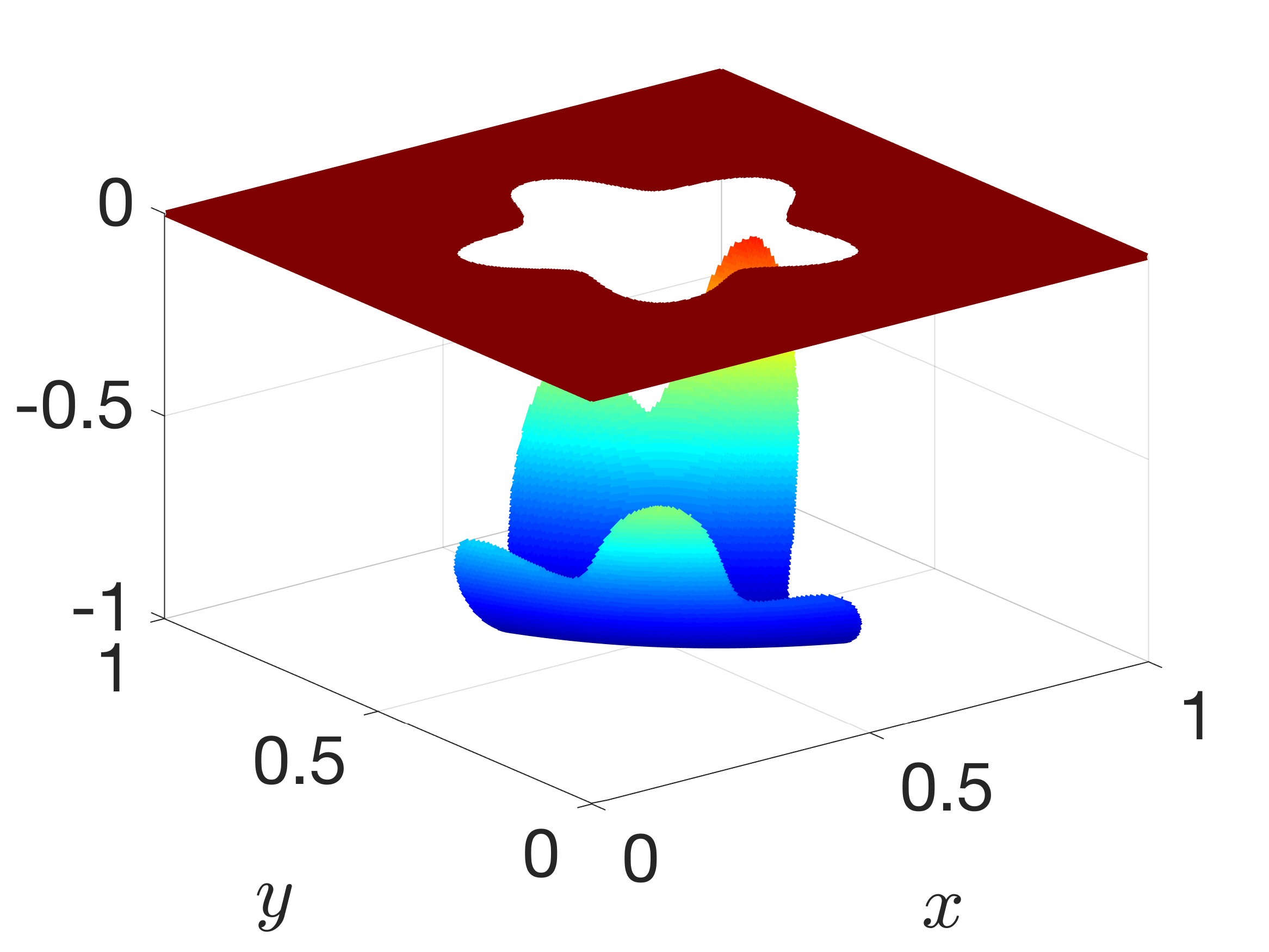}}{\footnotesize$t=0.5$}
	} 
	{\phantom{A}}
	}
       \caption{The components $H_x$, $H_y$ and $E_z$ at two time steps with $h = \tfrac{1}{336}$ and $\Delta t = \tfrac{h}{2}$ using the fourth-order staggered FD scheme with the CFM for the problem with a manufactured solution and the 5-star interface.}
        \label{fig:plotStarInterfaceFields}
\end{figure}

\FloatBarrier
\subsection{3-star interface}
We use the manufactured solution of the circular interface problem.
However, 
	a more complex interface is considered.
The level set function is the same than the 5-star interface but with $\omega=3$, $x_0 = y_0 = 0.55$, $r_0 = 0.25$ and
	$\epsilon = 0.15$.
The interface is illustrated in \cref{fig:triStarInterfaceGeo}.
\cref{fig:convPlotTriStarInterface} illustrates the convergence plots for both schemes using the $L^\infty$-norm and 
	the $L^1$-norm.
\cref{fig:plotTriStarInterfaceFields} shows the magnetic field and the electric field 
	at two different time steps using $h=\tfrac{1}{336}$, 
	and the fourth-order staggered FD scheme with the CFM.
As for previous interfaces, 
	the computed orders of convergence are in agreement with the theory and there is no spurious oscillation within 
	the computed solutions.
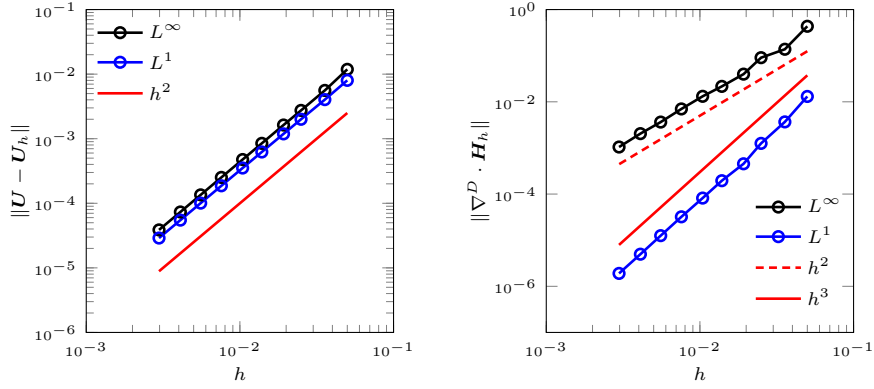
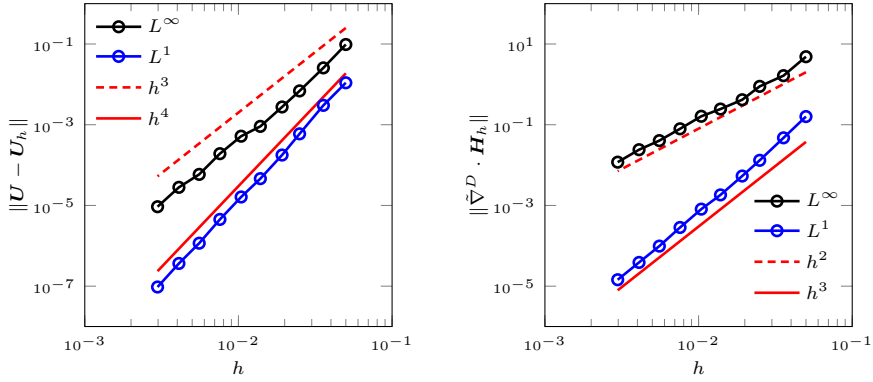
\begin{figure} 
 \centering
  \subfigure[second-order staggered FD scheme]{
		\setlength\figureheight{0.33\linewidth} 
		\setlength\figurewidth{0.33\linewidth} 
		\tikzset{external/export next=false}
%
%
\begin{tikzpicture}

\begin{axis}[%
width=0.951\figurewidth,
height=\figureheight,
at={(0\figurewidth,0\figureheight)},
scale only axis,
xmode=log,
xmin=0.001,
xmax=0.1,
xminorticks=true,
xlabel={\scriptsize$h$},
ymode=log,
ymin=1e-06,
ymax=0.1,
yminorticks=true,
ylabel={\scriptsize$\|\mathbold{U}-\mathbold{U}_h\|$},
axis background/.style={fill=white},
legend style={at={(0.005,0.99)},anchor=north west,legend cell align=left,align=left,draw=white!15!black,draw=none,fill=none},
legend style={font=\scriptsize},
ylabel style={yshift=-5pt},xlabel style={yshift=2.5pt},tick label style={font=\tiny} 
]
\addplot [color=black,line width=1pt,solid,mark=o,mark options={solid}]
  table[row sep=crcr]{%
0.05	0.0118576094074392\\
0.0357142857142857	0.00559567672797879\\
0.025	0.00274284953325326\\
0.0192307692307692	0.00162912992917532\\
0.0138888888888889	0.000847699942177268\\
0.0104166666666667	0.000475879657852167\\
0.00757575757575758	0.000251605913459655\\
0.00555555555555556	0.000135014780802985\\
0.00409836065573771	7.33458185769415e-05\\
0.00297619047619048	3.8717633249512e-05\\
};
\addlegendentry{$L^\infty$};

\addplot [color=blue,line width=1pt,solid,mark=o,mark options={solid}]
  table[row sep=crcr]{%
0.05	0.00801272579515834\\
0.0357142857142857	0.00402975315772282\\
0.025	0.00199937040359872\\
0.0192307692307692	0.00118841616639848\\
0.0138888888888889	0.000623095477250448\\
0.0104166666666667	0.000351716477556265\\
0.00757575757575758	0.000186609370962409\\
0.00555555555555556	0.000100566948913698\\
0.00409836065573771	5.48516606957572e-05\\
0.00297619047619048	2.89628245339375e-05\\
};
\addlegendentry{$L^1$};

\addplot [color=red,line width=1pt,solid]
  table[row sep=crcr]{%
0.05	0.0025\\
0.0357142857142857	0.00127551020408163\\
0.025	0.000625\\
0.0192307692307692	0.000369822485207101\\
0.0138888888888889	0.000192901234567901\\
0.0104166666666667	0.000108506944444444\\
0.00757575757575758	5.73921028466483e-05\\
0.00555555555555556	3.08641975308642e-05\\
0.00409836065573771	1.67965600644988e-05\\
0.00297619047619048	8.85770975056689e-06\\
};
\addlegendentry{$h^2$};

\end{axis}
\end{tikzpicture}%
		\setlength\figureheight{0.33\linewidth} 
		\setlength\figurewidth{0.33\linewidth} 
		\tikzset{external/export next=false}
%
%
\begin{tikzpicture}

\begin{axis}[%
width=0.951\figurewidth,
height=\figureheight,
at={(0\figurewidth,0\figureheight)},
scale only axis,
xmode=log,
xmin=0.001,
xmax=0.1,
xminorticks=true,
xlabel={\scriptsize$h$},
ymode=log,
ymin=1e-07,
ymax=1,
yminorticks=true,
ylabel={\scriptsize$\|\nabla^D\cdot\mathbold{H}_h\|$},
axis background/.style={fill=white},
legend style={at={(0.65,0.45)},anchor=north west,legend cell align=left,align=left,draw=white!15!black,draw=none,fill=none},
legend style={font=\scriptsize},
ylabel style={yshift=-5pt},xlabel style={yshift=2.5pt},tick label style={font=\tiny} 
]
\addplot [color=black,line width=1pt,solid,mark=o,mark options={solid}]
  table[row sep=crcr]{%
0.05	0.434149844201329\\
0.0357142857142857	0.137808971545013\\
0.025	0.0905448906266741\\
0.0192307692307692	0.0398155131831004\\
0.0138888888888889	0.0215994317780712\\
0.0104166666666667	0.0131540499836547\\
0.00757575757575758	0.00700459909040774\\
0.00555555555555556	0.0036356043511887\\
0.00409836065573771	0.00204138923481878\\
0.00297619047619048	0.00104651699939495\\
};
\addlegendentry{$L^\infty$};

\addplot [color=blue,line width=1pt,solid,mark=o,mark options={solid}]
  table[row sep=crcr]{%
0.05	0.0130696501697566\\
0.0357142857142857	0.00366861851807057\\
0.025	0.00125461146440794\\
0.0192307692307692	0.000453234103350705\\
0.0138888888888889	0.000194818046733771\\
0.0104166666666667	8.18286175996446e-05\\
0.00757575757575758	3.21065634802953e-05\\
0.00555555555555556	1.25412673710266e-05\\
0.00409836065573771	4.96570368714429e-06\\
0.00297619047619048	1.90383129572669e-06\\
};
\addlegendentry{$L^1$};

\addplot [color=red,line width=1pt,densely dashed]
  table[row sep=crcr]{%
0.05	0.125\\
0.0357142857142857	0.0637755102040816\\
0.025	0.03125\\
0.0192307692307692	0.018491124260355\\
0.0138888888888889	0.00964506172839506\\
0.0104166666666667	0.00542534722222222\\
0.00757575757575758	0.00286960514233242\\
0.00555555555555556	0.00154320987654321\\
0.00409836065573771	0.00083982800322494\\
0.00297619047619048	0.000442885487528345\\
};
\addlegendentry{$h^2$};

\addplot [color=red,line width=1pt,solid]
  table[row sep=crcr]{%
0.05	0.0375\\
0.0357142857142857	0.0136661807580175\\
0.025	0.0046875\\
0.0192307692307692	0.0021335912608102\\
0.0138888888888889	0.000803755144032922\\
0.0104166666666667	0.000339084201388889\\
0.00757575757575758	0.000130436597378746\\
0.00555555555555556	5.1440329218107e-05\\
0.00409836065573771	2.06515082760231e-05\\
0.00297619047619048	7.90866942014901e-06\\
};
\addlegendentry{$h^3$};

\end{axis}

\end{tikzpicture}%
		} 
  \subfigure[fourth-order staggered FD scheme]{
		\setlength\figureheight{0.33\linewidth} 
		\setlength\figurewidth{0.33\linewidth} 
		\tikzset{external/export next=false}
%
%
\begin{tikzpicture}

\begin{axis}[%
width=0.951\figurewidth,
height=\figureheight,
at={(0\figurewidth,0\figureheight)},
scale only axis,
xmode=log,
xmin=0.001,
xmax=0.1,
xminorticks=true,
xlabel={\scriptsize$h$},
ymode=log,
ymin=1e-08,
ymax=1,
yminorticks=true,
ylabel={\scriptsize$\|\mathbold{U}-\mathbold{U}_h\|$},
axis background/.style={fill=white},
legend style={at={(0.005,0.99)},anchor=north west,legend cell align=left,align=left,draw=white!15!black,draw=none,fill=none},
legend style={font=\scriptsize},
ylabel style={yshift=-5pt},xlabel style={yshift=2.5pt},tick label style={font=\tiny} 
]
\addplot [color=black,line width=1pt,solid,mark=o,mark options={solid}]
  table[row sep=crcr]{%
0.05	0.0972360896646167\\
0.0357142857142857	0.0256080067709567\\
0.025	0.00688178370793626\\
0.0192307692307692	0.00275534371235602\\
0.0138888888888889	0.000909949231679086\\
0.0104166666666667	0.000519734514534513\\
0.00757575757575758	0.000191875920187065\\
0.00555555555555556	5.88368111862358e-05\\
0.00409836065573771	2.81219892865998e-05\\
0.00297619047619048	9.3045795721558e-06\\
};
\addlegendentry{$L^\infty$};

\addplot [color=blue,line width=1pt,solid,mark=o,mark options={solid}]
  table[row sep=crcr]{%
0.05	0.0109542103258836\\
0.0357142857142857	0.0030200598748306\\
0.025	0.000594680374737356\\
0.0192307692307692	0.000177403292584226\\
0.0138888888888889	4.61307228358092e-05\\
0.0104166666666667	1.62162673009344e-05\\
0.00757575757575758	4.50567838714874e-06\\
0.00555555555555556	1.16016301468282e-06\\
0.00409836065573771	3.65176506430162e-07\\
0.00297619047619048	9.53713626981072e-08\\
};
\addlegendentry{$L^1$};

\addplot [color=red,line width=1pt,densely dashed]
  table[row sep=crcr]{%
0.05	0.25\\
0.0357142857142857	0.0911078717201166\\
0.025	0.03125\\
0.0192307692307692	0.0142239417387346\\
0.0138888888888889	0.00535836762688614\\
0.0104166666666667	0.00226056134259259\\
0.00757575757575758	0.000869577315858308\\
0.00555555555555556	0.000342935528120713\\
0.00409836065573771	0.000137676721840154\\
0.00297619047619048	5.27244628009934e-05\\
};
\addlegendentry{$h^3$};

\addplot [color=red,line width=1pt,solid]
  table[row sep=crcr]{%
0.05	0.01875\\
0.0357142857142857	0.0048807788421491\\
0.025	0.001171875\\
0.0192307692307692	0.000410306011694268\\
0.0138888888888889	0.000111632658893461\\
0.0104166666666667	3.53212709780093e-05\\
0.00757575757575758	9.88156040748077e-06\\
0.00555555555555556	2.85779606767261e-06\\
0.00409836065573771	8.46373290000947e-07\\
0.00297619047619048	2.35377066075863e-07\\
};
\addlegendentry{$h^4$};

\end{axis}
\end{tikzpicture}%
		\setlength\figureheight{0.33\linewidth} 
		\setlength\figurewidth{0.33\linewidth} 
		\tikzset{external/export next=false}
%
%
\begin{tikzpicture}

\begin{axis}[%
width=0.951\figurewidth,
height=\figureheight,
at={(0\figurewidth,0\figureheight)},
scale only axis,
xmode=log,
xmin=0.001,
xmax=0.1,
xminorticks=true,
xlabel={\scriptsize$h$},
ymode=log,
ymin=1e-06,
ymax=100,
yminorticks=true,
ylabel={\scriptsize$\|\tilde{\nabla}^D\cdot\mathbold{H}_h\|$},
axis background/.style={fill=white},
legend style={at={(0.65,0.45)},anchor=north west,legend cell align=left,align=left,draw=white!15!black,draw=none,fill=none},
legend style={font=\scriptsize},
ylabel style={yshift=-5pt},xlabel style={yshift=2.5pt},tick label style={font=\tiny} 
]
\addplot [color=black,line width=1pt,solid,mark=o,mark options={solid}]
  table[row sep=crcr]{%
0.05	4.84162004452322\\
0.0357142857142857	1.64374819144609\\
0.025	0.892423760583952\\
0.0192307692307692	0.415828725233098\\
0.0138888888888889	0.244581888656228\\
0.0104166666666667	0.161926732453821\\
0.00757575757575758	0.0792726438107252\\
0.00555555555555556	0.040383873397559\\
0.00409836065573771	0.0240524830785489\\
0.00297619047619048	0.0118167439920853\\
};
\addlegendentry{$L^\infty$};

\addplot [color=blue,line width=1pt,solid,mark=o,mark options={solid}]
  table[row sep=crcr]{%
0.05	0.159414487051912\\
0.0357142857142857	0.0473492474329155\\
0.025	0.0131277993999985\\
0.0192307692307692	0.00535677697868786\\
0.0138888888888889	0.00182556912799519\\
0.0104166666666667	0.000809824723517902\\
0.00757575757575758	0.000284409046179301\\
0.00555555555555556	9.83132391354671e-05\\
0.00409836065573771	3.87323224081587e-05\\
0.00297619047619048	1.44362248122291e-05\\
};
\addlegendentry{$L^1$};

\addplot [color=red,line width=1pt,densely dashed]
  table[row sep=crcr]{%
0.05	2\\
0.0357142857142857	1.02040816326531\\
0.025	0.5\\
0.0192307692307692	0.29585798816568\\
0.0138888888888889	0.154320987654321\\
0.0104166666666667	0.0868055555555556\\
0.00757575757575758	0.0459136822773186\\
0.00555555555555556	0.0246913580246914\\
0.00409836065573771	0.013437248051599\\
0.00297619047619048	0.00708616780045351\\
};
\addlegendentry{$h^2$};

\addplot [color=red,line width=1pt,solid]
  table[row sep=crcr]{%
0.05	0.0375\\
0.0357142857142857	0.0136661807580175\\
0.025	0.0046875\\
0.0192307692307692	0.0021335912608102\\
0.0138888888888889	0.000803755144032922\\
0.0104166666666667	0.000339084201388889\\
0.00757575757575758	0.000130436597378746\\
0.00555555555555556	5.1440329218107e-05\\
0.00409836065573771	2.06515082760231e-05\\
0.00297619047619048	7.90866942014901e-06\\
};
\addlegendentry{$h^3$};

\end{axis}

\end{tikzpicture}%
		} 
  \caption{Convergence plots for the problem with a manufactured solution and the 3-star interface using fourth-order approximations of correction functions, and either the second-order or fourth-order staggered FD scheme. It is recalled that $\mathbold{U} = [H_x,H_y,E_z]^T$.}
   \label{fig:convPlotTriStarInterface}
\end{figure}
\begin{figure}     
	\centering
	\subfigure[$H_x$]{
	\stackunder[5pt]{
	\stackunder[5pt]{\includegraphics[width=2.25in]{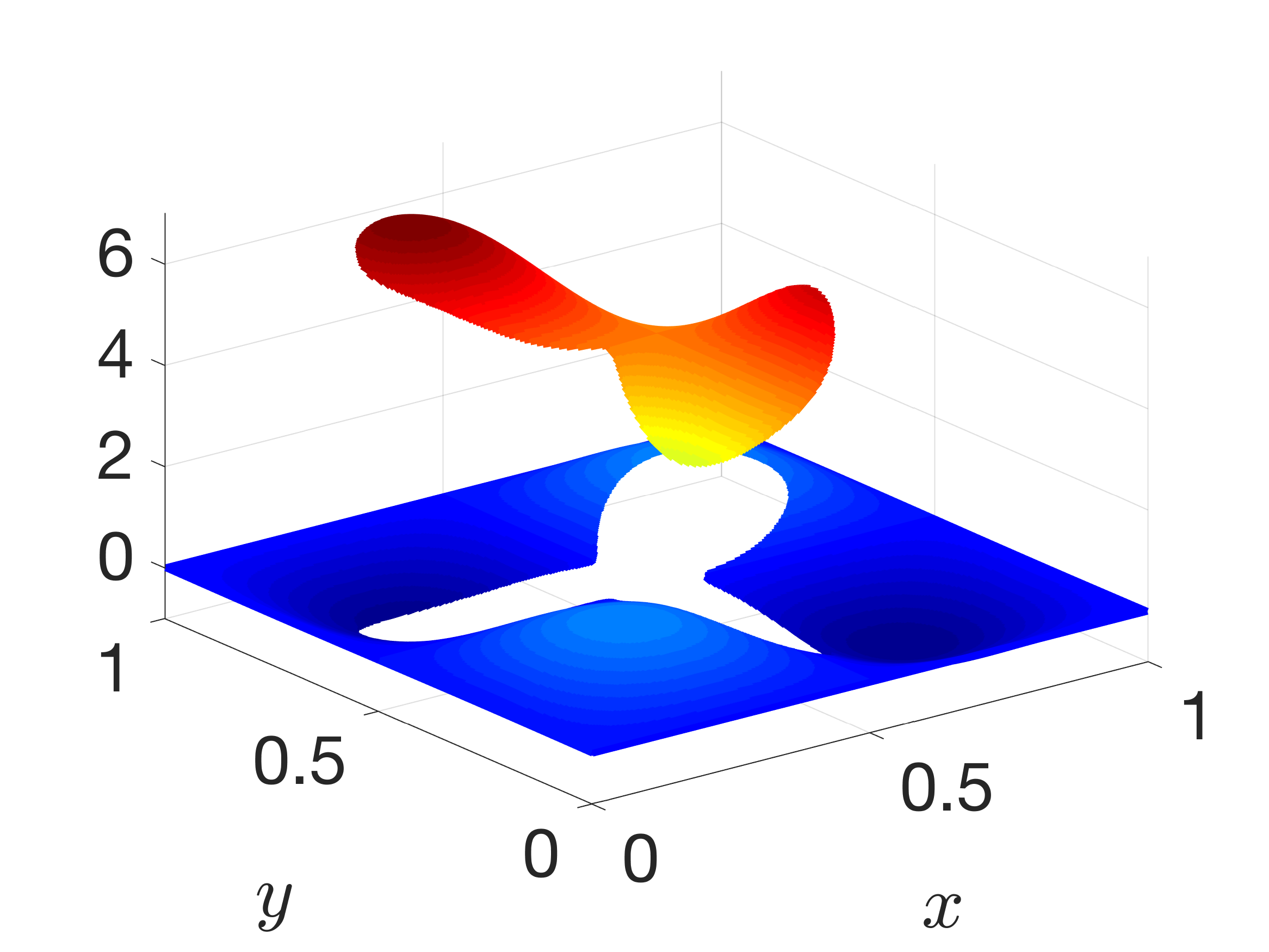}}{\footnotesize$t=0.25$}
	\stackunder[5pt]{	\includegraphics[width=2.25in]{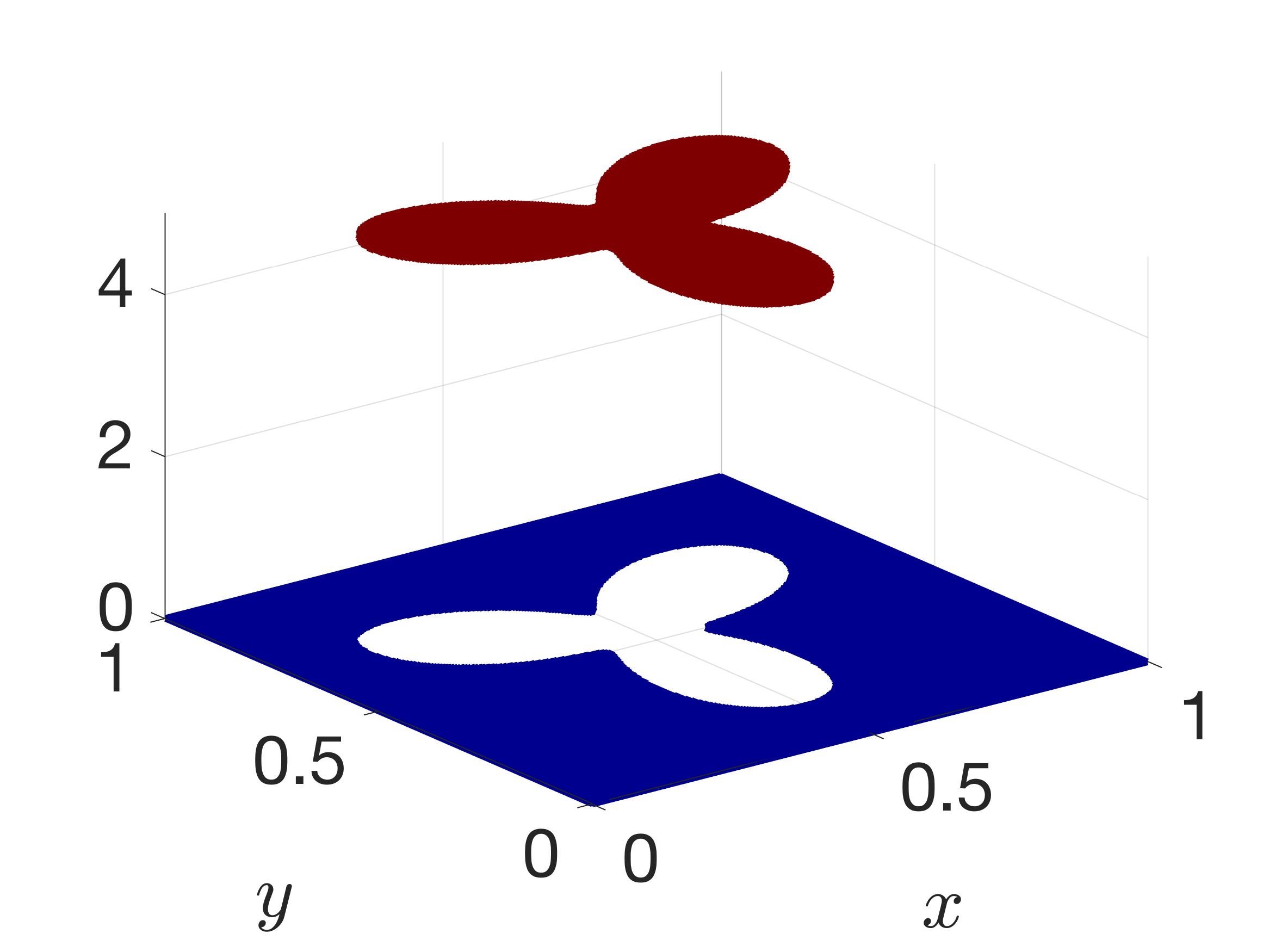}}{\footnotesize$t=0.5$}
	} 
	{\phantom{A}}
	}
\subfigure[$H_y$]{
	\stackunder[5pt]{
	\stackunder[5pt]{\includegraphics[width=2.25in]{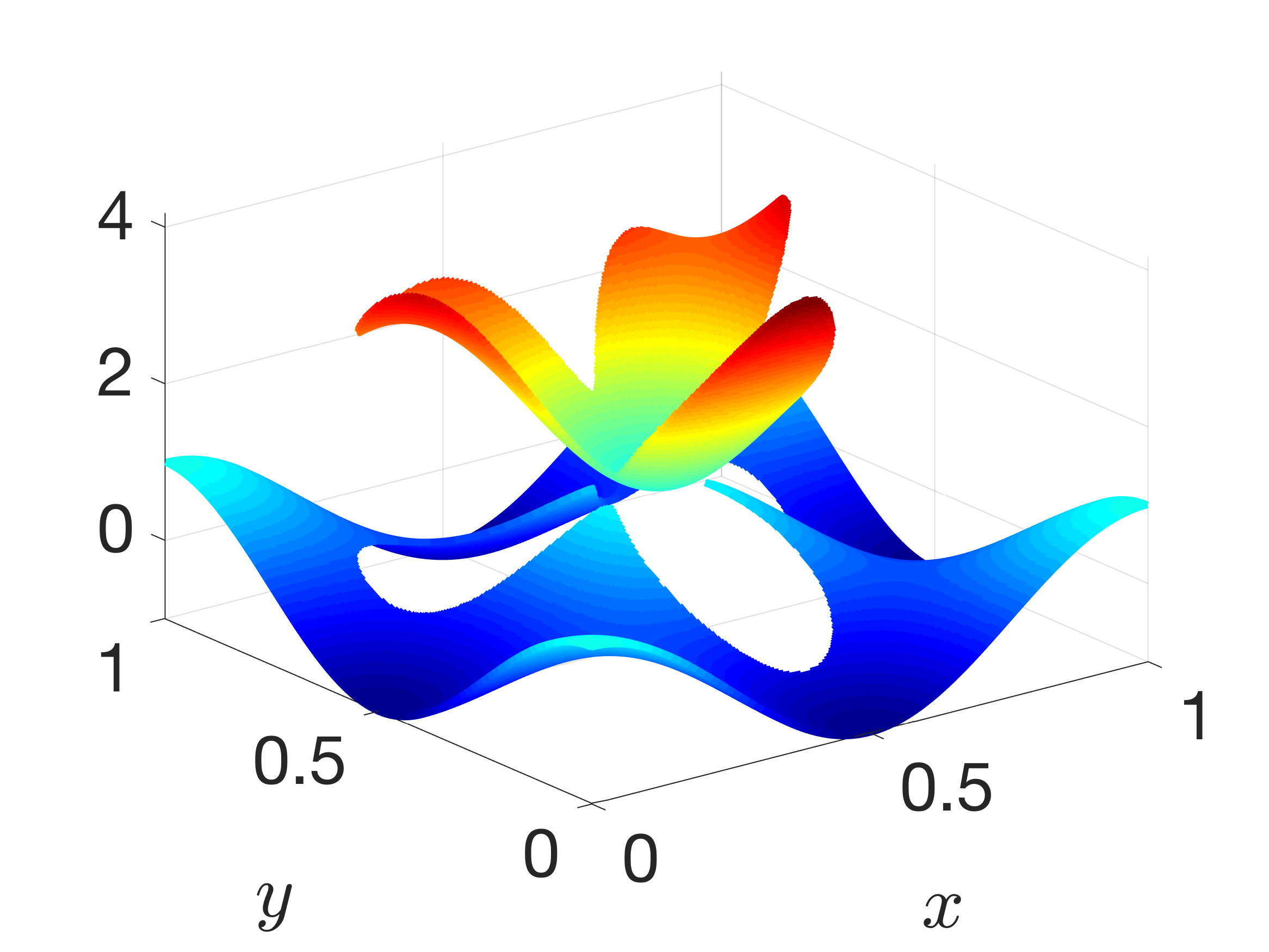}}{\footnotesize$t=0.25$}
	\stackunder[5pt]{\includegraphics[width=2.25in]{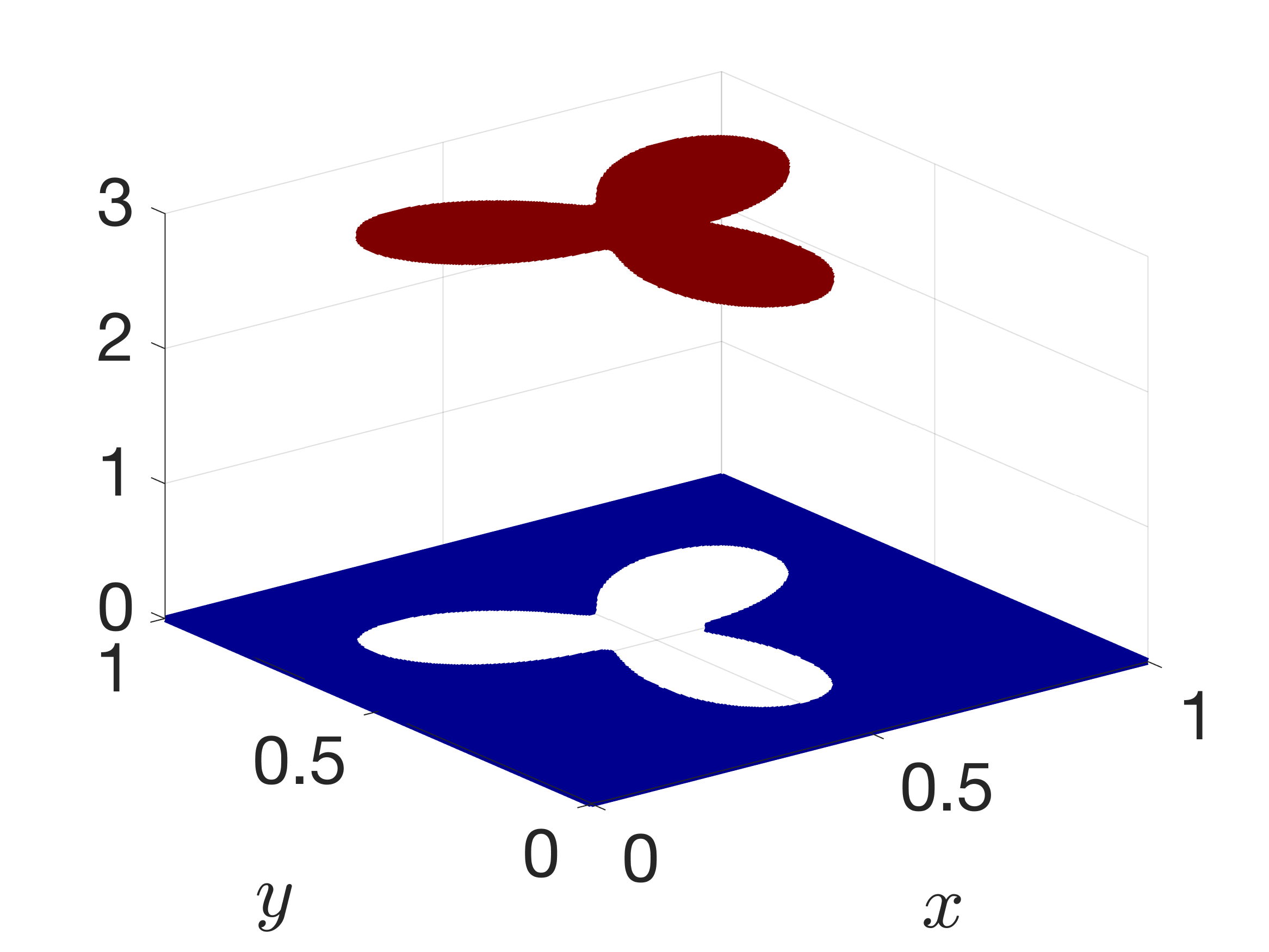}}{\footnotesize$t=0.5$}
	} 
	{\phantom{A}}
	}
	\subfigure[$E_z$]{
	\stackunder[5pt]{
	\stackunder[5pt]{\includegraphics[width=2.25in]{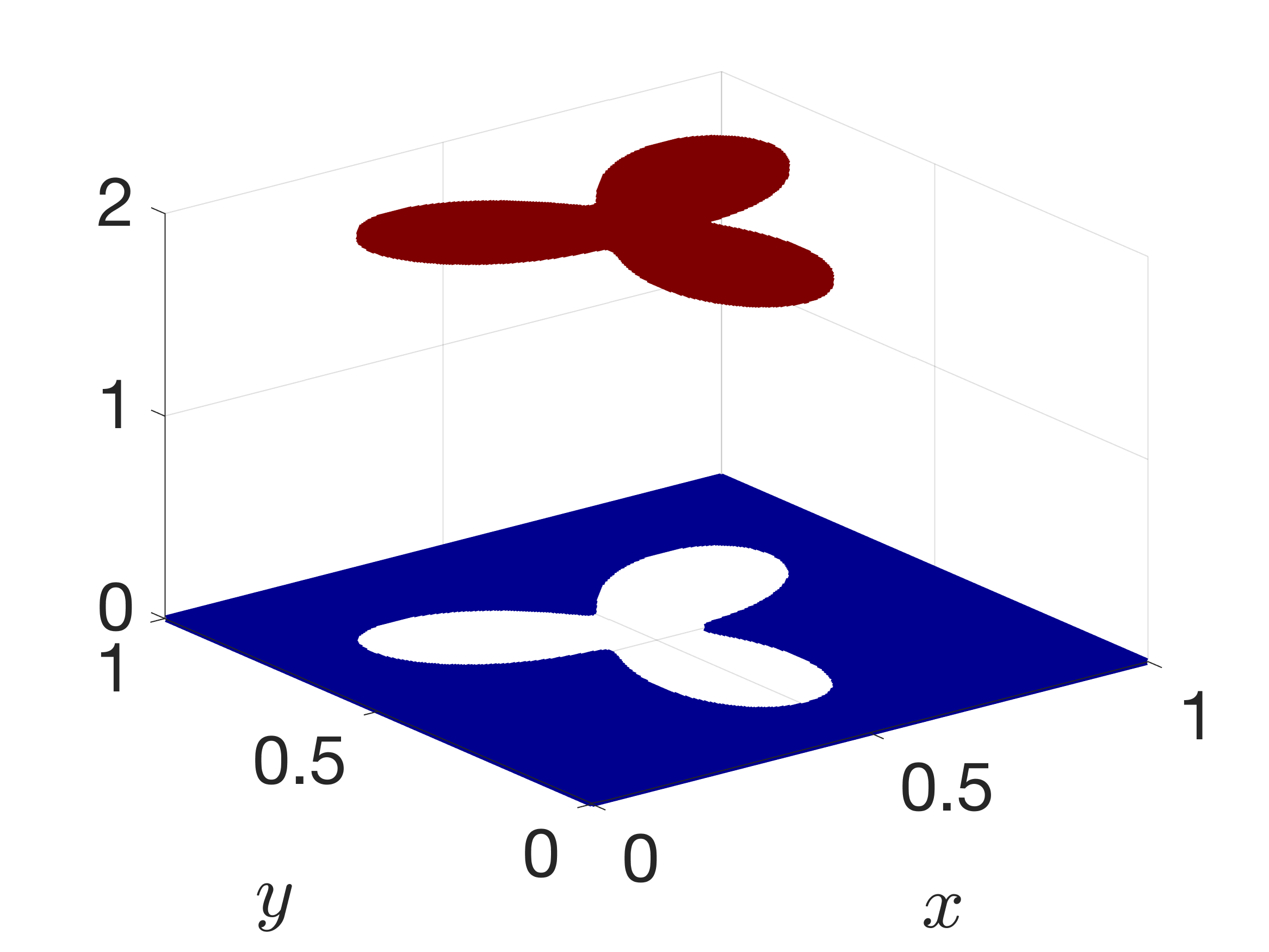}}{\footnotesize$t=0.25$}
	\stackunder[5pt]{\includegraphics[width=2.25in]{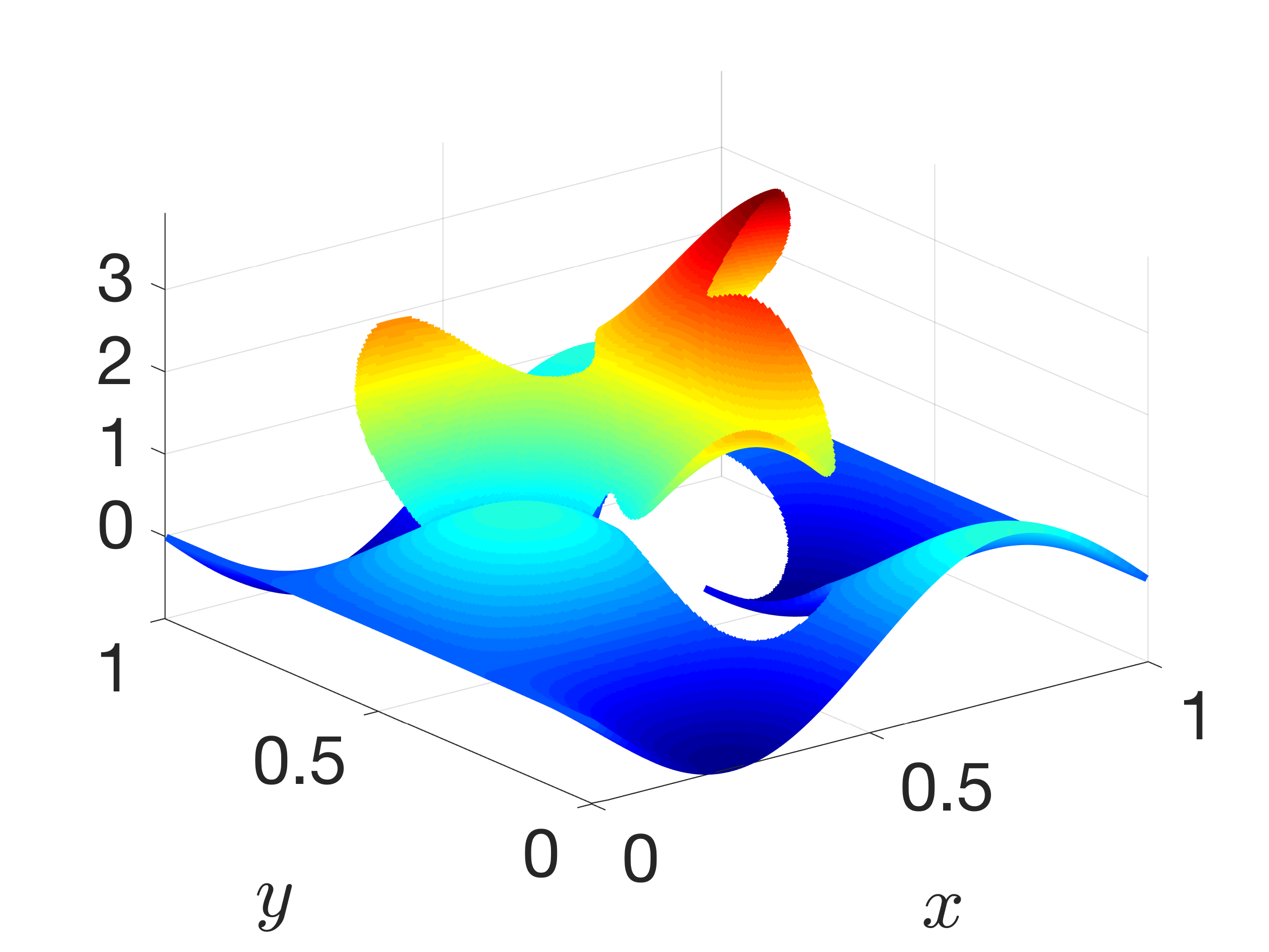}}{\footnotesize$t=0.5$}
	} 
	{\phantom{A}}
	}
       \caption{The components $H_x$, $H_y$ and $E_z$ at two time steps with $h = \tfrac{1}{336}$ and $\Delta t = \tfrac{h}{2}$ using the fourth-order staggered FD scheme with the CFM for the problem with a manufactured solution and the 3-star interface.}
       \label{fig:plotTriStarInterfaceFields}
\end{figure}

\FloatBarrier

\subsection{A remark on non-smooth interface}
 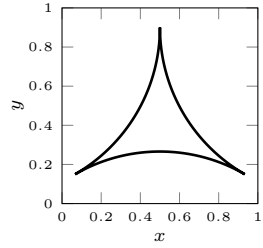
\begin{figure}
 \centering
		\setlength\figureheight{0.2\linewidth} 
		\setlength\figurewidth{0.2\linewidth} 
		\tikzset{external/export next=false}
%
%
\begin{tikzpicture}

\begin{axis}[%
width=\figurewidth,
height=\figureheight,
at={(0\figurewidth,0\figureheight)},
scale only axis,
xmin=0,
xmax=1,
xminorticks=true,
xlabel={\scriptsize$x$},
ymin=0,
ymax=1,
yminorticks=true,
ylabel={\scriptsize $y$},
axis background/.style={fill=white},
legend style={at={(0.01,0.99)},anchor=north west,legend cell align=left,align=left,draw=white!15!black,draw=none,fill=none},
legend style={font=\scriptsize},
ylabel style={yshift=-5pt},xlabel style={yshift=2.5pt},tick label style={font=\tiny} 
]
\addplot [color=black,line width=1pt,solid]
  table[row sep=crcr]{%
0.5	0.9\\
0.50043600708589	0.87252280358109\\
0.501743589321385	0.845073274369832\\
0.503921430082012	0.817679051715344\\
0.506967336466732	0.790367719277731\\
0.510878241506007	0.763166777253675\\
0.515650207249975	0.736103614686068\\
0.521278428733635	0.709205481885571\\
0.527757238815036	0.682499462991857\\
0.535080113881612	0.656012448702179\\
0.543239680418903	0.629771109194718\\
0.552227722435065	0.603801867273976\\
0.56203518973367	0.578130871765249\\
0.572652207026491	0.552783971184976\\
0.584068083877082	0.527786687713478\\
0.596271325465136	0.503164191496281\\
0.609249644160798	0.478941275299916\\
0.622989971897267	0.455142329547704\\
0.637478473329229	0.431791317760667\\
0.652700559763882	0.408911752428295\\
0.668640903850505	0.38652667133346\\
0.685283455013809	0.364658614355321\\
0.7026114556155	0.343329600773581\\
0.720607457827806	0.322561107096932\\
0.739253341201953	0.302374045438038\\
0.758530330913928	0.282788742456804\\
0.778419016669131	0.263824918893159\\
0.798899372246899	0.245501669709936\\
0.81995077566521	0.227837444865867\\
0.841552029945273	0.210850030738038\\
0.863681384455088	0.194556532212515\\
0.886316556810488	0.178973355461171\\
0.909434755311614	0.164116191422066\\
0.933012701892219	0.15\\
0.908998748225723	0.163361004996844\\
0.884572967489104	0.175953370168997\\
0.859759954373077	0.187764416072141\\
0.834584693478541	0.198782249984231\\
0.809072534159204	0.208995777880458\\
0.783249164996924	0.218394715603996\\
0.757140587935496	0.22696959922127\\
0.730773092098892	0.234711794551339\\
0.704173227320341	0.241613505859784\\
0.677367777408902	0.24766778370835\\
0.650383733180436	0.252868531952443\\
0.623248265280139	0.25721051387943\\
0.595988696824013	0.260689357481564\\
0.5686324758868	0.263301559858227\\
0.541207147864094	0.265044490743052\\
0.513740327736469	0.265916395152381\\
0.486259672263531	0.265916395152381\\
0.458792852135906	0.265044490743052\\
0.4313675241132	0.263301559858227\\
0.404011303175987	0.260689357481564\\
0.376751734719861	0.25721051387943\\
0.349616266819565	0.252868531952443\\
0.322632222591098	0.24766778370835\\
0.295826772679659	0.241613505859784\\
0.269226907901108	0.234711794551339\\
0.242859412064504	0.22696959922127\\
0.216750835003076	0.218394715603996\\
0.190927465840797	0.208995777880458\\
0.165415306521459	0.198782249984231\\
0.140240045626923	0.187764416072141\\
0.115427032510897	0.175953370168997\\
0.0910012517742768	0.163361004996844\\
0.0669872981077808	0.15\\
0.0905652446883862	0.164116191422066\\
0.113683443189512	0.178973355461171\\
0.136318615544912	0.194556532212514\\
0.158447970054727	0.210850030738038\\
0.18004922433479	0.227837444865867\\
0.201100627753101	0.245501669709936\\
0.221580983330869	0.263824918893159\\
0.241469669086072	0.282788742456805\\
0.260746658798048	0.302374045438038\\
0.279392542172194	0.322561107096932\\
0.2973885443845	0.343329600773581\\
0.314716544986191	0.364658614355321\\
0.331359096149495	0.38652667133346\\
0.347299440236118	0.408911752428295\\
0.362521526670771	0.431791317760667\\
0.377010028102733	0.455142329547704\\
0.390750355839202	0.478941275299915\\
0.403728674534864	0.503164191496281\\
0.415931916122918	0.527786687713478\\
0.427347792973509	0.552783971184977\\
0.43796481026633	0.578130871765248\\
0.447772277564935	0.603801867273976\\
0.456760319581097	0.629771109194719\\
0.464919886118388	0.656012448702179\\
0.472242761184964	0.682499462991857\\
0.478721571266365	0.709205481885571\\
0.484349792750025	0.736103614686069\\
0.489121758493993	0.763166777253675\\
0.493032663533268	0.790367719277731\\
0.496078569917988	0.817679051715344\\
0.498256410678615	0.845073274369832\\
0.499563992914109	0.87252280358109\\
0.5	0.9\\
};
\end{axis}
\end{tikzpicture}%
  \caption{A non-smooth interface.}
  \label{fig:sharpInterfaceGeo}
\end{figure}
This subsection studies the robustness of the proposed treatment of interface conditions by considering 
	a non-smooth interface illustrated in \cref{fig:sharpInterfaceGeo}.
This interface is built using three circles of radius $r=\tfrac{\sqrt{3}}{2}$ centered at $(0.5+r, 0.9)$,
	$(0.5-r,0.9)$ and $(0.5,-0.6)$. 
We note that the normal $\hat{\mathbold{n}}$ might not be well defined at the cusps.
We use the same manufactured solution than the circular interface problem.
\cref{fig:convPlotSharpInterface} illustrates the convergence plots for the fourth-order staggered FD scheme 
	with the CFM using the $L^\infty$-norm and the $L^1$-norm.
Using $L^1$-norm,
	$H_x$, $H_y$ and $E_z$ converge to fourth-order 
	while a third-order convergence is obtained for the divergence of the magnetic field.
Even though we use smooth manufactured solutions in each subdomain, 
	we highlight that this kind of solutions is misleading for interfaces with cusps or corners.
Indeed, 
	solutions of Maxwell interface problems with such interfaces have a singular part \cite{Costabel1999,Assous2000}, 
	which is not treated in this work. 
While it is unclear whether the computed solutions in \cref{fig:plotSharpInterfaceFields} represent accurately 
	the actual solution (regular and singular parts). 
It is interesting to note that the proposed numerical approach is robust, 
	converges to the prescribed order and provides solutions that are devoid of spurious oscillations.
It is therefore clear that much work is required to assess whether the numerical approach presented in this paper 
	can be used or modified to compute solutions of problems with non-smooth interfaces. 
\begin{figure}
 \centering
		\setlength\figureheight{0.33\linewidth} 
		\setlength\figurewidth{0.33\linewidth} 
		\tikzset{external/export next=false}
%
%
\begin{tikzpicture}

\begin{axis}[%
width=0.951\figurewidth,
height=\figureheight,
at={(0\figurewidth,0\figureheight)},
scale only axis,
xmode=log,
xmin=0.001,
xmax=0.1,
xminorticks=true,
xlabel={\scriptsize$h$},
ymode=log,
ymin=1e-08,
ymax=1,
yminorticks=true,
ylabel={\scriptsize$\|\mathbold{U}-\mathbold{U}_h\|$},
axis background/.style={fill=white},
legend style={at={(0.005,0.99)},anchor=north west,legend cell align=left,align=left,draw=white!15!black,draw=none,fill=none},
legend style={font=\scriptsize},
ylabel style={yshift=-5pt},xlabel style={yshift=2.5pt},tick label style={font=\tiny} 
]
\addplot [color=black,line width=1pt,solid,mark=o,mark options={solid}]
  table[row sep=crcr]{%
0.05	0.141125384328186\\
0.0357142857142857	0.0704970010051316\\
0.025	0.0129762950258199\\
0.0192307692307692	0.00760348582011172\\
0.0138888888888889	0.0216353618735043\\
0.0104166666666667	0.000533498912821173\\
0.00757575757575758	0.000345382377281211\\
0.00555555555555556	9.4870266035052e-05\\
0.00409836065573771	6.44674128151834e-05\\
0.00297619047619048	2.08817219000679e-05\\
};
\addlegendentry{$L^\infty$};

\addplot [color=blue,line width=1pt,solid,mark=o,mark options={solid}]
  table[row sep=crcr]{%
0.05	0.042528783391363\\
0.0357142857142857	0.015454831434736\\
0.025	0.000916263181377296\\
0.0192307692307692	0.000282144215195102\\
0.0138888888888889	0.000164343751200037\\
0.0104166666666667	2.15909126291379e-05\\
0.00757575757575758	7.16897232137316e-06\\
0.00555555555555556	1.99399539989871e-06\\
0.00409836065573771	6.19916070092365e-07\\
0.00297619047619048	1.57203329784787e-07\\
};
\addlegendentry{$L^1$};

\addplot [color=red,line width=1pt,densely dashed]
  table[row sep=crcr]{%
0.05	0.0375\\
0.0357142857142857	0.0136661807580175\\
0.025	0.0046875\\
0.0192307692307692	0.0021335912608102\\
0.0138888888888889	0.000803755144032922\\
0.0104166666666667	0.000339084201388889\\
0.00757575757575758	0.000130436597378746\\
0.00555555555555556	5.1440329218107e-05\\
0.00409836065573771	2.06515082760231e-05\\
0.00297619047619048	7.90866942014901e-06\\
};
\addlegendentry{$h^3$};

\addplot [color=red,line width=1pt,solid]
  table[row sep=crcr]{%
0.05	0.003125\\
0.0357142857142857	0.000813463140358184\\
0.025	0.0001953125\\
0.0192307692307692	6.83843352823781e-05\\
0.0138888888888889	1.86054431489102e-05\\
0.0104166666666667	5.88687849633488e-06\\
0.00757575757575758	1.64692673458013e-06\\
0.00555555555555556	4.76299344612102e-07\\
0.00409836065573771	1.41062215000158e-07\\
0.00297619047619048	3.92295110126439e-08\\
};
\addlegendentry{$h^4$};

\end{axis}
\end{tikzpicture}%
		\setlength\figureheight{0.33\linewidth} 
		\setlength\figurewidth{0.33\linewidth} 
		\tikzset{external/export next=false}
%
%
\begin{tikzpicture}

\begin{axis}[%
width=0.951\figurewidth,
height=\figureheight,
at={(0\figurewidth,0\figureheight)},
scale only axis,
xmode=log,
xmin=0.001,
xmax=0.1,
xminorticks=true,
xlabel={\scriptsize$h$},
ymode=log,
ymin=1e-06,
ymax=10,
yminorticks=true,
ylabel={\scriptsize$\|\tilde{\nabla}^D\cdot\mathbold{H}_h\|$},
axis background/.style={fill=white},
legend style={at={(0.65,0.45)},anchor=north west,legend cell align=left,align=left,draw=white!15!black,draw=none,fill=none},
legend style={font=\scriptsize},
ylabel style={yshift=-5pt},xlabel style={yshift=2.5pt},tick label style={font=\tiny} 
]
\addplot [color=black,line width=1pt,solid,mark=o,mark options={solid}]
  table[row sep=crcr]{%
0.05	6.42445175140259\\
0.0357142857142857	5.50618116808948\\
0.025	1.21321747596235\\
0.0192307692307692	0.909945283625262\\
0.0138888888888889	5.40012427705312\\
0.0104166666666667	0.183170112374739\\
0.00757575757575758	0.106996954498534\\
0.00555555555555556	0.0534054109625686\\
0.00409836065573771	0.035874857299973\\
0.00297619047619048	0.0211572201908439\\
};
\addlegendentry{$L^\infty$};

\addplot [color=blue,line width=1pt,solid,mark=o,mark options={solid}]
  table[row sep=crcr]{%
0.05	0.155830924330879\\
0.0357142857142857	0.0464381482005214\\
0.025	0.0165293218736576\\
0.0192307692307692	0.00690280822317113\\
0.0138888888888889	0.00417162700686198\\
0.0104166666666667	0.000738166341117011\\
0.00757575757575758	0.000301791387126868\\
0.00555555555555556	0.000114096245260521\\
0.00409836065573771	3.88097811188855e-05\\
0.00297619047619048	1.51010763899193e-05\\
};
\addlegendentry{$L^1$};

\addplot [color=red,line width=1pt,densely dashed]
  table[row sep=crcr]{%
0.05	2.25\\
0.0357142857142857	1.14795918367347\\
0.025	0.5625\\
0.0192307692307692	0.332840236686391\\
0.0138888888888889	0.173611111111111\\
0.0104166666666667	0.09765625\\
0.00757575757575758	0.0516528925619835\\
0.00555555555555556	0.0277777777777778\\
0.00409836065573771	0.0151169040580489\\
0.00297619047619048	0.0079719387755102\\
};
\addlegendentry{$h^2$};

\addplot [color=red,line width=1pt,solid]
  table[row sep=crcr]{%
0.05	0.03125\\
0.0357142857142857	0.0113884839650146\\
0.025	0.00390625\\
0.0192307692307692	0.00177799271734183\\
0.0138888888888889	0.000669795953360768\\
0.0104166666666667	0.000282570167824074\\
0.00757575757575758	0.000108697164482288\\
0.00555555555555556	4.28669410150892e-05\\
0.00409836065573771	1.72095902300193e-05\\
0.00297619047619048	6.59055785012418e-06\\
};
\addlegendentry{$h^3$};

\end{axis}
\end{tikzpicture}%
  \caption{Convergence plots for the problem with a manufactured solution and a non-smooth interface using fourth-order approximations of correction functions and the fourth-order staggered FD scheme. It is recalled that $\mathbold{U} = [H_x,H_y,E_z]^T$.}
  \label{fig:convPlotSharpInterface}
\end{figure}
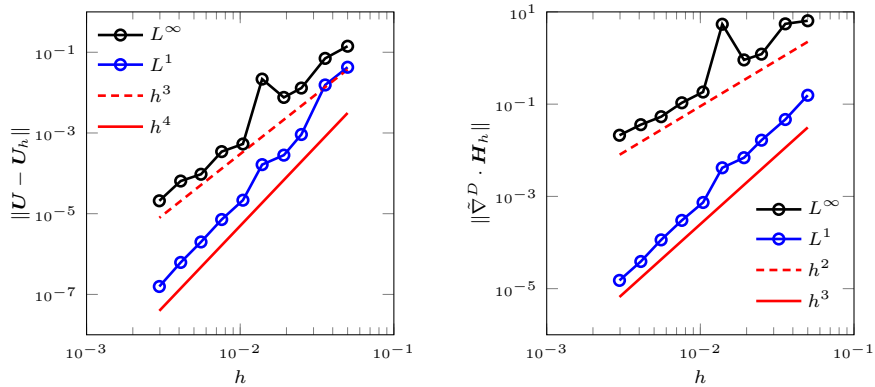
\begin{figure}
	\centering
	\subfigure[$H_x$]{
	\includegraphics[width=2.2in]{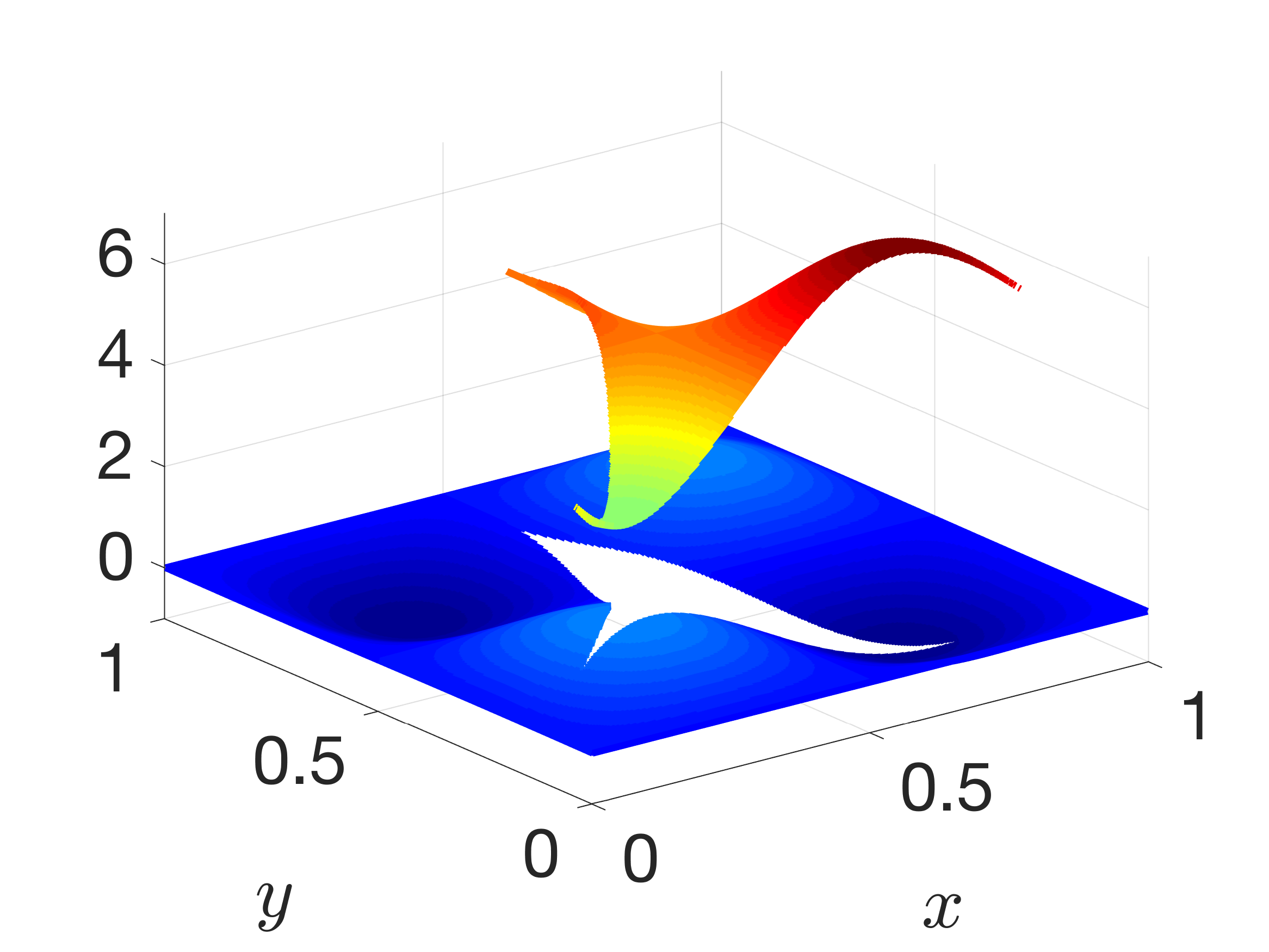}
	}
\subfigure[$H_y$]{
	\includegraphics[width=2.2in]{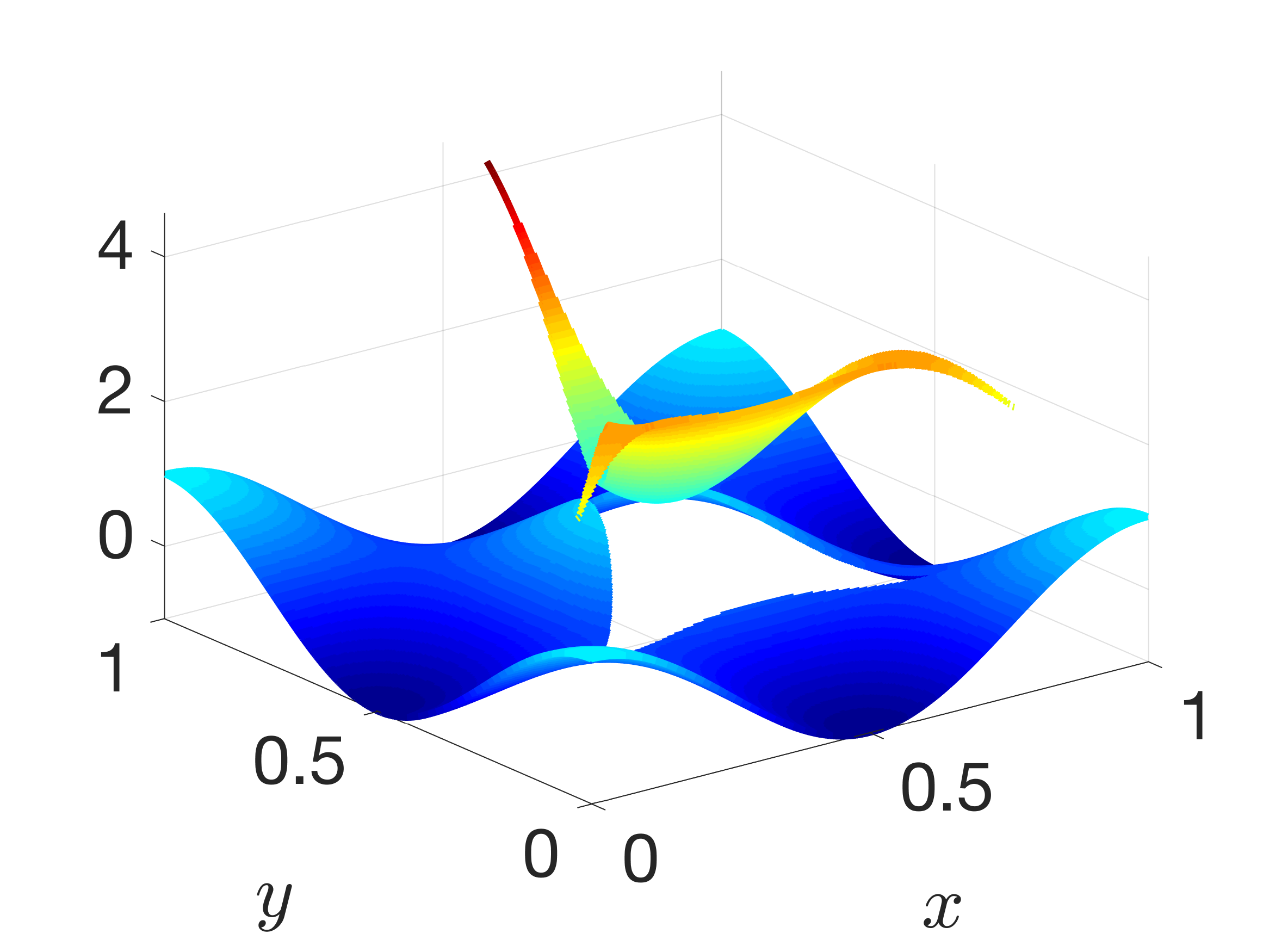}
	}
	\subfigure[$E_z$]{
	\includegraphics[width=2.2in]{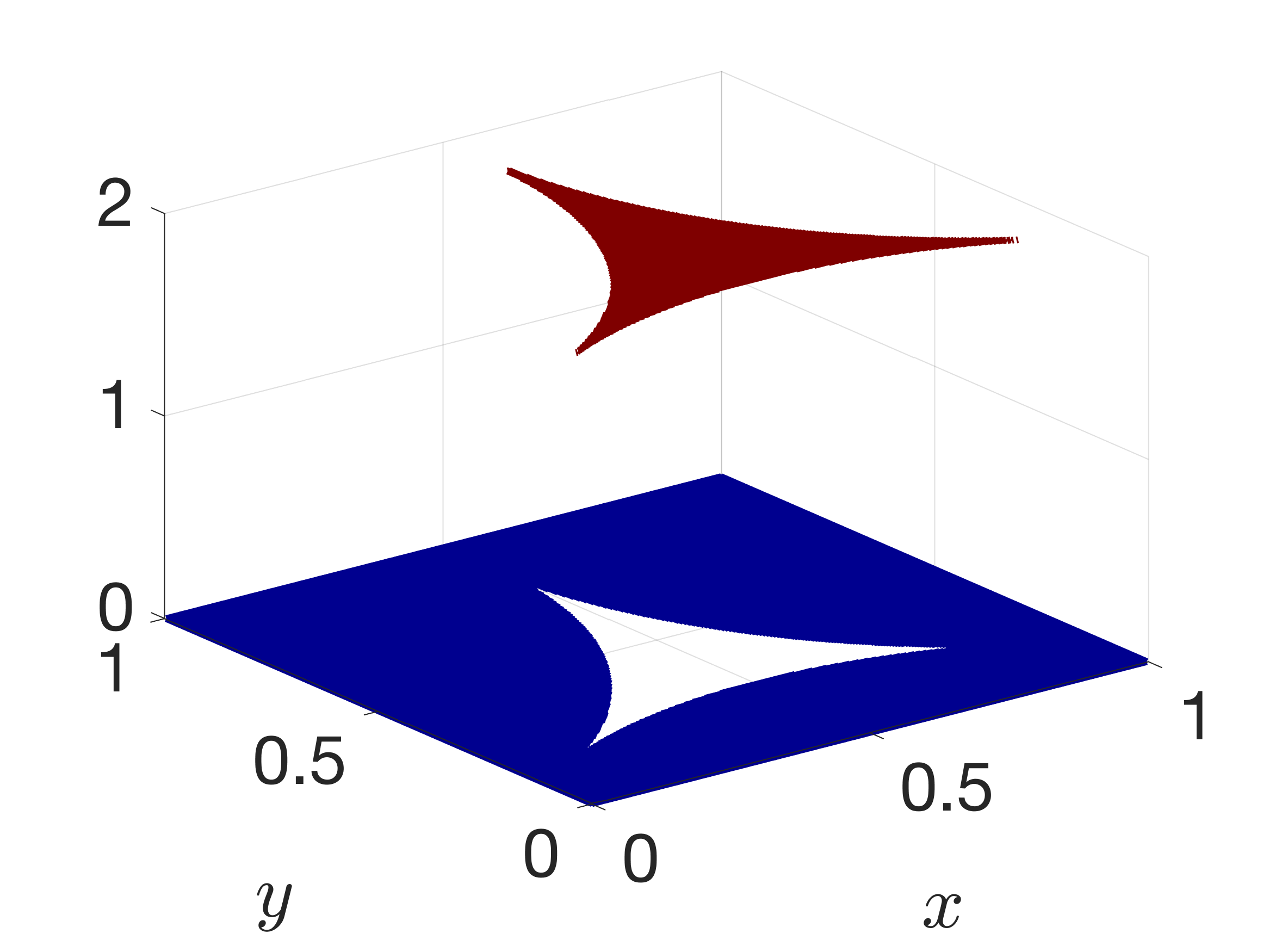}
	}
       \caption{The components $H_x$, $H_y$ and $E_z$ at $t=0.25$ with $h = \tfrac{1}{336}$ and $\Delta t = \tfrac{h}{2}$ using the fourth-order staggered FD scheme with the CFM for the problem with a manufactured solution and a non-smooth interface.}
       \label{fig:plotSharpInterfaceFields}
\end{figure}
\FloatBarrier	
\section{Conclusions}

This work uses the correction function method to develop high-order finite-difference time-domain schemes to handle Maxwell's equations with complex interface conditions and continuous coefficients.
The system of PDEs for which the solution corresponds to correction functions is derived from Maxwell's equations with interface conditions.
We have shown that this system of PDEs does not allow a perturbation on the solution to growth.
A functional that is a square measure of the error associated 
	with the correction functions' system of PDEs is minimized 
	to allow us to compute approximations of correction functions where it is needed.
A discrete divergence-free polynomial space in which the functional is minimized is chosen to satisfy the divergence constraints.
Approximations of correction functions are then used to correct either the second-order or fourth-order staggered FD scheme. 
We use a staggered grid in space to enforce discrete divergence constraints and the fourth-order Runge-Kutta time-stepping method. 
The discrete divergence constraint and the consistency of resulting schemes have been studied.
We have shown that an approximation of the magnetic field remains at divergence-free for a discrete measure of the divergence, 
	except for some nodes around the interface.
Moreover, 
	the leading error term associated with resulting schemes can be influenced 
	by the order of approximations of correction functions.
Numerical experiments have been performed in 2-D using different geometries of the interface. 
All convergence studies are in agreement with the theory. 
In all our numerical experiments,
	the discontinuities within solutions are accurately captured without spurious oscillations.
The proposed numerical strategy is a promising candidate to handle Maxwell's equations with interface conditions 
	without increasing its complexity for arbitrary geometries of the interface while keeping high-order accuracy.
Future work will include discontinuous coefficients to handle more realistic materials, 
	such as dielectrics, 
	and an extension of the proposed numerical strategy in 3-D.

\section*{Acknowledgments}
The authors are grateful to Professor Charles Audet for interesting and helpful conversations.
The research of Professor Jean-Christophe Nave was partially supported by the NSERC Discovery Program.
This is a pre-print of an article published in Journal of Scientific Computing. The final authenticated version is available online at: https://doi.org/10.1007/s10915-020-01148-6.

\bibliographystyle{siamplain}
\bibliography{references}
\end{document}